\documentclass{amsart}
\usepackage{amsmath,amsthm,amsfonts,amssymb,mathtools,mathrsfs,url}
  \usepackage{paralist}
  \usepackage{graphics}
    \usepackage{epsfig}
\usepackage{graphicx}
\usepackage[colorlinks=true]{hyperref}
\usepackage{graphicx}
\usepackage{amsmath}
\usepackage{amssymb}
\usepackage[utf8]{inputenc}
\usepackage{subcaption}
\usepackage{enumitem}
\usepackage{url}
\usepackage{xcolor}
\usepackage[toc,page]{appendix}
\usepackage{stmaryrd}
\usepackage{mathrsfs}
\usepackage{mathtools}
\usepackage[abbrev]{amsrefs}
\usepackage{tikz}
\usetikzlibrary{patterns.meta}

\numberwithin{equation}{section}

\hypersetup{urlcolor=blue, citecolor=green}

\mathtoolsset{showonlyrefs=true}

\calclayout

\newcommand{\R}{\mathbb R} %% permet de faire plus simplement un R avec double barre
\newcommand{\N}{\mathbb N}

\newcommand{\C}{\mathbb C}

\newcommand{\Pro}{\mathbb P}

\newcommand{\1}{\mathbf 1}
\newcommand{\D}{\mathrm{d}}
\newcommand{\om}{\omega}

\newcommand{\e}{\mathrm{e}}

\newcommand{\hop}{\vskip.3cm\noindent} % pour faire un saut sans "indenter"
\newcommand{\hip}{\vskip.1cm\noindent}

\newtheorem{thm}{Theorem}[section]   % enlever le % devant [section] pour avoir une autre numerotation...
\newtheorem{cor}[thm]{Corollary}%[section]
\newtheorem{prop}[thm]{Proposition}%[section]
\newtheorem{lem}[thm]{Lemma}%[section]
\newtheorem{defi}[thm]{Definition}%[section]
\newtheorem{rema}[thm]{Remark}%[section]
\newtheorem{hypo}[thm]{Hypothesis}%[section]

\DeclareRobustCommand{\SkipTocEntry}[5]{}

\title[Metastability results for a class of linear Boltzmann equations]{Metastability results for a class of linear Boltzmann equations} 
\author[T. Normand]{Thomas Normand}
\email{thomas.normand@math.u-bordeaux.fr}
\date{}

\begin{document}
\maketitle

\begin{abstract}
We consider a semiclassical linear Boltzmann model with a non local collision operator.
We provide sharp spectral asymptotics for the small spectrum in the low temperature regime from which we deduce the rate of return to equilibrium as well as a metastability result.
The main ingredients are resolvent estimates obtained via hypocoercive techniques and the construction of sharp Gaussian quasimodes through an adaptation of the WKB method.
\end{abstract}

\tableofcontents

\section{Introduction}
\subsection{Motivations}
We are interested in the linear Boltzmann equation:
\begin{equation}\label{Boltl1}
\left\{
\begin{aligned}
&h\partial_tu+v\cdot h\partial_xu-\partial_xV\cdot h\partial_vu+Q_{\mathcal H}(h,u)=0 \\ 
&u_{|t=0}=u_0
\end{aligned}
\right.
\end{equation}
in a semiclassical framework (i.e in the limit $h \to 0$), where $h$ is a \emph{semiclassical parameter} and corresponds to the temperature of the system.
Here we denoted for shortness $\partial_x$ and $\partial_v$ the partial gradients with respect to $x$ and $v$.
This equation is used to model the evolution of a system of charged particles in a gas on which acts an electrical force associated to the real valued potential $V$ that only depends on the space variable $x$. %\in \mathcal C^\infty(\R_x^d,\R)$.
%The operator $Q_{\mathcal H}$ is called \emph{collision operator} and models 
The interactions between the particles are modelled by the linear operator $Q_{\mathcal H}$ which is called \emph{collision operator}.
%It contains an integral transform term whose distributional kernel is then called \emph{collision kernel}.
Here the unknown is the function $u:\R_+\to L^1(\R^{2d})$ giving the probability density of the system of particles at time $t \in \R_+$, position $x \in \R^d$ and velocity $v\in \R^d$.
For our purpose, we introduce the square roots of the usual Maxwellian distributions 
\begin{align}\label{muh}
\mu_h(v)=\frac{\e^{-\frac{v^2}{4h}}}{(2\pi h)^{d/4}} \qquad \text{and } \qquad \mathcal M_h=\e^{-\frac{V}{2h}}\mu_h.
\end{align}
In many models, we have 
\begin{align}\label{ql1}
Q_{\mathcal H}(h,\mathcal M_h^2)=0 \qquad \text{and } \qquad Q_{\mathcal H}^*(h,1)=0
\end{align}
so in particular $\mathcal M_h^2$ is a stable state of \eqref{Boltl1}.
In order to do a perturbative study of the time independent operator near $\mathcal M^2_h$, we introduce the natural Hilbert space %for the study of the time independent operator is then (see for instance \cite{Herau} or \cite{Robbe})
$$\mathcal H= \big\{u \in \mathcal D' \,;\, \mathcal M_h^{-1} u\in L^2(\R^{2d})\big\}.$$
It is clear from the Cauchy Schwarz inequality that $\mathcal H$ is indeed a subset of $L^1(\R^{2d})$ provided that $\e^{-\frac{V}{2h}}\in L^2(\R^d_x)$.
In view of \eqref{ql1} and the definition of $\mathcal H$, it is more convenient to work with the new unknown
$$f=\mathcal M_h^{-1} u\, :\R_+\to L^2(\R^{2d})$$
for which the new equation becomes 
\begin{equation}\label{Boltgene}
\left\{
\begin{aligned}
&h\partial_tf+v\cdot h\partial_xf-\partial_xV\cdot h\partial_vf+Q_h(f)=0 \\ 
&f_{|t=0}=f_0
\end{aligned}
\right.
\end{equation}
where
$$Q_h=\mathcal M_h^{-1}\circ Q_{\mathcal H}(h, \cdot) \circ \mathcal M_h.$$
Our study will be focused on the new time independent operator 
\begin{align}\label{Ph}
P_h&=v\cdot h\partial_x-\partial_xV\cdot h\partial_v+Q_h \nonumber \\
&=X_0^h+Q_h 
\end{align}
for some specific choices of the collision operator $Q_h$, where the notation $X_0^h$ will stand for the operator $v\cdot h\partial_x-\partial_xV\cdot h\partial_v$, but also for the vector field $(x,v)\mapsto h(v,-\partial_x V(x))$.
There are plenty of different collision operators studied in the literature, their main properties being that these are symmetric integral operators acting as multiplicators in the position variable $x$ and canceling the Maxwellian distribution.
Our work is in particular motivated by the study of the \emph{mild relaxation} operator introduced in \cite{theserobbe} and given by $H_0(1+H_0)^{-1}$ with $H_0$ the harmonic oscillator in velocity defined by
\begin{align}\label{h0}
H_0=-h^2\Delta_v+\frac{v^2}{4}-\frac{hd}{2}.
\end{align}
In this spirit, the collision operators we will be working with will always be bounded and self-adjoint so, $(X_0^h,\mathcal C^\infty_c(\R^{2d}))$ being essentially skew-adjoint, the operator $P_h$ (endowed with the appropriate domain) is maximal accretive and \eqref{Boltgene} is well-posed.
More generally, some interesting cases of collision operators are given by functions of $H_0$ (see for instance \cite{Lemoprxu11,Lemoprxu12,Lemoprxu13,Herau,theserobbe}) which is the setting that we will adopt.
 
%Our goal will be to give precise spectral asymptotics for the operator $P_h$.
This paper is concerned with the spectral study of the operator $P_h$.
This type of questions has recently known some major progress on the impulse of microlocal methods.
%In the case of the Boltzmann equation \eqref{Boltgene}, a first step was achieved in \cite{Robbe} where the author established a rough localization of the small spectrum of $P_h$ similar to the one obtained for example for the Witten Laplacian in \cite{HelSjo}. 
%This was done using hypocoercive techniques to get some resolvent estimates.
In the case of the linear Boltzmann equation \eqref{Boltgene}, the use of hypocoercive techniques in 2015 in \cite{Robbe} enabled to get some resolvent estimates and establish a rough localization of the small spectrum of $P_h$ which consists of exponentially small eigenvalues in correspondance with the minima of the potential $V$. 
This type of result is similar to the one obtained for example for the Witten Laplacian by Helffer and Sj\"ostrand in \cite{HelSjo} in the 1980's.
Such a localization already leads to return to equilibrium and metastability results which can be improved as the description of the small spectrum becomes more precise.
For example, sharp asymptotics of the small eigenvalues of the Witten Laplacian were obtained later in the 2000's in \cite{BoGaKl} and \cite{HeKlNi} and later again for Kramers-Fokker-Planck type operators by H\'erau et al. in \cite{HHS}.
In these papers, the idea was to exhibit a supersymmetric structure for the operator and then study both the derivative acting from 0-forms into 1-forms and its adjoint with the help of basic quasimodes.
In \cite{theserobbe}, Robbe managed to show that the Boltzmann equation \eqref{Boltgene} with mild relaxation enjoys such a supersymmetric structure.
However, in that case, the matrix appearing in the modification of the inner product does not obey good estimates with respect to the semiclassical parameter $h$.
This is why our goal here will be to give precise spectral asymptotics for the operator $P_h$ through a more recent approach which consists in directly constructing a family of accurate quasimodes for our operator in the spirit of \cite{LPMichel} and \cite{BonyLPMichel}.

The aim of this paper is twofold.
In a first time, we want to prove a result similar to the one obtained by Robbe in \cite{Robbe} but for a large class of collision operators.
% for a class of collision kernels associated to pseudo-differential operators satisfying nice symbol properties.
The second goal is to provide complete asymptotics of the small eigenvalues of $P_h$ as it was done in \cite{HeKlNi} for the Witten Laplacian or in \cite{HHS, HHS11} with recent improvements by Bony et al. in \cite{BonyLPMichel} in the case of Fokker-Planck type differential operators.
We manage to establish such results for the equation \eqref{Boltgene} for a class of pseudo-differential collision operators presenting nice symbol properties as well as a factorized structure.

\subsection{Setting and main results}
For $d'\in \N^*$ and $Z\in \C^{d'} $, we use the standard notation $\langle Z \rangle=(1+|Z|^2)^{1/2}$.
In this paper, we will treat the case of collision operators of the form
$$Q_h=\varrho(H_0)$$
with $\varrho$ satisfying the following:
\begin{hypo}\label{hyporho}
The function $\varrho: \R_+ \to \R_+$ vanishes at the origin and for all $t\geq 0$,
$$\varrho(t)\geq \frac1C \frac{t}{\langle t \rangle}.$$
Moreover, it admits an analytic extension to $\{\mathrm{Re}\, z > -\frac1C\}$ for which there exist $\varrho_\infty \in \R_+$ and $\alpha>0$ such that $\varrho(z)=\varrho_\infty+O(\langle z \rangle^{-\alpha})$.
\end{hypo}
\hip
In particular, $Q_h$ will be bounded uniformly in $h$ and self-adjoint.
An example of such collision operator is the \emph{mild relaxation} operator introduced in \cite{theserobbe} and given by $H_0(1+H_0)^{-1}$.
In order to state the consequences of Hypothesis \ref{hyporho}, let us introduce a few notations of semiclassical microlocal analysis which will be used in all this paper.
These are mainly extracted from \cite{Zworski}, chapter 4.
%\begin{defi}\label{symbols}
We will denote $\Xi\in \R^{d'}$ %(resp. $\eta$) 
the dual variable of $X$ %(resp. $v$) 
and use the semiclassical Fourier transform
$$\mathcal F_h(f)(\Xi)= \int_{\R^{d'}}\e^{-\frac ih X\cdot \Xi}f(X) \, \D X.$$
We consider the space of semiclassical symbols 
\begin{align*}
S^\kappa \big(\langle (X,\Xi) \rangle^k\big)=\big\{a_h \in \mathcal C^\infty(\R^{2d'}) \, ; \, \forall \alpha \in \N^{2d'}, \exists \, C_\alpha>0\;
		\emph{\text{such that }} |\partial^\alpha a_h(X,\Xi)|\leq C_\alpha h^{-\kappa |\alpha|}\langle (X,\Xi) \rangle^k\big\}
\end{align*}
where $k\in \R$ and $\kappa \in [0, 1/2]$. Note that those symbols are allowed to depend on $h$; however, in order to shorten the notations, we will drop the index $h$ in the rest of the paper when dealing with semiclassical symbols.
%\end{defi}
%\begin{defi}\label{oph}
Given a symbol $a \in S^\kappa(\langle (X,\Xi) \rangle^k)$, we define the associated semiclassical pseudo-differential operator for the Weyl quantization acting on functions $u \in \mathcal S(\R^{d'})$ by
$$\mathrm{Op}_h(a)u(X)=(2\pi h)^{-d'}\int_{\R^{d'}} \int_{\R^{d'}}\e^{\frac ih (X-X')\cdot \Xi}a\Big(\frac{X+X'}{2},\Xi\Big)u(X')\,\D X'\D \Xi$$
where the integrals may have to be interpreted as oscillating integrals.
We will denote $\Psi^\kappa(\langle (X,\Xi) \rangle^k)$ the set of such operators.
%Note that 
%the operator $\mathrm{Op}_h(a)$ admits the distributional kernel
%$$K_h(X,X')=\mathcal F_h^{-1}\bigg(a\bigg(\frac{X+X'}{2},\cdot\bigg)\bigg)(X-X').$$
%Conversely, if an operator $\mathrm{Op}_h(a)\in \Psi^\kappa(\langle (X,\Xi) \rangle^k)$ admits the distributional kernel $K_h(X,X')$, then its symbol is given by
%\begin{align}\label{FK}
%a(X,\Xi)=\mathcal F_h\Big(\big(K_h\circ A\big)(X,\cdot)\Big)(\Xi)
%\end{align}
%where $A$ denotes the change of variables 
%\begin{align}\label{mata}
%A(X,X')=(X+X'/2,X-X'/2).
%\end{align}
In our setting, we will denote $\xi$ (resp. $\eta$) the dual variable of $x$ (resp. $v$).
%and $A_v$ the change of variables 
%\begin{align}\label{Av}
%A_v:(x,v,v')\mapsto (x,v+v'/2,v-v'/2).
%\end{align}
We also need to introduce the notion of analytic symbols.
For our purpose, we almost always consider symbols that do not depend on the variable $\xi$.
\begin{defi}\label{symbolo}
For $\tau>0$, let us introduce the set 
$$\Sigma_\tau=\{z\in \C \, ; \, |\mathrm{Im}\,z|< \tau\}^d\subset \C^d.$$
For $k \in \R$, we denote $S^0_\tau(\langle (x,v,\eta) \rangle^k)$ the space of symbols $a_h\in S^0(\langle (x,v,\eta) \rangle^k)$ independent of $\xi$ such that:
\begin{enumerate}[label=$\mathrm{(\roman*)}$]
\item For all %$\beta \in \N^{2d}$ and 
$(x,v)\in \R^{2d}$, %$\partial^\beta_{(x,v)} 
$a_h(x,v,\cdot)$ is analytic on $\Sigma_\tau$ \label{holom}
\item For all $\beta \in \N^{2d}$, there exists $C_\beta>0$ such that $|\partial_{(x,v)}^{\beta} a_h |\leq C_\beta \langle (x,v,\eta) \rangle^k$%\qquad \text{on } 
on $\R^{2d}\times \Sigma_\tau.$ \label{majobande}
\end{enumerate}
We will also use the notation $a_h=O_{S^0_\tau(\langle(x,v,\eta) \rangle^{k})}(h^N)$ to say that for all $\alpha \in \N^{3d}$, there exists $C_{\alpha, N}$ such that
$| \partial^\alpha  a_h |\leq C_{\alpha, N} \, h^N \langle (x,v,\eta) \rangle^k$
on $\R^{2d}\times \Sigma_\tau$.
\end{defi}
\hip
Here again, we will drop the index $h$ in the notations of analytic symbols.
Using the Cauchy-Riemann equations, we see that item \ref{holom} from Definition \ref{symbolo} implies that for all $\beta \in \N^{2d}$ and $(x,v)\in \R^{2d}$, the functions $\partial^\beta_{(x,v)} a(x,v,\cdot)$ are also analytic on $\Sigma_\tau$.
Besides, the Cauchy formula implies that for any $\tilde \tau<\tau$, $\alpha \in \N^d$ and $\beta\in \N^{2d}$, there exists $C_{\alpha,\beta}$ such that
$$|\partial_\eta^\alpha\partial_{(x,v)}^\beta a |\leq C_{\alpha,\beta} \langle (x,v,\eta) \rangle^k \qquad \text{ on } \R^{2d}\times \Sigma_{\tilde \tau}$$
i.e up to taking $\tau$ smaller, item \ref{majobande} from Definition \ref{symbolo} can be extended to $\beta\in \N^{3d}$.
Let us introduce a slightly unusual notion of "expansion" where the coefficients are allowed to depend on $h$: we will say that 
\begin{align}\label{exph}
a \sim_h \sum_{j \geq 0} h^j a_j
\end{align}
in $S^0(\langle (x,v,\eta) \rangle^k)$ (resp. in $S^0_\tau(\langle (x,v,\eta) \rangle^k)$) if $(a_j)_{j \geq 0}\subset S^0(\langle (x,v,\eta) \rangle^k)$ (resp. $(a_j)_{j \geq 0}\subset S^0_\tau(\langle (x,v,\eta) \rangle^k)$) is a family of symbols which may depend on $h$ and are such that for all $N\in \N$, 
$$ a-\sum_{j=0}^{N-1}h^j a_j =O_{S^0(\langle(x,v,\eta)\rangle^{k})}(h^N) \qquad \text{\big(resp. } O_{S^0_\tau(\langle(x,v,\eta)\rangle^{k})}(h^N)\text{\big)}$$
Finally, we also have the usual notion of classical expansion for a symbol: $a \sim \sum_{j \geq 0} h^j a_j$ in $S^0(\langle (x,v,\eta) \rangle^k)$ (resp. in $S^0_\tau(\langle (x,v,\eta) \rangle^k)$) means that $a \sim_h \sum_{j \geq 0} h^j a_j$ in $S^0(\langle (x,v,\eta) \rangle^k)$ (resp. in $S^0_\tau(\langle (x,v,\eta) \rangle^k)$) and the $(a_j)_{j \geq 0}$ are independent of $h$.\\
We now extend these notions to matrix valued symbols: if 
$$M=(m_{p,q})\mathop{}_{\substack{1\leq p \leq n_1 \\  1\leq q \leq n_2}}$$
\sloppy is a matrix of functions such that each $m_{p,q}\in S^\kappa(\langle (x,v,\eta) \rangle^k)$ (resp. $m_{p,q}\in S^0_\tau(\langle (x,v,\eta) \rangle^k)$), we say that $M\in \mathcal M_{n_1,n_2}\big(S^\kappa(\langle (x,v,\eta) \rangle^k)\big)$ \big(resp. $M\in \mathcal M_{n_1,n_2}\big(S^0_\tau(\langle (x,v,\eta) \rangle^k)\big)$\big) and we denote 
$$\mathrm{Op}_h(M)=\Big(\mathrm{Op}_h(m_{p,q})\Big)\mathop{}_{\substack{1\leq p \leq n_1 \\  1\leq q \leq n_2}}.$$
The notation 
$$M=O_{\mathcal M_{n_1,n_2}\big(S^0(\langle(x,v,\eta) \rangle^{k})\big)}(h^N) \qquad \text{\big(resp. } M=O_{\mathcal M_{n_1,n_2}\big(S^0_\tau(\langle(x,v,\eta) \rangle^{k})\big)}(h^N)\text{\big)}$$
means that for all $(p,q)\in \llbracket 1,n_1 \rrbracket\times \llbracket 1,n_2 \rrbracket$, the symbol $m_{p,q}$ is $O_{S^0(\langle(x,v,\eta) \rangle^{k})}(h^N)$ (resp. $O_{S^0_\tau(\langle(x,v,\eta) \rangle^{k})}(h^N)$).
Furthermore, the notions of expansions $M\sim_h \sum_{n \geq 0} h^n M_n$ and $M\sim \sum_{n \geq 0} h^n M_n$ in $\mathcal M_{n_1,n_2}\big(S^0(\langle (x,v,\eta) \rangle^k)\big)$ \big(resp. $\mathcal M_{n_1,n_2}\big(S^0_\tau(\langle (x,v,\eta) \rangle^k)\big)$\big) are straightforward adaptations of the ones for scalar symbols.

These notions enable us to introduce a new class of collision operators which appears to be more general that the one given by Hypothesis \ref{hyporho}.
Let us denote $b_h$ the twisted derivative
\begin{align}\label{bh}
b_h=h\partial_v +v/2
\end{align}
so that in particular with the notation \eqref{h0} we have $H_0=b_h^*b_h$.
\begin{hypo}\label{hypom}
There exists $\tau>0$ and a symmetric matrix of analytic symbols
$$M^h(x,v,\eta)=\big(m_{p,q}(x,v,\eta)\big)_{1\leq p,q \leq d}\in \mathcal M_d\big(S^0_\tau(\langle (v,\eta) \rangle^{-2})\big)$$
sending $\R^{3d}$ into $\mathcal M_d(\R)$ and such that, with the notation \eqref{bh}, the collision operator $Q_h$ satisfies
%\color{red} Attention peut etre $m_0$ au lieu de $m^h$ pour certaines, dépend si les coeff des expansions dépendent encore de $h$.
%\color{black}
\begin{enumerate}[label=\alph*)]
	\item %There exists $M^h=(m_{p,q})_{1\leq p,q \leq d}\in \mathcal C^\infty(\R_x\times \R_v\times \R_\eta\,;\, \mathcal M_d(\R) )$ a matrix of symbols depending on $h$ such that 
$Q_h=b_h^*\circ\text{Op}_h(M^h)\circ b_h$ \label{facto}
%	\item $M^h$ does not depend on $\xi$
%	\item For all $(x,v,\eta)\in \R^{3d}$, $M^h(x,v,\eta)\in \mathcal M_d(\R)$ 
	%\item For all $(p,q)\in \llbracket 1,d \rrbracket^2$, $m_{p,q}=m_{p,q}$ so $M^h$ is symmetric \label{sym}
	%\item For all $(p,q)\in \llbracket 1,d \rrbracket^2$, $m_{p,q} \in S^0_\tau(\langle (v,\eta) \rangle^{-2})$ \label{stau}%S^0_\tau(\langle (v_j,v_k,\eta_j,\eta _k) \rangle^{-2}) $ %\color{red}(à mettre après déf de $\tau$ et $\Sigma_\tau$ et déf $\ref{symbolo}$).\color{black}
	\item %For all $(p,q)\in \llbracket 1,d \rrbracket^2$, $m_{p,q} \sim \sum_{n \geq 0}h^n m_{p,q}^n$ in $S^0_\tau(\langle (v,\eta) \rangle^{-2}) $ so 
$M^h\sim \sum_{n \geq 0}h^nM_n$ in $\mathcal M_d\big(S^0_\tau(\langle (v,\eta) \rangle^{-2})\big)$ \label{expm}%\color{red}(à mettre après déf $\ref{symbolo}$ et expliquer le dev asymptotique dans ce cas).\color{black}
	\item For all $(x,v,\eta)\in \R^{3d}$, $M^h(x, v, \eta ) =M^h(x, v, -\eta )$ \label{paire}
	\item For all $(x,v,\eta)\in \R^{3d}$, $M_0(x,v,\eta) \geq \frac 1C \langle (v,\eta) \rangle^{-2} \; \mathrm{Id}$.\label{minom}
\end{enumerate}
\end{hypo}
\hip
Since the $(M_n)_n$ do not depend on $h$, we easily get that these matrices of symbols are also even in $\eta$, symmetric, independent of $\xi$ and with values in $\mathcal M_d(\R)$; so in particular item \ref{minom} makes sense.
This will enable us to establish Lemma \ref{mino Q} which is sometimes reffered to as \emph{microscopic coercivity} (see for instance \cite{DoMoSc}).
As announced, we have the following Lemma which is proven in Appendix \ref{rho}:
\begin{lem}\label{1impl2}
Hypothesis \ref{hyporho} implies Hypothesis \ref{hypom}.
\end{lem}
\hip
We will also make a few confining assumptions on the function $V$, assuring for instance that the bottom spectrum  of the associated Witten Laplacian is discrete. In particular, our potential will satisfy Assumption 2 from \cite{LPMichel} and Hypothesis 1.1 from \cite{Robbe}.
\begin{hypo}\label{V}
The potential $V$ is a smooth Morse function depending only on the space variable $x\in \R^d$ with values in $\R$ which is bounded from below and such that
$$|\partial_x V(x)|\geq \frac1C\qquad \text{for }|x|>C.$$
Moreover, for all $\alpha \in \N^d$ with $|\alpha|\geq 2$, there exists $C_\alpha$ such that 
$$|\partial_x^\alpha V|\leq C_\alpha.$$
In particular, for every $0\leq k\leq d$, the set of critical points of index $k$ of $V$ that we denote $\mathcal U^{(k)}$ is finite and we set 
\begin{align}\label{n0}
n_0=\# \mathcal U^{(0)}.
\end{align}
Finally, we will suppose that $n_0\geq 2$.
\end{hypo}
\hip
The last assumption comes from the fact that when $n_0=1$, the so-called \emph{small spectrum} of the operator $P_h$ (i.e its eigenvalues with exponentially small modulus) is trivial, so there is nothing to study.
It is shown in \cite{MeSc}, Lemma 3.14 that for a function $V$ satisfying Hypothesis \ref{V}, we have $V(x)\geq |x|/C$ outside of a compact. In particular, under Hypothesis \ref{V}, it holds $\e^{-V/2h}\in L^2(\R^{d}_x)$.
Moreover, in our setting, $X_0^h$ is a smooth vector field whose differential is bounded on $\R^{2d}$, so the operator $X_0^h$ endowed with the domain
$$D=\{u \in L^2(\R^{2d}) \, ; \, X_0^hu \in L^2(\R^{2d})\}$$
is skew-adjoint on $L^2(\R^{2d})$ and the set $\mathcal S(\R^{2d})$ is a core for this operator.
Therefore, $(P_h,D)^*=(-X_0^h+Q_h,D)$ and $(P_h,D)$ is m-accretive on $L^2(\R^{2d})$.

We can now state our first result which consists in giving a rough localization of the small spectrum of $P_h$ that we prove in Section \ref{sectionrough}.
\begin{thm}\label{thmRobbe}
Assume that Hypotheses \ref{hypom} %(or equivalently \eqref{pseudo} and \eqref{minoq}) 
and \ref{V} are satisfied and recall the notation \eqref{n0}.
Then the operator $(P_h,D)$ admits $0$ as a simple eigenvalue.
Moreover, there exists $c>0$ and $h_0>0$ such that for all $0<h\leq h_0$,  
$\mathrm{Spec}(P_h)\cap \{\mathrm{Re} \,z\leq ch^2\}$ consists of exactly $n_0$ eigenvalues (counted with algebraic multiplicity) %that are in $\{\mathrm{Re}$ $ z>0\}$ and 
that are exponentially small with respect to $1/h$ and for all $0<\tilde c\leq c$, the resolvent estimate 
$$(P_h-z)^{-1}=O(h^{-2})$$
holds uniformly in $\{\mathrm{Re} \,z\leq ch^2\} \backslash B(0, \tilde c h^2)$.
Finally, except for $0$, the real parts of these small eigenvalues are positive.
\end{thm}
\hip
This result can be seen as a generalization of Theorem 3.0.2 from \cite{theserobbe} (up to the $h^2$ instead of $h$) as we saw that the \emph{mild relaxation} operator (which is the collision operator studied in this reference) satisfies our hypotheses.
In our case we get a localization of order $h^2$ because we adopt a simpler proof based on hypocoercivity (inspired by \cite{Robbe}) than the one presented in \cite{theserobbe}.

In order to study the long time behavior of the solutions of \eqref{Boltgene}, we need a precise description of the small spectrum of $P_h$.
To this aim, we construct in Sections \ref{sectionquasim} and \ref{sectionequations} in the spirit of the WKB method a family of accurate quasimodes localized around the minima of $V$ that enables us to establish sharp asymptotics of the small eigenvalues of $P_h$.
This leads in Section \ref{sectionvp} to the following Theorem which is the main result of this paper.
For the sake of simplicity, we make in the statement an additionnal assumption (Hypothesis \ref{jvide}) on the topology of the potential $V$ that could actually be omitted (see \cite{Michel} or \cite{BonyLPMichel}). 
\begin{thm}\label{thmToto}
Suppose that %the assumptions of Theorem \ref{thmRobbe} are satisfied, as well as 
Hypotheses \ref{hypom}, \ref{V} and \ref{jvide} are satisfied and denote $\underline{\mathbf m}$ a global minimum of $V$.
According to Theorem \ref{thmRobbe}, we can associate to each $\mathbf m \in \mathcal U^{(0)}\backslash \{\underline{\mathbf m}\}$ a non zero exponentially small eigenvalue of $P_h$ that we denote $\lambda(\mathbf m,h)$.
These eigenvalues satisfy the following formula:
$$\lambda(\mathbf m,h)=h\e^{-2\frac{S(\mathbf m)}{h}}\frac{\det (\mathrm{Hess}_{\mathbf m}V)^{1/2}}{2\pi} B_h(\mathbf m) $$
with $S$ defined in Definition \ref{j et s} and $B_h(\mathbf m)$ admitting a classical expansion whose first term is 
$$\sum_{\mathbf s \in \mathbf j(\mathbf m)} |\det (\mathrm{Hess}_{\mathbf s}V)|^{-1/2} \,M_0(\mathbf s,0,0) \nu_2^{\mathbf s} \cdot \nu_2^{\mathbf s}$$
where the map $\mathbf j$ is also defined in Definition \ref{j et s}, the matrix $M_0(\mathbf s,0,0)$ is introduced in Hypothesis \ref{hypom} and the vector $\nu_2^{\mathbf s}\in \R^d$ is defined in Proposition \ref{phinu}.
\end{thm}

When Hypothesis \ref{hypom} is replaced by Hypothesis \ref{hyporho}, we can give a slightly more precise statement. In that case, denoting $\boldsymbol \mu_{\mathbf s}$ the only negative eigenvalue of $\mathrm{Hess}_{\mathbf s}V$, the first term of $B_h(\mathbf m)$ is 
$$\frac12\sum_{\mathbf s \in \mathbf j(\mathbf m)} |\det (\mathrm{Hess}_{\mathbf s}V)|^{-1/2}\Big(-\varrho'(0)+\sqrt{\varrho'(0)^2-4\boldsymbol \mu_ {\mathbf s}}\Big).$$
Indeed, under Hypothesis \ref{hyporho}, it is shown in Appendix \ref{rho}, more precisely in \eqref{m0} that $M_0(\mathbf s,0,0)=\tilde \varrho(0)\,\mathrm{Id}=\varrho'(0)\,\mathrm{Id}$.
Thanks to Proposition \ref{phinu}, we then have
$$\mathrm{Hess}_{\mathbf s}V \nu_2=-\varrho'(0)^2(1+\nu_2^2)\nu_2^2\, \nu_2$$
and consequently
$$\nu_2^2=-\frac12 +\frac{\sqrt{\varrho'(0)^2-4\boldsymbol \mu_ {\mathbf s}}}{2\varrho'(0)}$$
so the statement follows.

Finally, Section \ref{sectionral} consists in using the sharp localization obtained in Theorem \ref{thmToto} in order to discuss the phenomena of return to equilibrium and metastability for the solutions of \eqref{Boltgene}.
More precisely, we are able to give a sharp rate of convergence of the semigroup $\e^{-tP_h/h}$ towards $\Pro_1$, the orthogonal projector on Ker $P_h$ : denoting $\lambda^*$ a non zero eigenvalue of $P_h$ whose real part is minimal, we establish that the rate of return to equilibrium is essentially given by $\mathrm{Re}\,\lambda^*/h$:
\begin{cor}\label{ral}
Under the assumptions of Theorem \ref{thmToto}, there exists $h_0>0$ such that for all $0<h\leq h_0$, $t\geq 0$ and $N\geq 1$, there exists $C_N>0$ such that 
$$\|\e^{-tP_h/h}-\Pro_1\|\leq C_N \e^{-t \, \mathrm{Re}\, \lambda^*(1-C_N h^N)/h}.$$
Moreover, if $\lambda^*$ does not share its expansion given by Theoerm \ref{thmToto} with another eigenvalue of $P_h$ (in particular it is a simple eigenvalue), then $\lambda^*$ is real and we even have
$$\|\e^{-tP_h/h}-\Pro_1\|\leq C \e^{-t \lambda^*/h}.$$
\end{cor}
\hip
Besides, in the spirit of \cite{BonyLPMichel}, we also show the metastable behavior of the solutions of \eqref{Boltgene}:
\begin{cor}\label{meta}
Suppose that the assumptions of Theorem \ref{thmToto} hold true.
Let us consider some local minima $\mathbf m_1=\underline{\mathbf m}$, $\mathbf m_2$, $\dots$, $\mathbf m_K$ such that 
$$S\big(\mathcal U^{(0)}\big)=\{+\infty=S(\mathbf m_1) > S(\mathbf m_2) > \dots > S(\mathbf m_K)\}$$
for the map $S$ from Definition \ref{j et s}.
% using Lemma \ref{Pff} and the notation \eqref{lamtilde}, the sequence of approximated eigenvalues $\tilde \lambda_{\mathbf m_1}=0$, $\dots$, $\tilde \lambda_{\mathbf m_K}$ is increasing and each small eigenvalue of $P_h$ shares the classical expansion of one of the $(\tilde \lambda_{\mathbf m_k})_{1\leq k \leq K}$.
For $2\leq k \leq K$, denote $\Pro_k$ the spectral projection associated to the eigenvalues that are $O\big(\e^{-2\frac{S(\mathbf m_k)}{h}}\big)$.
Then for any times $(t_k^\pm)_{1\leq k \leq K}$ satisfying 
$$t_K^-\geq h^{-1}|\ln (h^\infty)| \quad \text{and }\quad t_k^-\geq |\ln (h^\infty)|\e^{2\frac{S(\mathbf m_{k+1})}{h}}\quad \text{for }\quad k=1,\dots, K-1$$
as well as
$$t_1^+=+\infty \qquad \text{and }\quad t_k^+=O\Big(h^\infty\e^{2\frac{S(\mathbf m_{k})}{h}}\Big)\qquad \text{for }\quad k=2,\dots, K $$
one has 
$$\e^{-tP_h/h}=\Pro_k+O(h^\infty)\qquad \text{on }[t_k^-,t_k^+].$$ 
\end{cor}
\hip
In other words, we have shown the existence of timescales on which, during its convergence towards the global equilibrium, the solution of \eqref{Boltgene} will essentially visit the metastable spaces associated to the small eigenvalues of $P_h$.

%In fact, Theorems \ref{thmRobbe} and \ref{thmToto} and their corollaries apply for any collision operator $Q_h$ satisfying the result from Lemma \ref{hypom}.
The results presented in this paper should be reasonably easy to adapt to the case of collision operators satisfying Hypothesis \ref{hypom} with the space $S^0$ replaced by $S^\kappa$ for $\kappa \in [0,1/2[$ (we should get some expansions in powers of $h^{1-2\kappa}$ instead of just $h$).
Another perspective would then be to study the critical case $\kappa=1/2$ which should in particular cover the \emph{linear relaxation} collision operator corresponding to the linear BGK model
\begin{align}\label{qlr}
Q_h=h(1-\Pi_h)
\end{align}
 where $\Pi_h$ denotes the orthogonal projection on 
\begin{align}\label{Eh}
E_h=\mu_h\, L^2(\R^{d}_x)
\end{align}
and for which Robbe gave a first localization of the small spectrum of the associated operator $X_0^h+Q_h$ in \cite{Robbe}.

\section{Rough description of the small spectrum} \label{sectionrough}

%\color{red}Section à reprendre ac le nouvel énoncé
%\color{black}

Throughout the paper, we assume that Hypotheses \ref{hypom} and \ref{V} hold true.
This implies in particular that $Q_h$ is bounded uniformly in $h$ and self-adjoint in $L^2(\R^{2d})$.
Let us begin with a Lemma which consists in comparing our collision operator with the one introduced in \eqref{qlr} and studied in \cite{Robbe}.
This will in particular enable us to use some computations from \cite{Robbe} later on.

\begin{lem}\label{mino Q}
There exists $h_0>0$ such that for all $0<h<h_0$, 
$$Q_h\geq \frac hC (1-\Pi_h)$$
where $\Pi_h$ is the projection introduced in \eqref{qlr}.
In particular, $Q_h$ is non negative.
\end{lem}

\begin{proof}
Since the space $E_h$ defined in \eqref{Eh} is contained in $\mathrm{Ker}\,Q_h$, it is enough to prove that $\langle Q_hu,u \rangle \geq \frac hC \|u\|^2$ for $u \in E_h^\perp$.
Let $u\in E_h^\perp$ and recall the notations $H_0$ and $H_1$ from \eqref{h0} and \eqref{h1}.
Let us consider an approximate square root $A$ of $(1+H_1)$ given by
$$A=\mathrm{Op}_h\Big(\big(1+v^2/4+\eta^2+h(1-d/2)\big)^{1/2}\mathrm{Id}\Big)\in \Psi^0\big(\langle (v,\eta) \rangle\big).$$
By symbolic calculus, we easily have $A^2=1+H_1+h^2R_1$ with $R_1\in \Psi^0\big(\langle (v,\eta) \rangle^2\big)$.
Besides, the symbol of $A$ is clearly elliptic so $A$ is invertible and its inverse is also a pseudo-differential operator satisfying $A^{-2}=(1+H_1)^{-1}+h^2R_2$ with $R_2\in \Psi^0\big(\langle (v,\eta) \rangle^{-2}\big)$ (see for instance \cite{DimassiSjostrand}, chapter 8).
Thus, using the factorization from Hypothesis \ref{hypom} and the self-adjointness of $A$, we get
$$\langle Q_hu,u \rangle =\big\langle A\,\mathrm{Op}_h(M^h) A\, A^{-1}b_h u\, , \, A^{-1}b_h u \big\rangle.$$
Now according to Hypothesis \ref{hypom} and symbolic calculus again, the principal symbol of $A\,\mathrm{Op}_h(M^h) A$ is elliptic so we can use the G\aa{}rding inequality to write
\begin{align*}
\langle Q_hu,u \rangle &\geq \frac1C \big\langle A^{-2}b_h u\, ,  b_h  u \big\rangle\\
	&\geq \frac1C \big\langle b_h^* (1+H_1)^{-1} b_h u\, ,   u \big\rangle-\frac{h^2}{C} \big|\big\langle  b_h^*\, R_2 \,b_h u\, ,   u \big\rangle  \big|.
\end{align*}
Still using symbolic calculus, we get $b_h^*\, R_2 \,b_h=O(1)$ so applying \eqref{bhh1} we finally have 
$$\langle Q_hu,u \rangle \geq \frac1C \big\langle H_0(1+H_0)^{-1}  u\, ,   u \big\rangle-O(h^2)\| u\|^2$$
and the conclusion comes from the fact that the spectrum of $H_0(1+H_0)^{-1}|_{E_h^\perp}$ is contained in $[h/C , +\infty[$.
\end{proof}
We can already prove that $0$ is a simple eigenvalue of $(P_h,D)$ and that the other eigenvalues have positive real part.
It is easy to check that $\mathcal M_h$ defined in \eqref{muh} is in $\mathrm{Ker}\,P_h$.
Now let $\lambda \in \R$ and let us prove that for $u\in $ Ker $(P_h-i\lambda) $, one has $u\in \C\, \mathcal M_h$.
Since $X_0^h$ is skew-adjoint and %\color{red} 
$Q_h$ is self-adjoint and non-negative%\color{black}
, we have 
$$0=\mathrm{Re} \langle (P_h-i\lambda)u, u \rangle=\|Q_h^{1/2}u\|^2$$
so in particular $u\in $ Ker $Q_h=E_h$ according to Lemma \ref{mino Q}.
Therefore, $u=w\mu_h$ with $w \in  L^2(\R^d_x)$ and using that $\mu_h^{-1}X_0^hu=i \lambda w$ does not depend on $v$, we get in the sense of distributions $\partial_x (\e^{V/2h}w)=0$ which yields the desired result.

\subsection{Hypocoercivity}
Let us now use the dilatation operators
\begin{equation*}\label{ST}
S_h:\left\{\begin{aligned}
L^2(\R^{2d})&\to L^2(\R^{2d})\\ 
u&\mapsto h^{-d/2}u\Big(\frac{.}{\sqrt h}\,, \, \frac{.}{\sqrt h}\Big)
\end{aligned}
\right.
\qquad \qquad
T_h:\left\{\begin{aligned}
L^2(\R^d_x)&\to L^2(\R^d_x)\\ 
u&\mapsto h^{-d/4}u\Big(\frac{.}{\sqrt h}\Big)
\end{aligned}
\right.
\end{equation*}
that were introduced in \cite{Robbe} in which these were combined with a scaling of $\Pi_h$ to conjugate $P_h$ to a non-semiclassical operator with $h$-dependent potential.
In our case, it will enable us to use some computations and results already established in \cite{Robbe}.

\begin{lem}
Denoting 
$$X_0=v\cdot \partial_x -\partial_x V_h(x)\cdot\partial_v$$
where $V_h=h^{-1}V(\sqrt h \; \cdot )$,
$$\tilde Q_1=h^{-1}S_h^{-1}Q_hS_h$$
and  
$$\mathrm{Dom}\,(P)=\{u\in L^2(\R^{2d})\, ; \, X_0u \in L^2(\R^{2d})\},\qquad P=X_0+\tilde Q_1,$$
one has
$$(hP\,,\, \mathrm{Dom}(P))=(S_h^{-1}P_hS_h\,,\, S_h^{-1}D).$$
Moreover, 
$$(hP\,,\, \mathrm{Dom}(P))^*=(S_h^{-1}P_h^*S_h\,,\, S_h^{-1}D).$$
\end{lem}

\begin{proof} 
We have for $u \in L^2(\R^{2d})$
$$hX_0u=S_h^{-1}X_0^hS_hu$$
so using that $S_h$ is bounded we get $\mathrm{Dom}\,(P)=S_h^{-1}D$.
Consequently,
$$(hP\,,\, \mathrm{Dom}(P))=(S_h^{-1}P_hS_h\,,\, S_h^{-1}D)$$
and the result for the adjoint follows immediately.
\end{proof}
We also recall the notations of the following differential operators from \cite{Herau} and \cite{Robbe}:
%$$a=\begin{pmatrix}
%\partial_{x_1}+\partial_{x_1} V_h/2\\
%\vdots\\
%\partial_{x_d}+\partial_{x_d} V_h/2
%\end{pmatrix}
%\qquad \text{and }
%b=\begin{pmatrix}
%\partial_{v_1}+v_1/2\\
%\vdots\\
%\partial_{v_d}+v_d/2
%\end{pmatrix}
%$$
$$a=\partial_x +\frac{\partial_x V_h}{2}
%\begin{pmatrix}
%a_1 \\
%\vdots \\
%a_d
%\end{pmatrix}
\qquad ; \qquad b=\partial_v +\frac v2
%\begin{pmatrix}
%b_1 \\
%\vdots \\
%b_d
%\end{pmatrix}
\qquad \text{and } \qquad \Lambda^2=a^*a+b^*b+1.$$
%\begin{rema}\label{lambda}
The operator $(\Lambda ^2,\mathcal C^\infty_c(\R^{2d}))$ is essentially self-adjoint.
The Schwartz space $\mathcal S(\R^{2d})$ is included in the domain of its self-adjoint extension $(\Lambda^2,D(\Lambda^2))$ which is invertible.
%\end{rema}
%\hip
We can then define the operator $L=\Lambda^{-2}a^*b$, which is bounded uniformly in $h$ (see \cite{Robbe}, Lemma 2.7), as well as the perturbation $h \varepsilon (L+L^*)=O(h)$ where $\varepsilon>0$ will be chosen small enough later.\\
Besides, notice that $a^*a=-\Delta_x+|\partial_xV_h|^2/4-\Delta V_h/2=:\Delta_{V_h/2}$ is the Witten Laplacian in $x$ associated to the potential $V_h/2$ and that  
\begin{align*}
\Delta_{V/2}^h&:=hT_ha^*aT_h^{-1}\\
	&=-h^2\Delta_x+|\partial_xV|^2/4-h\Delta V/2
\end{align*}
is the semi-classical Witten Laplacian associated to the potential $V/2$.
The small spectrum of this operator was first studied by Helffer and Sjöstrand in \cite{HelSjo} and we now know that we can construct an orthonormal family $(\varphi_j)_{1\leq j \leq n_0}\subset \mathcal C^\infty_c(\R^{d}_x)$ of quasimodes associated to this operator given by
$$\varphi_j =\chi_j \e^{-\frac{V-V(x_j)}{2h}}$$
where $x_j$ is one of the local minima of $V$ and $\chi_j$ is a cut-off function localizing around $x_j$.
Recall the notation $\mu_h$ from \eqref{muh} and let us now define the families of functions  
$$g_j^h=\varphi_j\mu_h\qquad \text{and } \qquad g_j=S_h^{-1}g_j^h$$
for $1\leq j \leq n_0$.
These are actually quasimodes for our operators $P_h$ and $P_h^*$ :

\begin{lem}\label{quasi-ortho}
The family $(g_j^h)_{1\leq j \leq n_0}$ is orthonormal and there exists $\alpha>0$ such that for all $1\leq j\leq n_0$,
$$P_hg_j^h=O_{L^2}(\e^{-\frac{\alpha}{h}}),\qquad \quad P_h^*g_j^h=O_{L^2}(\e^{-\frac{\alpha}{h}}).$$
Moreover, $P_hg_j^h$ and $P_h^*g_j^h$ are in $\mathcal S (\R^{2d}) \subseteq D$ and we have 
$$P_h^*P_hg_j^h=O_{L^2}(\e^{-\frac{\alpha}{h}}),\qquad \quad P_hP_h^*g_j^h=O_{L^2}(\e^{-\frac{\alpha}{h}}).$$
\end{lem}

\begin{proof} The proof is the same as the one of Lemma 2.4 from \cite{Robbe} since with the notation \eqref{Eh} and Lemma \ref{mino Q} we also have $E_h=$ Ker $Q_h$.\end{proof}

One of the key results of this section is that the real part of the perturbation of our operator is bounded from below on a subspace of finite codimension given by the orthogonal of the quasimodes:
\begin{prop}\label{hypo}
Denote $N_{h,\varepsilon}^{\pm}$ the bounded self-adjoint operator $\mathrm{Id}\pm \varepsilon h (L+L^*)$. 
There exists $\varepsilon >0$ and $h_0>0$ such that for all $h \in ]0,h_0]$ and $u \in \mathcal S(\R^{2d})\cap (g_j)_{1\leq j\leq n_0}^{\perp}$, one has
$$\mathrm{Re}\langle N_{h,\varepsilon}^+Pu,u\rangle \geq \frac hC \|u\|^2$$
as well as
$$\mathrm{Re}\langle N_{h,\varepsilon}^-P^*u,u\rangle \geq \frac hC \|u\|^2.$$
%De plus
\end{prop}

\begin{proof}
One has for $u \in \mathcal S(\R^{2d})$, using the fact that $X_0$ is skew-adjoint:
\begin{align*}
\mathrm{Re}\langle N_{h,\varepsilon}^+ Pu,u\rangle &=\mathrm{Re}\langle Pu,N_{h,\varepsilon}^+ u\rangle \\
					&=\mathrm{Re}\langle \tilde Q_1u,N_{h,\varepsilon}^+ u\rangle +\mathrm{Re}\langle X_0 u,N_{h,\varepsilon}^+ u\rangle \\
					&=\|\tilde Q_1^{1/2}u\|^2+h \varepsilon \mathrm{Re}\langle \tilde Q_1u,(L+L^*) u\rangle +h \varepsilon \mathrm{Re}\langle X_0 u,(L+L^*) u\rangle\\
					&=\|\tilde Q_1^{1/2}u\|^2+h \varepsilon \mathrm{Re}\langle \tilde Q_1u,(L+L^*) u\rangle +h \varepsilon \mathrm{Re}\langle [L,X_0]u,u \rangle \\
					&=I+ h II+ h III
\end{align*}
Note that if we replace $P$ by $P^*$ and $N_{h,\varepsilon}^+$ by $N_{h,\varepsilon}^-$, we get $I-hII+hIII$.
%\begin{lem}\label{commu}
%For $u=(u_1, \dots, u_p) \in \mathcal C^{\infty}(\R^{2d},\R^p)$, let us denote $X_0u=(X_0u_1, \dots, X_0u_p)$.
%One has on $\mathcal C^{\infty}(\R^{2d})$,
%$$[b,X_0]=a, \; [b^*,X_0]=a^*, \;[X_0,a^*]=b^*\mathrm{Hess}V_h \quad \text{et} \quad [X_0,a]=\mathrm{Hess}V_hb.$$
%\end{lem}
Besides, it is also proven in \cite{Robbe} that 
\begin{align*}
[L,X_0]&=\mathcal A+\Lambda^{-2}a^*a
\end{align*}
where $\mathcal A$ is also bounded uniformly in $h$.
%as well as the following lemma, where $\tau$ was introduced in Theorem $\ref{Witt1}$, $G$ stands for Vect$\big((g_j)_{1\leq j \leq n_0}\big)\subset L^2(\R^{2d})$ and $\Pi_1$ denotes the orthogonal projection on $\mu_1L^2(\R^{d}_x)$:
%
%\begin{lem}\label{tau}
%For all $u \in \mathcal S(\R^{2d})\cap G^{\perp}$, one has $\mathrm{Re}\langle \Lambda^{-2}a^*a\Pi_1u,\Pi_1u\rangle \geq \frac{\tau}{4} \|\Pi_1u\|^2$.
%\end{lem}
Since $\|Q_h\|\leq C$ and $Q_h\geq \frac hC(1-\Pi_h)$ according to Lemma \ref{mino Q}, we get $\|\tilde Q_1\|\leq \frac Ch$ and $\tilde Q_1\geq \frac1C(1-\Pi_1)$.
Hence
\begin{align}\label{I+II}
I\pm h II&\geq I -h|II| \nonumber\\
	&\geq \|\tilde Q_1^{1/2}u\|^2-h \varepsilon \|\tilde Q_1u\| \| (L+L^*)u\| \nonumber\\
	&\geq \|\tilde Q_1^{1/2}u\|^2-\sqrt C h^{\frac12} \varepsilon  \|\tilde Q_1^{1/2}u\| \| (L+L^*)u\| \nonumber\\	
	&\geq \frac12 \|\tilde Q_1^{1/2}u\|^2 - 2Ch\varepsilon^2\|L\|^2\|u\|^2 \nonumber \\
	&\geq \frac{1}{2C} \|(1-\Pi_1)u\|^2 - 2Ch\varepsilon^2\|L\|^2\|u\|^2
\end{align}
We can combine this with the following estimate from \cite{Robbe} (proof of Proposition 2.5): there exists $\delta>0$ such that for $u \in (g_j)_{1\leq j\leq n_0}^{\perp}$,
$$III\geq-\frac{1}{4}\|(\mathrm{Id}-\Pi_1)u\|^2-\varepsilon^2\| \mathcal A\|^2\|u\|^2 +\frac{\varepsilon \delta}{4} \|\Pi_1u\|^2-\varepsilon \|(\mathrm{Id}-\Pi_1)u\|^2.$$
This yields for $\varepsilon<\frac{\delta}{4(\|\mathcal A\|^2+C\|L\|^2)}$ that
\begin{align}\label{III}
I\pm h II+h III&\geq \frac{1}{C}\|(\mathrm{Id}-\Pi_1)u\|^2+h \frac{\varepsilon \delta}{4} \|\Pi_1u\|^2-h\varepsilon^2\Big(\|\mathcal A\|^2+C\|L\|^2\Big)\|u\|^2\nonumber\\
	&\geq \frac hC \|u\|^2.
\end{align}
so the proof is complete.
\end{proof}
\hip
%\begin{rema}\label{etend}
This result extends to $u\in  (g_j)_{1\leq j\leq n_0}^{\perp}\cap \mathrm{Dom}\,(P)$ since $\mathcal S(\R^{2d})$ is a core for both $(P,\mathrm{Dom}\,(P))$ and $(P^*,\mathrm{Dom}\,(P^*))$.
It only differs from Proposition 2.5 in \cite{Robbe} by a factor $h$ in the estimate.
This comes from the fact that in our case, $\tilde Q_1=O(h^{-1})$ and not $O(1)$ (because $Q_h=O(1)$ and not $O(h)$) so we have to use a perturbation of order $h$ (the operator $N_{h,\varepsilon}^\pm$) to obtain the gain in $\|(1-\Pi_1)u\|^2$ in \eqref{I+II}.
As a consequence, the gain in $\|\Pi_1u\|^2$ from \eqref{III} is of order $h$ and not of order 1.

%\end{rema}

\begin{cor}\label{2.10}
There exists $c>0$ and $h_0>0$ such that for all $h \in ]0,h_0]$, $u \in D\cap (g_j^h)_{1\leq j\leq n_0}^{\perp}$ and $z \in \C$ uith $\mathrm{Re}\, z\leq ch^2 $
$$\|(P_h-z)u\|\geq ch^2\|u\|\qquad \text{and }\qquad \|(P_h^*-z)u\|\geq ch^2\|u\|.$$ 
\end{cor}
\hip
%This result also extends to $u\in  (g_j^h)_{1\leq j\leq n_0}^{\perp}\cap D$, just like in Remark \ref{etend}. 

\begin{proof}
Recall that $N_{h,\varepsilon}^+=1+O(h)$.
Hence, for $u \in D\cap (g_j^h)_{1\leq j\leq n_0}^{\perp}$, we have by putting $u=S_hw$ and using that $S_h$ is unitary
\begin{align*}
\|(P_h-z)u\|\|u\| &\geq \frac12 \|(P_h-z)u\|\|N_{h,\varepsilon}^+w\| \\
			&\geq \frac12 \mathrm{Re }\langle (P_h-z)u, S_h N_{h,\varepsilon}^+ w \rangle \\
			&= \frac 12 \mathrm{Re }\langle N_{h,\varepsilon}^+ (hP-z) w,w \rangle \\
			 &\geq \frac{h^2}{C} \|u\|^2 - \mathrm{Re }\, z \|N_{h,\varepsilon}^+\|\|u\|^2 \\
			&\geq \frac{h^2}{2C} \|u\|^2 
\end{align*} 
if $\mathrm{Re}\, z\leq h^2/2C $.
The same proof holds when replacing $P$ by $P^*$ and $N_{h,\varepsilon}^+$ by $N_{h,\varepsilon}^-$.
\end{proof}
\subsection{Resolvent estimates and first localization of the small eigenvalues}
\sloppy Using Lemma $\ref{quasi-ortho}$, it is clear that for $u\in \mathrm{Span}\big((g_j^h)_{1\leq j\leq n_0}\big)$ and $A\in \{P_h, P_h^*, P_h^*P_h, P_hP_h^*\}$ we have 
$$\|Au\|^2=O(\e^{-\frac{2\alpha}{h}})\|u\|^2.$$
Now if we denote $\Pro$ the orthogonal projection on $\mathrm{Span}\big((g_j^h)_{1\leq j\leq n_0}\big)$, we get by using Corollary $\ref{2.10}$ that for $z \in \C$ such that Re $z\leq ch^2$ and $u \in D$
\begin{align*}
\|(P_h-z)u\|^2&=\|(P_h-z)(\mathrm{Id}-\Pro)u+(P_h-z)\Pro u\|^2\\
		&=\|(P_h-z)(\mathrm{Id}-\Pro)u\|^2+\|(P_h-z)\Pro u\|^2+2 \mathrm{Re}\langle (P_h-z)(\mathrm{Id}-\Pro)u,(P_h-z)\Pro u\rangle \\
		&\geq c^2h^4 \|(\mathrm{Id}-\Pro)u\|^2 +|z|^2\|\Pro u\|^2 -O(\e^{-\frac{\alpha}{h}})\|u\|^2+2 \mathrm{Re}\langle (P_h-z)(\mathrm{Id}-\Pro)u,(P_h-z)\Pro u\rangle.
\end{align*}
The last term equals 
\begin{align*}
2 \mathrm{Re}\Big[ \langle (\mathrm{Id}-\Pro)u,P_h^*P_h\Pro u\rangle-z\langle (\mathrm{Id}-\Pro)u,P_h\Pro u\rangle-\bar z\langle (\mathrm{Id}-\Pro)u,P_h^*\Pro u\rangle\Big]=(1+|z|)O(\e^{-\frac{\alpha}{h}})\|u\|^2.
\end{align*}
Therefore choosing $\tilde c \leq c$, there exists $h_0>0$ such that for $h\leq h_0$ and $z$ such that $\tilde c h^2\leq |z| \leq c h^2$
\begin{align*}
\|(P_h-z)u\|^2&\geq \Big(|z|^2+O(\e^{-\frac{\alpha}{h}})\Big)\|u\|^2\geq \frac{\tilde c^2h^4}{2}\|u\|^2.
\end{align*}
Once again, the same estimate holds with $P_h^*$ instead of $P_h$ and since the annulus we are working on is invariant by complex conjugation we also have
$$\|(P_h-z)^*u\|\geq\frac{\tilde c h^2}{2}\|u\|\,.$$
%\hspace*{\fill} $\Box$
Therefore, we get the following resolvent estimate on the annulus centered in 0 and of radiuses $\tilde ch^2$ and $ch^2$:
\begin{align}\label{annulus}
\|(P_h-z)^{-1}\|=O(h^{-2}) \quad \text{for } \tilde c h^2\leq |z| \leq c h^2.
\end{align}
We can now consider the spectral projection 
\begin{align}\label{Pi0}
\Pi_0=\frac{1}{2i\pi}\int_{|z|=c h^2}(z-P_h)^{-1}\D z
\end{align}
and its range that we denote $H$.
This operator will yield some information on $\mathrm{Spec}( P_h)\cap B(0, c h^2)$ and therefore enable us to prove the main statement from Theorem \ref{thmRobbe}.

The main point is that $H$ is of dimension $n_0$.
It can be obtained by a direct adaptation of the proof of Proposition 3.1 from \cite{Robbe}.
Hence $\mathrm{Spec}( P_h)\cap B(0, c h^2)$ which is the same as $\mathrm{Spec}( P_h|_H)$ consists of $n_0$ eigenvalues (counted with algebraic multiplicity).
Here again, our result slightly differs from the one in \cite{Robbe} as we do not rule out the possibilities that $P_h|_H$ contains some Jordan blocks and that some of its eigenvalues are not real.
It only remains to prove that these are exponentially small with respect to $1/h$.
We begin by noticing that thanks to Lemma \ref{quasi-ortho}, we have $(z-P_h)g_j^h=zg_j^h+O(\e^{-\frac{\alpha}{h}})$ and $(z-P_h^*)g_j^h=zg_j^h+O(\e^{-\frac{\alpha}{h}})$ from which we easily deduce
\begin{align}\label{pi0gjh}
\Pi_0 g_j^h=g_j^h+O(\e^{-\frac{\alpha}{h}}) \qquad \text{and} \qquad \Pi_0^* g_j^h=g_j^h+O(\e^{-\frac{\alpha}{h}}).
\end{align} 
In particular, $(\Pi_0g_j^h)_{1\leq j \leq n_0}$ is almost orthonormal so for $u=\sum u_j \Pi_0g_j^h \in H$, we have
$$\|u\|^2=\big(1+O(\e^{-\alpha/h})\big)\sum_{j=1}^{n_0}|u_j|^2.$$
Therefore it is enough to prove that $P_h$ is exponnentially small on $(\Pi_0g_j^h)_{1\leq j \leq n_0}$.
But thanks to the resolvent estimate \eqref{annulus}, it is easy to see that $\Pi_0=O(1)$ and since $P_h$ and $\Pi_0$ commute, we get the desired result.
%\hspace*{\fill} $\Box$

To complete the proof of Theorem \ref{thmRobbe}, it only remains to show the existence of the resolvent on $\{\mathrm{Re}\, z \leq ch^2\}\backslash B(0,\tilde ch^2)$ as well as the estimate in $O(h^{-2})$.

\begin{lem}\label{v+r}
Denote $\hat \Pi_0=1-\Pi_0$.
For all $u \in L^2(\R^{2d})$, we have
$$\hat \Pi_0u=w+r$$
whith $w \in (g_j^h)_{1\leq j\leq n_0}^{\perp}$ and $r \in \mathrm{Span}\big((g_j^h)_{1\leq j\leq n_0}\big)$ satisfying 
$r=O(\e^{-\frac{\alpha}{h}})\|\hat \Pi_0u\|$.
\end{lem}

\begin{proof} 
First we take for $r$ the orthogonal projection of $\hat \Pi_0u$ on $\mathrm{Span}\big((g_j^h)_{1\leq j\leq n_0}\big)$.
Then we notice that using \eqref{pi0gjh}, we get
$$\langle g_j^h,\hat \Pi_0u\rangle =\langle \hat \Pi_0^*g_j^h,\hat \Pi_0u\rangle =O(\e^{-\frac{\alpha}{h}})\|\hat \Pi_0u\|$$
which implies the announced estimate.
\end{proof}

\begin{lem}\label{detail}
For all $r' \in \mathrm{Span}\big((g_j)_{1\leq j\leq n_0}\big)$, we have $N_{h,\varepsilon}^\pm r' \in \mathrm{Dom}\,(P^*)= \mathrm{Dom}\,(P)$.
Moreover, the restrictions to the finite dimensional subspace $\mathrm{Span}\big((g_j)_{1\leq j\leq n_0}\big)$ of the operators $PN_{h,\varepsilon}^\pm$ and $P^*N_{h,\varepsilon}^\pm$ are all $O(1)$.
\end{lem}

\begin{proof}
For the first statement, it is sufficient to show that for $1\leq j \leq n_0$, the functions $Lg_j$ and $L^*g_j$ are both in $\mathrm{Dom}\, (P)$.
But we have in the sense of distributions
\begin{align}\label{X0L}
X_0Lg_j%\overset{\mathcal D'}{
=%}
[X_0,L]g_j+LX_0g_j
\end{align}
and we saw in the proof of Proposition \ref{hypo} that $[X_0,L]$ is a bounded operator on $L^2(\R^{2d})$ so it is then clear that $X_0Lg_j \in L^2(\R^{2d})$ i.e $Lg_j \in \mathrm{Dom}\,(P)$.
The same goes easily for $L^*g_j$.
For the second statement, using Lemma \ref{quasi-ortho} and the fact that $\tilde Q_1=O(h^{-1})$, it suffices to notice that for $1\leq j \leq n_0$, \eqref{X0L} implies that $X_0Lg_j$ and $X_0L^*g_j$ are both $O(1)$ as we saw that $L$ and $[X_0,L]$ are $O(1)$.
%De la même manière, il suffit d'après la Proposition $\ref{96}$ de montrer que $PB_\varepsilon^\pm g_j$ et $P^*B_\varepsilon^\pm g_j$ sont bornés uniformément en $h\leq 1$ pour conclure la preuve.
%Mais d'après les Lemmes $\ref{blam}$ et $\ref{borne}$, $[X_0,L]$ est borné uniformément en $h$, ce qui combiné avec le Lemme $\ref{quasi-ortho}$ et ($\ref{X0L}$) implique bien $PB_\varepsilon^\pm g_j=O(1)=P^*B_\varepsilon^\pm g_j$.
\end{proof}

\begin{prop}\label{1-pi0l2}
Consider $\hat P_h$ the restriction of $P_h$ to $\hat \Pi_0D$ acting on $\hat \Pi_0L^2(\R^{2d})$.
Then for all $z\in \C$ such that $\mathrm{Re}$ $z\leq ch^2$, the resolvent $(\hat P_h-z)^{-1}$ exists and we have the uniform estimate $$(\hat P_h-z)^{-1}=O(h^{-2}).$$
\end{prop}

\begin{proof}
We actually prove that the result of Proposition $\ref{hypo}$ remains true when replacing the set $(g_j)_{1\leq j\leq n_0}^\perp\cap \mathrm{Dom}\,(P)$ by $S_h^{-1}\hat \Pi_0D$.
We will deduce that the result of Corollary $\ref{2.10}$ also remains true when taking $u \in \hat \Pi_0D$ instead of $(g_j^h)_{1\leq j\leq n_0}^\perp\cap D$, which is precisely the statement that we want to prove.
Let $u \in D$, using the notations from Lemma $\ref{v+r}$ we have
\begin{align*}
\mathrm{Re}\,\langle PS_h^{-1}\hat \Pi_0u,N_{h,\varepsilon}^+S_h^{-1}\hat \Pi_0u\rangle &=\mathrm{Re}\,\langle PS_h^{-1}w,N_{h,\varepsilon}^+S_h^{-1}w\rangle +\mathrm{Re}\,\langle PS_h^{-1}w,N_{h,\varepsilon}^+S_h^{-1}r\rangle \\
		&\qquad \qquad +\mathrm{Re}\,\langle PS_h^{-1}r,N_{h,\varepsilon}^+S_h^{-1}w\rangle +\mathrm{Re}\,\langle PS_h^{-1}r,N_{h,\varepsilon}^+S_h^{-1}r\rangle \,.
\end{align*}
Now let us denote $w'=S_h^{-1}w\in (g_j)_{1\leq j\leq n_0}^\perp\cap \mathrm{Dom}\,(P)$ and $r'=S_h^{-1}r\in \mathrm{Span}\big((g_j)_{1\leq j\leq n_0}\big)$.
We can use Proposition $\ref{hypo}$ as well as Lemmas \ref{v+r} and \ref{detail} to get
\begin{align*}
\mathrm{Re}\,\langle N_{h,\varepsilon}^+PS_h^{-1}\hat \Pi_0u,S_h^{-1}\hat \Pi_0u\rangle &=\mathrm{Re}\,\langle Pw',N_{h,\varepsilon}^+w'\rangle +\mathrm{Re}\,\langle w',P^*N_{h,\varepsilon}^+r'\rangle +\mathrm{Re}\,\langle N_{h,\varepsilon}^+Pr',w'\rangle +\mathrm{Re}\,\langle Pr',N_{h,\varepsilon}^+r'\rangle \\
		&\geq \frac hC \|w\|^2-O\big(\|w\|\,\|r\|\big)-O(\e^{-\frac{\alpha}{h}}\|r\|)\\
		&\geq \frac {h}{2C} \|S_h^{-1}\hat \Pi_0u\|^2.
\end{align*}
As usual, all of the above remains true with $P^*$ and $N_{h,\varepsilon}^-$ instead of $P$ and $N_{h,\varepsilon}^+$ so the proof is now complete.
\end{proof}
\hip
\textit{End of Proof of Theorem \ref{thmRobbe}} : 
Let $z\in \C$ satisfying $\mathrm{Re}\, z\leq ch^2$ and $|z|\geq \tilde ch^2$ and recall the notation $H=\mathrm{Ran}\,\Pi_0$.
We already know from Proposition \ref{1-pi0l2} that $\hat P_h-z$ is invertible, but it is clearly also the case of $P_h|_H-z$ since $P_h|_H=O(\e^{-\alpha/h})$.
Therefore $P_h-z$ is invertible and we have
\begin{align}\label{Ph-z-1}
(P_h-z)^{-1}=(\hat P_h-z)^{-1}\hat \Pi_0+(P_h|_H-z)^{-1}\Pi_0.
\end{align}
Besides, we easily have for such $z$ that $\|(P_h|_H-z)u\|\geq \frac 1C  h^2\|u\|$ which combined with \eqref{Ph-z-1},  Proposition \ref{1-pi0l2} and the fact that $\|\Pi_0\|=O(1)$ yields the estimate $(P_h-z)^{-1}=O(h^{-2})$.
\hspace*{\fill} $\Box$

\section{Accurate quasimodes} \label{sectionquasim}

\subsection{General form}
%\subsection{Topology of the potential}\label{topo}
Let us denote 
 $$W(x,v)=\frac{V(x)}{2}+\frac{v^2}{4}$$
the global potential on $\R^{2d}$.
Before we can construct our quasimodes, we need to recall the general labeling of the minima which originates from \cite{HeKlNi} and was generalized in \cite{HHS11}, as well as the topological constructions that go with it.
In our case, it has to be done for the global potential, i.e the function $W$.
However, by the definition of $W$, a strong connection between these constructions for $W$ and the ones for $V$ will appear, leading to simplifications.
In order to give a proper statement about this connection, let us construct the labelings for both $W$ and $V$.
To this aim, we consider $d'\in \N^*$ and a smooth Morse function $Y$ on $\R^{d'}$ bounded from below, having at least two local minima and such that $|\nabla Y|\geq 1/C$ outside of a compact.
According to Hypothesis \ref{V}, one can for instance take $Y=V/2$ or $Y=W$ and recall that as we discussed following Hypothesis \ref{V}, it implies that $Y(X)\geq |X|/C$ outside of a compact.
We also denote $\mathcal U^{(k),Y}$ the critical points of $Y$ of index $k$.
For shortness, we will write "CC" instead of "connected component".

\begin{lem}\label{1.4}
If $X\in \mathcal U^{(1),Y}$, then there exists $r_0>0$ such that for all $0<r<r_0$, $X$ has a connected neighborhood $U_r$ in $B(X,r)$ such that $U_r\cap \{Y< Y(X)\}$ has exactly 2 CCs.\\
\end{lem}

\begin{proof} Let $X\in \mathcal U^{(1),Y}$; according to the Morse Lemma, there exists a connected neighborhood $U_r$ of $X$, $r'>0$ and $\varphi:U_r\to B(0,r')$ a smooth diffeomorphism such that 
$$Y\circ \varphi^{-1} =Y(X)+\frac12 \langle \mathrm{Hess}_{X}Y \, \cdot, \cdot \rangle.$$
Besides, it is easy to see that
$$U_r\cap \{Y< Y(X)\}=\varphi^{-1}\big(\{y \in B(0,r')\, ;\, \langle \mathrm{Hess}_{X}Y \, y, y \rangle<0\}\big)$$
and $\{y \in B(0,r')\, ;\, \langle \mathrm{Hess}_{X}Y \, y, y \rangle<0\}$ has exactly 2 CCs.
\end{proof}

\begin{lem}\label{1.4bis}
Let $X\in \R^{d'}$ and suppose there exists $r_0>0$ such that for every neighborhood $U$ of $X$ in $B(X,r_0)$, the set $U\cap \{Y< Y(X)\}$ is not connected. Then $X\in \mathcal U^{(1),Y}$.
\end{lem}

\begin{proof} 
First we clearly have that $\nabla Y(X)= 0$ since otherwise one could use the implicit function theorem to find a neighborhood $U$ of $X$ in $B(X,r_0)$ such that $U\cap \{Y< Y(X)\}$ is connected.
It is also clear that $X\notin \mathcal U^{(0),Y}$ so let us assume by contradiction that $X\in \mathcal U^{(k),Y}$ with $k \geq 2$.
Then using the Morse Lemma as in the proof of Lemma \ref{1.4}, we would once again get that $X$ has a neighborhood $U$ in $B(X,r_0)$ such that $U\cap \{Y< Y(X)\}$ %is connected.
has the same number of CCs as $\{y \in B(0,r)\, ;\, \langle \mathrm{Hess}_{X}Y \, y, y \rangle<0\}$ which is connected since $k\geq 2$.
Hence $X$ has to be in $\mathcal U^{(1),Y}$.
\end{proof}
\hip
In view of the result from Lemma \ref{1.4} and following the approach from \cite{HeKlNi,HHS11}, we give the following definition:
\begin{defi}\label{ssv}
\begin{enumerate}
\item We say that $X\in \mathcal U^{(1),Y}$ is a separating saddle point and we denote $X\in \mathcal V^{(1),Y}$ if for every $r>0$ small enough, the two CCs of $U_r\cap \{Y< Y(X)\}$ are contained in different CCs of $\{Y< Y(X)\}$.
\item We say that $\sigma \in \R$ is a separating saddle value if $\sigma \in Y(\mathcal V^{(1),Y})$.
\item Finally, we say that a set $E\subset \R^{d'}$ is critical if there exists $\sigma \in Y(\mathcal V^{(1),Y})$ such that $E$ is a CC of $\{Y< \sigma\}$ satisfying $\partial E \cap \mathcal V^{(1),Y}\neq \emptyset$.
\end{enumerate}
\end{defi}

\begin{lem}\label{2min}
Let $\mathbf m$, $\mathbf m'$ two distinct local minima of $Y$.
The real number
$$\sigma=\sup \big\{a \in \R\, ;\, \mathbf m \text{ and }\mathbf m' \text{ are in two different CCs of }\{Y<a\}\big\}$$
is well defined and $\{Y<\sigma\}$ has at least two CCs $\Omega \ni \mathbf m$ and $\Omega' \ni \mathbf m'$.
Moreover, $\sigma$ is a separating saddle value and $\Omega$, $\Omega'$ are critical.%and the set $\overline{\Omega} \cap \overline{\Omega'}$ is non empty and included in $\mathcal V^{(1)}$.
%In particular
\end{lem}

\begin{proof} 
We can assume that $Y(\mathbf m)\leq Y(\mathbf m')$ so taking $a:=\inf_{\mathcal A}Y$ where $\mathcal A$ is a well chosen annulus centered in $\mathbf m'$, we see that 
\begin{align}\label{2ccdiff}
\big\{a \in \R\, ;\, \mathbf m \text{ and }\mathbf m' \text{ are in two different CCs of }\{Y<a\}\big\}\neq \emptyset
\end{align}
and it is then clear that $\sigma$ is well defined.
Besides, if $(\sigma_n)_{n\geq 1}$ is an increasing sequence in the set from \eqref{2ccdiff} that converges towards $\sigma$ and $\gamma:[0,1]\to \R^{d'}$ is a continuous path linking $\mathbf m$ and $\mathbf m'$, then 
$$\gamma([0,1])\cap \big(\R^{d'}\backslash \{Y<\sigma\}\big)=\bigcap_{n \geq 1}\gamma([0,1])\cap \big(\R^{d'}\backslash \{Y<\sigma_n\}\big)$$
is non empty by compactness so we can consider $\Omega \ni \mathbf m$ and $\Omega' \ni \mathbf m'$ two different CCs of $\{Y<\sigma\}$.
To prove that $\sigma$ is a separating saddle value, we will actually show that there exists %$(\mathbf m'',0)\in\mathcal U^{(0)}\backslash\{\mathbf m\}$ such that the 
a CC of $\{Y<\sigma\}$ %containing $(\mathbf m'',0)$ 
that we denote $\Omega''$ which is not $\Omega$ and satisfies $\overline{\Omega}\cap\overline{\Omega''}\neq \emptyset$.
Assume by contradiction that there exists $\varepsilon>0$ such that $(\Omega+B(0,\varepsilon))\backslash \overline\Omega$ is included in $\{Y\geq \sigma\}$. 
In that case, the points of $(\Omega+B(0,\varepsilon))\backslash \overline\Omega$ on which $Y$ takes the value $\sigma$ are local minima of $Y$ which is a Morse function, so there are finitely many such points.
%Since $Y$ is a Morse function, there are finitely many points in $(\Omega+B(0,\varepsilon))\backslash \Omega$ on which $Y$ takes the value $\sigma$
Thus, up to taking $\varepsilon$ smaller, we can assume that 
$$\Gamma:=\mathrm{dist}(\cdot,\Omega)^{-1}(\{\varepsilon\})\subseteq \{Y>\sigma\}.$$
Hence there exists $\delta>0$ such that the minimum of $Y$ on $\Gamma$ is $\sigma+\delta$.
Since any continuous path linking $\mathbf m$ and $\mathbf m'$ has to cross $\Gamma$, $\mathbf m$ and $\mathbf m'$ are in two different CCs of $\{Y<\sigma+\delta/2\}$.
%Then there exists $\varepsilon>0$ such that $\Omega$ and $\Omega'$ are separated by $(\Omega+2\varepsilon) \backslash (\Omega+B(0,\varepsilon))$.
%Choosing $\varepsilon$ small enough, we can suppose that there is no critical point of $Y$ in $(\Omega+2\varepsilon) \backslash \Omega$ so 
%$\inf_{(\Omega+2\varepsilon) \backslash (\Omega+B(0,\varepsilon))}Y>\sigma.$
This contradicts the maximality of $\sigma$ and proves the existence of $\Omega''$.
%In the set from (\ref{2ccdiff}), the two CCs mentionned depend continuously on $a$, so by maximality of $\sigma$, $\overline{\Omega} \cap \overline{\Omega'}\neq \emptyset$.
Hence, Lemma \ref{1.4bis} implies that $\overline{\Omega} \cap \overline{\Omega''} \subseteq \mathcal U^{(1),Y}$ and then $\overline{\Omega} \cap \overline{\Omega''} \subseteq \mathcal V^{(1),Y}$ follows obviously from the definition of $\mathcal V^{(1),Y}$.
\end{proof}

Thanks to Lemma \ref{2min}, we know that $\mathcal V^{(1),Y}\neq \emptyset$. Let us then denote $\sigma_2>\dots>\sigma_N$ where $N\geq 2$ the different separating saddle values of $Y$ and for convenience we set $\sigma_1=+\infty$.
We call \emph{labeling} of the minima of $Y$ any injection $l:\mathcal U^{(0),Y}\to \llbracket 1,N \rrbracket \times \N^*$.
If $l(\mathbf m)=(k,j)$, we denote for shortness $\mathbf m=\mathbf m_ {k,j}$.
We are going to introduce the usual labeling of the minima for a potential $Y$ (see for instance \cite{HeKlNi, HHS11, LPMichel}).
We adopt a slightly unusual point of view in order to facilitate the establishment of the correspondence between the constructions for $W$ and the ones for $V/2$ that we will state later on.
For $\sigma \in \R\cup\{+\infty\}$, let us denote $\mathcal C_\sigma^Y$ the set of all the CCs of $\{Y<\sigma\}$.
Given a labeling $l$ of the minima, we denote for $k\in \llbracket 1,N \rrbracket$ 
$$\mathcal U^{(0),Y}_k=l^{-1}(\llbracket 1,k \rrbracket\times \N^*)\cap \{Y<\sigma_k\}$$
and we say that the labeling is \emph{adapted} to the separating saddle values if for all $k\in \llbracket 1,N \rrbracket$, each element of $l^{-1}(\{k\}\times \N^*)$ is a global minimum of $Y$ restricted to some CC of $\{Y<\sigma_k\}$ and the map $T_k^Y: \mathcal U^{(0),Y}_k \to \mathcal C_ {\sigma_k}^Y$ sending $\mathbf m \in \mathcal U^{(0),Y}_k$ on the element of $\mathcal C_ {\sigma_k}^Y$ to which it belongs is bijective.
In particular, $l^{-1}(\{k\}\times \N^*)$ is contained in $\mathcal U^{(0),Y}_k$.
Such labelings exist, one can for instance easily check that the usual labeling procedure presented in \cite{HHS11} is adapted to the separating saddle values.

\begin{lem}\label{critic}
Under an adapted labeling of the minima of $Y$, for any $2\leq k \leq N$, the elements of $T_k^Y\big(l^{-1}(\{k\}\times \N^*)\big)$ are critical.
\end{lem}

\begin{proof} 
Let $\mathbf m_{k,j}\in l^{-1}(\{k\}\times \N^*)$. There exists a CC of $\{ Y<\sigma_{k-1}\}$ that we call $E$ which is such that $T_k^Y(\mathbf m_{k,j})\subseteq E$ and $E$ contains some $\mathbf m_{k',j'}\in E$ for $1\leq k'\leq k-1$ and $j' \in \N^*$ by bijectivity of $T_{k-1}^Y$.
Therefore, $\mathbf m_{k',j'}$ and $\mathbf m_{k,j}$ are in the same CC of $\{ Y<\sigma_{k-1}\}$ but are not both in $T_k^Y(\mathbf m_{k,j})$ this time by bijectivity of $T_k^Y$.
Applying Lemma \ref{2min} to $\mathbf m_{k',j'}$ and $\mathbf m_{k,j}$, %the separating saddle value that we get has to be $\sigma_k$ so $E_{k,j}$ is one of the CCs of $\{ Y<\sigma_{k}\}$ called $\Omega$ and $\Omega'$ in Lemma \ref{2min} and which are critical.
we obtain a separating saddle value $\tilde \sigma$ which is the maximal real number such that $\mathbf m_{k',j'}$ and $\mathbf m_{k,j}$ are in two different CCs of $\{ Y<\tilde\sigma\}$.
Therefore we get $\tilde \sigma=\sigma_k$ so $T_k^Y(\mathbf m_{k,j})$ is one of the CCs of $\{ Y<\tilde \sigma\}$ called $\Omega$ and $\Omega'$ in Lemma \ref{2min} and which are critical.
\end{proof}

\begin{defi}\label{j et s}
Given an adapted labeling, we can now define the following mappings:
\begin{enumerate}[label=\textbullet]
\item $E^Y: \mathcal U^{(0),Y}\xrightarrow{{}\quad{}} \mathcal P(\R^{d'})$\\
	${} \; \mathbf m_{k,j} \xmapsto{{}\quad{}} T_k^Y(\mathbf m_{k,j})$.
\item 
$\mathbf j^Y:\mathcal U^{(0),Y}\to \mathcal P\big(\mathcal V^{(1),Y}\cup \{\mathbf s_1\}\big)$\\
given by $\mathbf j^Y(\underline{\mathbf m})=\mathbf s_1$ where $\mathbf s_1$ is a fictive saddle point such that $Y(\mathbf s_1)=\sigma_1=+\infty$; and for $2\leq k\leq N$, $\mathbf j^Y(\mathbf m_{k,j})=\partial E^Y(\mathbf m_{k,j})\cap \mathcal V^{(1),Y}$ which is not empty according to Lemma \ref{critic} and included in $\{Y=\sigma_k\}$.
\item
$\boldsymbol \sigma^Y:\mathcal U^{(0),Y}\to Y(\mathcal V^{(1),Y})\cup \{\sigma_1\}$\\
${} \quad \mathbf m \mapsto Y(\mathbf j^Y(\mathbf m))$\\
where we allow ourselves to identify the set $Y(\mathbf j^Y(\mathbf m))$ and its unique element in $Y(\mathcal V^{(1),Y})\cup \{\sigma_1\}$.
\item
$S^Y: \mathcal U^{(0),Y}\xrightarrow{{}\quad{}} ]0,+\infty]$\\
	${} \quad \mathbf m \xmapsto{{}\quad{}} \boldsymbol \sigma^Y(\mathbf m)-Y(\mathbf m)$.
\end{enumerate}
\end{defi}

Let us now state a Lemma that will enable us to show that, roughly speaking, the previous constructions for $Y=V/2$ are the projections on $\R^d_x$ of the ones for $Y=W$.
First, we give the following easy observation.

\begin{rema}\label{obs}
By definition of $W$, we have $V/2=W(\cdot,0)$. Moreover, if $(x_0,v_0)\in \{W<\sigma\}$, then $\{x_0\}\times B(0,|v_0|)\subseteq \{W<\sigma\}$.
\end{rema}
\hip
For shortness, we denote $\mathcal C_{\sigma}=\mathcal C_{\sigma}^{V/2}$ and $\widetilde{\mathcal C}_{\sigma}=\mathcal C_ {\sigma}^{W}$ as well as $\mathcal U^{(k)}=\mathcal U^{(k),V/2}$ and $\widetilde{\mathcal U}^{(k)}=\mathcal U^{(k),W}$ (we do similarly with $\mathcal V$ or $\mathcal U_k$ instead of $\mathcal U$).
Notice that $\widetilde{\mathcal U}^{(k)}=\mathcal U^{(k)}\times\{0\}$.
We introduce the natural projection $\pi_x:\R^{2d}\to \R^d_x$ sending $(x,v)$ on $x$ that we also consider as a map from $\mathcal P(\R^{2d})$ to $\mathcal P(\R^d_x)$.

\begin{lem}\label{corres}
For all $\sigma \in \R$, the projection $\pi_x$ sends $\widetilde{\mathcal C}_{\sigma}$ in $\mathcal C_ {\sigma}$.
Moreover, the map $\pi_x:\widetilde{\mathcal C}_{\sigma} \to  \mathcal C_ {\sigma}$ is bijective.
\end{lem}

\begin{proof}
The proof of the first statement is an easy consequence of Remark \ref{obs}.
For the second statement, let $x\in E\in \mathcal C_{\sigma}$ and denote $\tilde E$ the element of $\widetilde{\mathcal C}_ {\sigma}$ containing $(x,0)$. By the first statement, we necessarily have $\pi_x(\tilde E)= E$ so we have shown the surjectivity.
Now let $\tilde E_1$, $\tilde E_2 \in \widetilde{\mathcal C}_ {\sigma}$ such that $\pi_x(\tilde E_1)=\pi_x(\tilde E_2)= E_1$.
Let also $(x_1,v_1)\in \tilde E_1$ and $(x_2,v_2)\in \tilde E_2$.
Since $x_1$, $x_2\in  E_1$, there exists a path $(\gamma(t),0)$ from $(x_1,0)$ to $(x_2,0)$ contained in $\{W<\sigma\}$. Thus, the concatenation of the paths $(x_1,(1-t)v_1)$, $(\gamma(t),0)$ and $(x_2,tv_2)$ yields a path linking $(x_1,v_1)$ and $(x_2,v_2)$ in $\{W<\sigma\}$.
Hence $\tilde E_1=\tilde E_2$ and we get the injectivity.
\end{proof}

\begin{prop}\label{lienVW}
\begin{enumerate}
\item We have $\widetilde{\mathcal V}^{(1)}=\mathcal V^{(1)}\times \{0\}$. In particular, $V/2$ and $W$ have the same separating saddle values.
\item A set $\tilde E\in \widetilde{\mathcal C}_\sigma$ is critical if and only if $\pi_x(\tilde E)$ is critical.
\item A labeling $((\mathbf m,0)_ {k,j})_{k,j}$ is adapted to $W$ if and only if $(\mathbf m_ {k,j})_{k,j}$ is adapted to $V/2$.
\end{enumerate}
\hip
Moreover, given an adapted labeling, the mappings from Definition \ref{j et s} satisfy
$$E^{V/2}(\mathbf m_ {k,j})=\pi_x\big(E^{W}(\mathbf m_ {k,j},0)\big)\quad \text{and}\quad \mathbf j^W(\mathbf m_{k,j},0)=\mathbf j^{V/2}(\mathbf m_{k,j})\times\{0\}.$$
\end{prop}

\begin{proof}
Let $\tilde E \in \widetilde{\mathcal C}_\sigma$. Thanks to Remark \ref{obs}, we easily have 
\begin{align}\label{projadh}
(x,0)\in \partial\tilde E \iff x\in \partial\big(\pi_x(\tilde E)\big).
\end{align}
$a)$: We already know that $\widetilde{\mathcal U}^{(1)}=\mathcal U^{(1)}\times\{0\}$. Besides, we easily deduce from \eqref{projadh} and Lemma \ref{corres} that $(\mathbf s,0)\in\widetilde{\mathcal U}^{(1)}$ is in the closure of two distinct CCs of $\{W<W(\mathbf s,0)\}$ if and only if $\mathbf s\in\mathcal U^{(1)}$ is in the closure of two distinct CCs of $\{V<V(\mathbf s)\}$ so the first item is proven.\\
$b)$: This is also a straightforward consequence of \eqref{projadh} and Lemma \ref{corres} combined with item $a)$.\\
$c)$: Let $\tilde E\in\widetilde{\mathcal C}_{\sigma_k}$. By Remark \ref{obs}, we easily have 
\begin{align}\label{minssimin}
 \quad (\mathbf m,0) \text{ is a global minimum of } W|_{\tilde E} \iff \mathbf m \text{ is a global minimum of } V|_{\pi_x(\tilde E)}.
\end{align}
Besides, since $\widetilde{\mathcal U}^{(0)}_k=\mathcal U^{(0)}_k\times\{0\}$, we have that $\pi^k$ defined as $\pi_x:\widetilde{\mathcal U}^{(0)}_k \to \mathcal U^{(0)}_k$ is bijective.
We can then conclude as 
\begin{align}\label{TWTV}
T_k^W=\pi_x^{-1} \circ T_k^{V/2} \circ \pi^k
\end{align}
 where $\pi_x$ denotes the bijective map from Lemma \ref{corres}.\\
The last statement is a direct consequence of \eqref{TWTV}, \eqref{projadh} and item $a)$.
\end{proof}
\hip
From now on, we fix a labeling $(\mathbf m_{k,j})_ {k,j}$ adapted to $V$.
Note that $\boldsymbol \sigma^{V/2}(\mathbf m)=\boldsymbol \sigma^{W}(\mathbf m,0)$ and $S^{V/2}(\mathbf m)=S^W(\mathbf m,0)$.
For shortness, we will denote in the rest of the paper $\mathbf j=\mathbf j^{V/2}$, $\boldsymbol \sigma=\boldsymbol \sigma^{V/2}$ and $S=S^{V/2}$.
However, be careful that we choose to denote $E=\pi_x^{-1}\circ E^{V/2}$ so that the range of $E$ is in $\mathcal P(\R^{2d})$.
Following \cite{BoGaKl, HeKlNi, HHS11, LPMichel}, we can now state our last assumption that allows us to treat the generic case.
As mentionned in the introduction, this assumption could actually be omitted (see \cite{Michel} or \cite{BonyLPMichel}) but this would introduce additionnal difficulties that are not the main concern of this paper.

\begin{hypo}\label{jvide}
For all $\mathbf m \in \mathcal U^{(0)}$, we have 
\begin{enumerate}[label=\alph*)]
\item $\mathbf m$ is the only global minimum of $V|_{E^{V/2}(\mathbf m)}$ \label{mseul}
\item for any $\mathbf m' \in \mathcal U^{(0)}\backslash \{\mathbf m\}$, the sets $\mathbf j(\mathbf m)$ and $\mathbf j(\mathbf m')$ do not intersect.\label{jvide2}
\end{enumerate}
\end{hypo}
\hip
According to Proposition \ref{lienVW} and \eqref{minssimin}, this hypothesis is equivalent to the facts that $(\mathbf m,0)$ is the only global minimum of $W|_{E(\mathbf m)}$ and $\mathbf j^W(\mathbf m,0)\cap\mathbf j^W(\mathbf m',0)=\emptyset$ which is what we use in practice.
\hop

Recall the notation \eqref{exph} and let us extend our notions of asymptotic expansions to smooth functions that are not necessarily symbols. Throughout the paper, for $d'\in \N^*$, $\Omega \subseteq \R^{d'}$ and $a \in \mathcal C^{\infty}(\Omega)$ a function depending on $h$ and such that for all $\beta\in \N^{d'}$ we have $\partial^\beta a =O_{L^\infty}(1)$, we will denote $a \sim_h \sum_{j \geq 0}h^j a_j$, where  $(a_j)_{j \geq 0}\subset \mathcal C^{\infty}(\Omega)$ are allowed to depend on $h$, provided that for all $\beta \in \N^{d'}$ and $N \in \N$, there exists $C_{\beta, N}$ such that 
$$\Big\|\partial^\beta \Big(a-\sum_{j = 0}^{N-1}h^j a_j\Big)\Big\|_{\infty, \Omega}\leq C_{\beta, N} h^N.$$
It implies in particular that $\partial^\beta a_j=O_{L^\infty}(1)$.
We will also say that $a \in \mathcal C^{\infty}(\Omega)$ admits a classical expansion on $\Omega$ and we will denote $a \sim \sum_{j \geq 0}h^j a_j$ if $a \sim_h \sum_{j \geq 0}h^j a_j$ and the $(a_j)$ are independent of $h$.
From now on, the letter $r$ will denote a small universal positive constant whose value may decrease as we progress in this paper (one can think of $r$ as $1/C$).
For $x\in \R^d$, we denote $B_0(x,r)=B(x,r)\times B(0,r)\subseteq \R^{2d}$.
%From now on, we write $X=(x,v) \in \R^{2d}$ and $X^*=(\xi, \eta)\in \R^{2d}$. 
%\color{purple}MES QUASIMODES AURONT LA MEME FORME QUE DANS MON MEMOIRE AVEC $\ell$ en haut de l'inté au lieu de la forme lin.
%\color{black}
We essentially follow the quasimodal construction from \cite{BonyLPMichel}.
We will also denote 
$$H_W=h^{-1}X_0^h=\begin{pmatrix}
v \\
-\partial_x V 
\end{pmatrix}.$$

Let %$\mathbf s \in \mathcal U^{(1)}$, 
$\mathbf m  \in \mathcal U^{(0)}\backslash\{\underline{\mathbf m}\}$; for each $\mathbf s \in \mathbf j(\mathbf m)$ we introduce a function 
%$$w^{\mathbf s,h}(x,v)=A_h^{-1}\int_0^{\ell^{\mathbf s,h}(x,v)}\zeta(s)\e^{-s^2/2h}\D s.$$
%Here 
$\ell^{\mathbf s,h}$ that will appear in our quasimodes.
Note that thanks to item \ref{jvide2} from Hypothesis \ref{jvide}, each $\ell^{\mathbf s,h}$ corresponds to a unique $\mathbf m  \in \mathcal U^{(0)}\backslash\{\underline{\mathbf m}\}$.
Our goal will be to find some functions $\ell^{\mathbf s,h}$ such that our quasimodes are the most accurate possible.
In order to begin the computations that will yield the equations that the function $\ell^{\mathbf s,h}$ should satisfy, we will for the moment assume that it satisfies the following:
\hop

\begin{minipage}{0.02\linewidth}
\begin{align}\label{hypol}
\,
\end{align}
\end{minipage}
\begin{minipage}{0.93\linewidth}
\begin{enumerate}[label=\alph*)]
\item $\ell^{\mathbf s,h}$ is a smooth real valued function on $\R^{2d}$ whose support is contained in $B_0(\mathbf s,3r)$ % where $\delta(r)>0$ can be chosen later 
\label{hypol1}
\item $\ell^{\mathbf s,h}$ admits a classical expansion $\ell^{\mathbf s,h}(x,v)\sim \sum h^j\ell^{\mathbf s}_j(x,v)$ on $B_0(\mathbf s,2r)$ 
\item $\ell^{\mathbf s}_0$ vanishes at $(\mathbf s,0)$ \label{hypoln-2}
\item $(\mathbf s,0)$ is a local minimum of the function $W+(\ell_0^{\mathbf s})^2/2$ which is non degenerate \label{hypoln-1}%$\widetilde W$ that we will introduce in \eqref{ytilde}.
\item the functions $\theta_{\mathbf m,h}$ (which depends on $\ell^{\mathbf s,h}$) and $\chi_{\mathbf m}$ that we will introduce in \eqref{thetainte}-\eqref{chim} are such that $\theta_{\mathbf m,h}$ is smooth on a neighborhood of supp $\chi_{\mathbf m}$. \label{hypoln}
\end{enumerate}
\end{minipage}
\hop
Once we will have found the desired function $\ell^{\mathbf s,h}$, we will see in Proposition \ref{lexist} that these assumptions are actually satisfied.
Denote $\zeta \in \mathcal C^{\infty}_c(\R, [0,1])$ an even cut-off function supported in $[-\gamma,\gamma]$ that is equal to $1$ on $[-\gamma/2,\gamma/2]$ where $\gamma>0$ is a parameter to be fixed later and 
\begin{align}\label{approxa}
A_h=\frac12 \int_\R \zeta(s) \e^{-\frac{s^2}{2h}}\D s=\int_0^{\gamma} \zeta(s) \e^{-\frac{s^2}{2h}}\D s=\frac{\sqrt{\pi h}}{\sqrt 2}(1+O(\e^{-\alpha/h}))\qquad \text{for some }\alpha>0.
\end{align}
We now define for each $\mathbf m  \in \mathcal U^{(0)}\backslash\{\underline{\mathbf m}\}$ a function $\theta_{\mathbf m,h}$ as follows: if $(x,v) \in B_0(\mathbf s,r)\cap \{|\ell^{\mathbf s,h}|\leq 2\gamma\}$ for some $\mathbf s \in \mathbf j(\mathbf m)$, 
\begin{align}\label{thetainte}
\theta_{\mathbf m,h}(x,v)=\frac12 \Big(1+A_h^{-1}\int_0^{\ell^{\mathbf s,h}(x,v)}\zeta(s)\e^{-s^2/2h}\D s\Big)
\end{align} 
whereas we set 
\begin{align}\label{theta1}
\theta_{\mathbf m,h}=1 \quad  \text{on } \Big(E(\mathbf m)+B(0, \varepsilon) \Big)\backslash \Big(\bigsqcup_{\mathbf s\in\mathbf j(\mathbf m)} \big(B_0(\mathbf s,r)\cap \{|\ell^{\mathbf s,h}|\leq 2\gamma\}\big) \Big)
\end{align}
with $\varepsilon(r)>0$ to be fixed later and 
\begin{align}\label{theta0}
\theta_{\mathbf m,h}=0 \quad  \text{everywhere else.}
\end{align}
Note that $\theta_{\mathbf m,h}$ takes values in $[0,1]$.
Denote $\Omega$ the CC of $\{W\leq \boldsymbol \sigma(\mathbf m)\}$ containing $\mathbf m$.
The CCs of $\{W\leq \boldsymbol \sigma(\mathbf m)\}$ are separated so for $\varepsilon>0$ small enough, there exists $\tilde \varepsilon>0$ such that 
$$\min\,\big\{W(x,v)\, ; \, \D \big((x,v),\Omega\big)=\varepsilon\big\}=\boldsymbol \sigma(\mathbf m)+2 \tilde \varepsilon.$$
Thus the distance between $\{W \leq \boldsymbol \sigma(\mathbf m)+\tilde \varepsilon\}\cap \big(\Omega+B(0,\varepsilon)\big)$ and $\partial \big(\Omega+B(0,\varepsilon)\big)$ is positive and we can consider a cut-off function 
\begin{align}\label{chim}
\chi_{\mathbf m}\in \mathcal C^{\infty}_c(\R^{2d},[0,1])
\end{align}
such that 
$$\chi_{\mathbf m} =1 \text{ on } \{W \leq \boldsymbol \sigma(\mathbf m)+\tilde \varepsilon\}\cap \big(\Omega+B(0,\varepsilon)\big)$$
 and 
$$\mathrm{supp} \,\chi_{\mathbf m} \subset  \big(\Omega+B(0,\varepsilon)\big).$$
To sum up, we have the following picture:

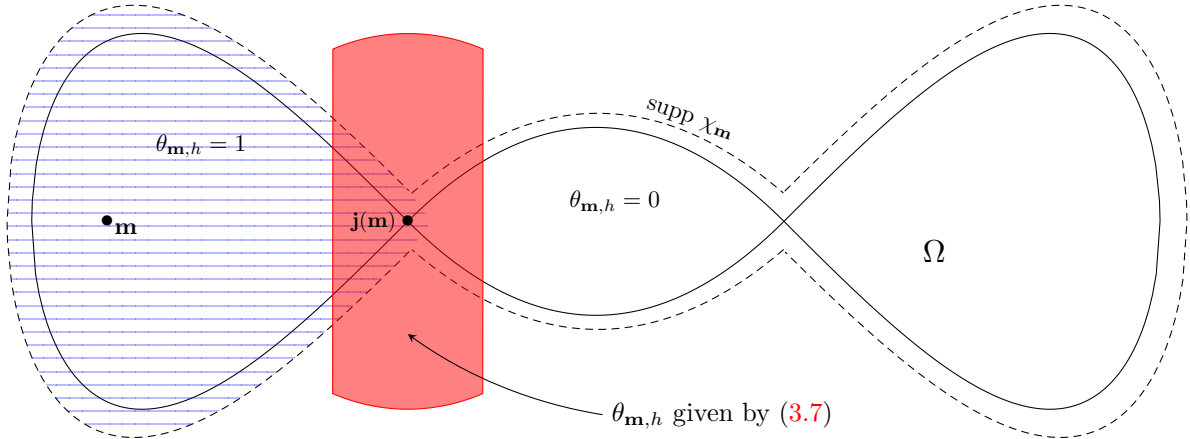
\begin{center}
\begin{tikzpicture}[scale=5]
%\fill[pattern color=blue,pattern=north east lines,opacity=0.6] (0.1,0) -- plot[samples=100,domain=0:-0.9] (\x, {\x*sqrt(1-(\x)^2)-0.1}) -- plot[samples=100,domain=-0.85:-1] (\x-0.1, {\x*sqrt(1-(\x)^2)}) --  plot[samples=100,domain=-1:-0.85] (\x-0.1, {-\x*sqrt(1-(\x)^2)}) -- plot[samples=100,domain=-0.9:0] (\x, {-\x*sqrt(1-(\x)^2)+0.1}) -- cycle ;

\draw plot[samples=100,domain=1:2] (\x, {(\x-1)*sqrt(1-((\x-1))^2)}) 
		-- plot[samples=100,domain=2:1] (\x, {-(\x-1)*sqrt(1-((\x-1))^2)}) -- cycle ;
\draw plot[samples=100,domain=0:1] (\x, {\x*(1-\x)}) -- plot[samples=100,domain=1:0] (\x, {-\x*(1-\x)}) -- cycle ;
\draw plot[samples=100,domain=-1:0] (\x, {\x*sqrt(1-(\x)^2)}) -- plot[samples=100,domain=0:-1] (\x, {-\x*sqrt(1-(\x)^2)}) -- cycle ;

\begin{scope} [xshift=-6.5,scale=1.15]
\draw[densely dashed] plot[samples=100,domain=1.07:2] (\x, {(\x-1)*sqrt(1-((\x-1))^2)}) 
		-- plot[samples=100,domain=2:1.07] (\x, {-(\x-1)*sqrt(1-((\x-1))^2)}) ;
\end{scope}
\begin{scope} [xshift=-2,scale=1.15]
\draw[densely dashed] plot[samples=100,domain=0.07:0.93] (\x, {\x*(1-\x)});
\draw[densely dashed] plot[samples=100,domain=0.93:0.07] (\x, {-\x*(1-\x)});
\end{scope}
\begin{scope} [xshift=2.4,scale=1.15]
\draw[densely dashed] plot[samples=100,domain=-0.07:-1] (\x, {\x*sqrt(1-(\x)^2)}) -- plot[samples=100,domain=-1:-0.07] (\x, {-\x*sqrt(1-(\x)^2)});
\fill[pattern={Lines[angle=45,distance=5pt]},pattern color=blue,opacity=0.7] (-0.02,0) -- plot[samples=100,domain=-0.07:-1] (\x, {\x*sqrt(1-(\x)^2)}) -- plot[samples=100,domain=-1:-0.07] (\x, {-\x*sqrt(1-(\x)^2)}) -- cycle;
\end{scope}

%\draw[densely dashed] plot[samples=100,domain=1:1.9] (\x, {(\x-1)*sqrt(1-((\x-1))^2)+0.1})
%		-- plot[samples=100,domain=1.85:2] (\x+0.1, {(\x-1)*sqrt(1-((\x-1))^2)}) -- plot[samples=100,domain=2:1.85] (\x+0.1, {-(\x-1)*sqrt(1-((\x-1))^2)}) -- plot[samples=100,domain=1.9:1] (\x, {-(\x-1)*sqrt(1-((\x-1))^2)-0.1}) ;
%\draw[densely dashed] plot[samples=100,domain=0:1] (\x, {\x*(1-\x)+0.1});
%\draw[densely dashed] plot[samples=100,domain=1:0] (\x, {-\x*(1-\x)-0.1});
%\draw[densely dashed] plot[samples=100,domain=0:-0.9] (\x, {\x*sqrt(1-(\x)^2)-0.1}) -- plot[samples=100,domain=-0.85:-1] (\x-0.1, {\x*sqrt(1-(\x)^2)}) --  plot[samples=100,domain=-1:-0.85] (\x-0.1, {-\x*sqrt(1-(\x)^2)}) -- plot[samples=100,domain=-0.9:0] (\x, {-\x*sqrt(1-(\x)^2)+0.1}) ;
%\draw[densely dashed,blue]  (0,0.1) -- (0.1,0) -- (0,-0.1);

\coordinate (ne) at (0.5*0.4,{0.5*sqrt(1-(0.4)^2)});
\coordinate (nw) at (-0.5*0.4,{0.5*sqrt(1-(0.4)^2)});
\coordinate (se) at (0.5*0.4,{-0.5*sqrt(1-(0.4)^2)});
\coordinate (sw) at (-0.5*0.4,{-0.5*sqrt(1-(0.4)^2)});
\filldraw[red,fill opacity=0.5] (ne) arc (1.16 r:{(pi-1.16) r}:0.5) -- (sw) arc ({-(pi-1.16) r}:-1.16 r:0.5) -- cycle;

\draw (0,0) node {$\bullet$};
\draw (-0.09,0) node[scale=0.8] {$\mathbf j(\mathbf m)$};
\draw (-0.8,0) node {$\bullet$};
\draw (-0.75,-0.02) node[scale=0.9] {$\mathbf m$};
\draw (-0.55,0.2) node[scale=0.9] {$\theta_{\mathbf m,h}=1$};
\draw (0.55,0.05) node[scale=0.9] {$\theta_{\mathbf m,h}=0$};
\draw (1.4,-0.08) node[scale=1.2] {$\Omega$};
\draw (0.75,0.27) node[scale=0.9,rotate=-19] {supp $\chi_{\mathbf m}$};

\draw[>=stealth,<-] (0,-0.3) arc ({(-0.7*pi) r}:{(-0.55*pi) r}:1.2) node[right] {$\theta_{\mathbf m,h}$ given by \eqref{thetainte}};

\end{tikzpicture}
\end{center}
We also denote 
$$W_{\mathbf m}(x,v)=W(x,v)-V(\mathbf m)/2$$ 
and it is clear that on the support of $\nabla \chi_{\mathbf m}$, we have
%$W_{\mathbf m}=
$$W_{\mathbf m}\geq S(\mathbf m)+\tilde \varepsilon.$$
%\end{rema}
Our quasimodes will be the $L^2$-renormalizations of the functions 
\begin{align}\label{quasim}
f_{\mathbf m,h}(x,v)=\chi_{\mathbf m}(x,v)\theta_{\mathbf m,h}(x,v)\e^{-W_{\mathbf m}(x,v)/h}\quad ; \quad  \mathbf m \in \mathcal U^{(0)}\backslash\{\underline{\mathbf m}\}%\in \mathcal C^{\infty}_c(\R^{2d})
\end{align}
and for $\mathbf m=\underline{\mathbf m}$, 
$$f_{\underline{\mathbf m},h}(x,v)=\e^{-W_{\underline{\mathbf m}}(x,v)/h}\in \mathrm{Ker}\, P_h.$$
Note that these functions belong to $\mathcal C^\infty_c(\R^{2d})$ thanks to our assumption on the $(\ell^{\mathbf s,h})_{\mathbf s\in \mathbf j(\mathbf m)}$ and that for $\mathbf m\neq \underline{\mathbf m}$, we have 
\begin{align}\label{suppf}
\mathrm{supp}\, f_{\mathbf m,h} \subseteq E(\mathbf m)+B(0,\varepsilon')
\end{align}
where $\varepsilon'=\max(\varepsilon,r)$.

\subsection{Action of the operator $P_h$}

Let us fix $\mathbf m \in \mathcal U^{(0)}\backslash\{\underline{\mathbf m}\}$. %and $(\mathbf s,0) \in \mathbf j(\mathbf m)$.
 %\color{green}($\ell^m=\ell$, si $\ell^s$ on précisera).
%\color{black}
%We will also suppose that $\theta_{\mathbf m,h}$ is smooth on a neighborhood of $\chi_{\mathbf m}$ so that the quasimode $f_{\mathbf m,h}$ belongs to $\mathcal C^\infty_c(\R^{2d})$.
%Once we will have chosen the functions $(\ell^{\mathbf s,h})_{(\mathbf s,0) \in \mathbf j(\mathbf m)}$, we will see in Lemma \ref{llisse} that this assumption is satisfied.
We will denote
\begin{align}\label{ytilde}
\widetilde W_{\mathbf m,h}=W_{\mathbf m}+\sum_{\mathbf s \in \mathbf j(\mathbf m)}(\ell^{\mathbf s,h})^2/2%\qquad \qquad \text{dépend de h à travers }\ell^m
\end{align}
and
\begin{align}\label{psi}
\psi^{\mathbf m,h}(x,v,v')=\int_0^1\partial_v \widetilde W_{\mathbf m,h}(x, v'+t(v-v')) \D t.
\end{align}
%It is clear that $\nabla \widetilde W_{\mathbf m,h}$ is compactly supported so in particular $\psi_{\mathbf m,h}$ is a bounded function.
%Thanks to \eqref{gradyborne} and the assumptions on $\ell$, it is clear that there exists $\tau>0$ such that for $\|\psi_{\mathbf m,h}\|_\infty < \tau$.
%$$\Sigma_\tau=\{z\in \C \, ; \, |\mathrm{Im}z|< \tau\}$$
%$$S^0_\tau(\langle (x,v,\xi,\eta) \rangle^N)=\{a\in S^0(\langle (x,v,\xi,\eta) \rangle^N) ; \forall \beta \in \N^{2d}, \partial^\beta_{(x,v)} a(x,v,\cdot,\cdot)\text{ is analytic in }\Sigma_\tau^{2d} \text{ and } a(x,v,\xi,\eta)\leq \langle (x,v,\xi,\eta) \rangle^N \text{ on } \R^2d\times \Sigma_\tau^{2d}\}$$

\begin{rema}\label{defg}
Using Hypothesis \ref{hypom}, it is easy to see that % if we denote
%$$g(x,v,\eta)=(-i\, {}^t\eta+\, {}^tv/2)\# m^h $$
$b_h^*\text{Op}_h(M^h)=\text{Op}_h(g^h)$, with 
$$g^h=(-i\, {}^t\eta+\, {}^tv/2)M^h-\frac{h}{2}( {}^t\nabla_v-\frac i2  {}^t\nabla_\eta)M^h\in \mathcal M_{1,d}\big(S^0_\tau(\langle (v,\eta) \rangle^{-1})\big)$$
where 
$${}^t\nabla_vM^h=\Big(\sum_{k=1}^d\partial_{v_k}m_{k,j}\Big)_{1\leq j \leq d}$$
and ${}^t\nabla_\eta$ is defined similarly.
\end{rema}

\begin{prop}\label{phf}
Let $f_{\mathbf m,h}$ be the quasimode defined in \eqref{quasim}.
With the notations introduced in \eqref{approxa} and \eqref{ytilde}, one has
$$P_hf_{\mathbf m,h}=\frac h2 A_h^{-1}\om^{\mathbf m,h}\,\e^{\frac{-\widetilde W_{\mathbf m,h}}{h}}\1_{\mathbf j^W(\mathbf m)+B_0(0,2r)}+O_{L^2}\Big(h^\infty\e^{-\frac{S(\mathbf m)}{h}}\Big)$$
where $\om^{\mathbf m,h}$ is a function bounded uniformly in $h$ and defined on $\mathbf j^W(\mathbf m)+B_0(0,2r)$ by
\begin{align*}
\om^{\mathbf m,h}=&\sum_{ \mathbf s \in \mathbf j(\mathbf m)}\Big( H_W \cdot \nabla \ell^{\mathbf s,h}+I^{\mathbf s,h}\Big)
\end{align*}
with $I^{\mathbf s,h}(x,v)$ given for $(x,v)\in \mathbf j^W(\mathbf m)+B_0(0,2r)$ by the oscillatory integral
$$(2\pi h)^{-d}\int_{\R^d} \int_{|v'|\leq 2r} \e^{\frac ih \eta \cdot (v-v')}g^h\Big(x,\frac{v+v'}{2},\eta+i\psi^{\mathbf m,h}(x,v,v')\Big) \partial_v\ell^{\mathbf s,h}(x,v')\,\D v' \D \eta.$$
%is bounded uniformly in $h$ outside $B_0(\mathbf j(\mathbf m),r)$.
\end{prop}

\begin{proof}
In order to lighten the notations, we will drop some of the exponents and indexes $\mathbf m$, $\mathbf s$ and $h$ in the proof.
By \eqref{hypol}, we have on the support of $\chi$ that $\theta$ is smooth and
$$\nabla \theta=\frac{A_h^{-1}}{2}\sum_{\mathbf s\in \mathbf j(\mathbf m)}\e^{-(\ell^{\mathbf s})^2/2h}\zeta(\ell^{\mathbf s})\nabla \ell^{\mathbf s} \,\1_{B_0(\mathbf s,r)}.$$
Here we have to put the indicator function because $\zeta(\ell)\nabla \ell$ might have some support in %without it, it would not be clear that $\nabla \theta$ has no support in 
$B_0(\mathbf s,3r)\backslash B_0(\mathbf s,r)$.
%\color{blue}Rq: normalement sur un voisinage de support de $\chi  \cap \partial B_0(\mathbf j(\mathbf m),r)$ on a $|\ell| \geq \gamma$ donc $\nabla \theta$ est déja nul donc tout est bien C$^\infty$.
%\color{black}
We can then begin by computing
\begin{align}\label{x0f}
X_0^hf&=hH_W \cdot \nabla f \nonumber\\
			&=hH_W \cdot \nabla \theta \,\chi \e^{-W_{\mathbf m}/h}+hH_W \cdot \nabla \chi \,\theta \e^{-W_{\mathbf m}/h}\\
			&=\frac h2 A_h^{-1}\chi \e^{-\widetilde W/h} \sum_{\mathbf s\in \mathbf j(\mathbf m)} \zeta(\ell^{\mathbf s})  H_W \cdot \nabla \ell^{\mathbf s} \,\1_{B_0(\mathbf s,r)} +O\Big(h\e^{-\frac{S(\mathbf m)+\tilde \varepsilon}{h}}\Big). \nonumber
\end{align}
since $ W_{\mathbf m}\geq S(\mathbf m)+\tilde\varepsilon$ on the support of $\nabla \chi$.
Now we can use Remark \ref{defg} to write
\begin{align}\label{qf}
\qquad \quad Q_h(f)&=h\mathrm{Op}_h(g)\big((\partial_v \theta) \chi \e^{-W_{\mathbf m}/h}+(\partial_v\chi)\theta \e^{-W_{\mathbf m}/h}\big)\\
			&=\frac h2 A_h^{-1}\sum_{\mathbf s\in \mathbf j(\mathbf m)}\mathrm{Op}_h(g)\Big(\zeta(\ell^{\mathbf s})\chi \e^{-\widetilde W/h} \partial_v \ell^{\mathbf s} \,\1_{B_0(\mathbf s,r)} \Big)+O\Big(h\e^{-\frac{S(\mathbf m)+\tilde \varepsilon}{h}}\Big) \nonumber
\end{align}
since $g \in S(\langle (v,\eta) \rangle^{-1})$ and thus $\mathrm{Op}_h(g)$ is bounded uniformly in $h$.
But since $g$ does not depend on $\xi$,
% and using the definition of an oscillatory integral \color{blue} (or Proposition $\ref{limosci}$)
we have for $\mathbf s\in \mathbf j(\mathbf m)$
\begin{align}\label{pasxi}
(2 \pi h)^{d}\mathrm{Op}_h(g)\Big(\zeta(\ell)\chi \e^{-\widetilde W/h} &\partial_v \ell \,\1_{B_0(\mathbf s,r)} \Big)(x,v)=\int_{\R^{d}} \int_{|v'|\leq r} \e^{\frac ih \eta \cdot (v-v')}g\Big(x,\frac{v+v'}{2},\eta\Big) \nonumber\\
& \qquad \qquad \qquad \quad \times\chi(x,v') \zeta\big(\ell(x,v')\big) \e^{-\widetilde W(x,v')/h}\partial_v\ell(x,v')\;\D v' \D \eta \,\1_{B(\mathbf s,r)}(x) .
\end{align}
%\color{green} La limite au dessus pas utile là!
Let us now treat separately the cases $|v|\geq 2r$ and $|v|<2r$ .\\
When $|v|\geq 2r$, we have $|v-v'|\geq r$ so we can apply the non stationnary phase to the integral in $\eta$ to get that for all $x \in B(\mathbf s,r)$ and $N\geq 1$, there exists $C_N>0$ such that
\begin{align*}
\bigg|\int_{\R^{d}} %&
\int_{|v'|\leq r} \e^{\frac ih \eta \cdot (v-v')}g\Big(x,\frac{v+v'}{2},\eta\Big) \chi(x,v') \zeta\big(\ell(x,v')\big) \e^{-\widetilde W(x,v')/h}\partial_v\ell(x,v')\;\D v' \D \eta\bigg| \leq C_N h^{N} |v|^{-N} \e^{-\frac{S(\mathbf m)}{h}}
	%&\qquad \leq C  h^N  \sup_{|\beta| \leq N+d}\int_{\R^{d}}\int_{|v'|\leq r}  |v-v'|^{-N} \Big|\partial_\eta^\beta g\Big(x,\frac{v+v'}{2},\eta\Big)\Big|  \e^{-\widetilde W(x,v')/h}\; \D v' \D \eta \\
	%&\qquad \leq C_N h^{N} |v|^{-N} \e^{-\frac{S(\mathbf m)}{h}}
\end{align*}
where we used item \ref{hypoln-1} from \eqref{hypol}, the fact that $W_{\mathbf m}(\mathbf s,0)+\ell_0^2(\mathbf s,0)/2=S(\mathbf m)$ and the estimate $|v-v'|\geq |v|/2$.
Hence we have shown that %for all $x\in \R^d$ and $|v|\geq 2r$, 
\begin{align}\label{vgrand}
\qquad \quad Q_hf\,\1_{\{|v|\geq 2r\}}=O\Big(h^\infty \e^{-\frac{S(\mathbf m)}{h}}\Big) \quad \text{and }\quad P_hf\,\1_{\{|v|\geq 2r\}}=O\Big(h^\infty \e^{-\frac{S(\mathbf m)}{h}}\Big).
\end{align}
Now for the case $|v|<2r$, let us denote $J_1^{\mathbf s}(x,v)$ the RHS of \eqref{pasxi}. Proceeding as in \cite{Nakamura} in order to take the $\e^{-\widetilde W(x,v')/h}$ in front of the oscillatory integral,%(which is justified by Proposition $\ref{limosci}$)
we get that for any %$(\mathbf s,0)\in \mathbf j(\mathbf m)$ and 
$x\in B(\mathbf s,r)$,
\begin{align}\label{expdevant}
J_1^{\mathbf s}(x,v)=\e^{-\widetilde W(x,v)/h}J_2^{\mathbf s}(x,v)
\end{align}
where
\begin{align}
J_2^{\mathbf s}(x,v)=\int_{\R^d} \int_{|v'|\leq r} \e^{\frac ih \big(\eta-i\psi(x,v,v')\big) \cdot \big(v-v'\big)}g\Big(x,\frac{v+v'}{2},\eta\Big) \chi(x,v') \zeta\big(\ell(x,v')\big) \partial_v\ell(x,v')\;\D v' \D \eta \,\1_{B(\mathbf s,r)}(x)
\end{align}
and $\psi$ is the function defined in \eqref{psi}.
%Hence according to Proposition $\ref{limosci}$ we get 
%%\begin{align*}%\label{eta}
%$$\int_{\R}  \e^{\frac ih\big(\eta_j-i[\psi(x,v,v')]_{j}\big) (v_j-v'_j)}  g\Big(x,\frac{v+v'}{2},\eta\Big) \D \eta_j= %\partial_{(x,v)}^\beta g\Big(\frac{x+x'}{2},\frac{v+v'}{2},\eta\Big)\D \xi_j=\\
%-\int_{\R}  \e^{\frac ih \eta_j (v_j-v'_j)}  g\Big(x,\frac{v+v'}{2},\eta+i[\psi(x,v,v')]_{j}e_j\Big)   \D \eta_j  $$
%%\end{align*}
%which combined with Proposition $\ref{fubosci}$ yields
%\begin{align*}
%&\int_{\R^d} \int_{|v'|\leq r} \e^{\frac ih \big(\eta-i\psi(x,v,v')\big) \cdot \big(v-v'\big)}g\Big(x,\frac{v+v'}{2},\eta\Big)\chi(x,v') \zeta\big(\ell(x,v')\big) \partial_v\ell(x,v')\;\D v' \D \eta=\\
%&\color{red} (-1)^d \color{black}\int_{\R^d} \int_{|v'|\leq r} \e^{\frac ih \eta \cdot (v-v')}g\Big(x,\frac{v+v'}{2},\eta +i\psi(x,v,v')\Big)\chi(x,v') \zeta\big(\ell(x,v')\big) \partial_v\ell(x,v')\;\D v' \D \eta.
%\end{align*}
For $K\subset \{1,\dots, d\}$ and $z\in \C^d$, denote  $z_K=(z_j)_{j\in K}$.
We also denote for $d'\in \N$ and $1\leq j\leq d'$
\begin{align}\label{ej}
e_j=(\delta_{k,j})_{1\leq k \leq d'}\in \N^{d'}
\end{align}
the elements of the canonical basis of $\C^{d'}$.
Now notice that $\psi$ is a smooth function and that using the expansion of $\ell$ and \eqref{ytilde}, we get on $B_0(\mathbf s,2r)\times \{|v'|\leq 2r\}$,
$$\psi(x,v,v')=\frac{v+v'}{4}+\int_0^1\big(\ell_0\partial_v \ell_0\big)(x,v'+t(v-v'))\D t+O(h).$$
In particular, we can choose $r$ small enough so that $|\psi|<\tau$ on $B_0(\mathbf s,2r)\times \{|v'|\leq 2r\}$.
Besides, since $g \in S^0_\tau(\langle(v,\eta) \rangle^{-1})$, we have for all $K\subset \{1,\dots, d\}$ and $k\in \{1,\dots, d\}\backslash K$ that the symbol
$$\eta_k \mapsto g\Big(x,\frac{v+v'}{2},\eta+i\sum_{j\in K}[\psi(x,v,v')]_{j}e_j\Big) $$
has an analytic continuation to $\{|\eta_k|<\tau\}$ for any $x\in B(\mathbf s,r)$, $v$, $v'\in B(0,2r)$ and $\eta \in \R^d$.
Hence, one can use the Cauchy formula which combined with the decay of $g$ yields
\begin{align}
\int_{\R}  \e^{\frac ih\big(\eta_k-i[\psi(x,v,v')]_k\big) (v_k-v_k')}  g\Big(x,\frac{v+v'}{2},\eta+i\sum_{j\in K}[\psi(x,v,v')]_{j}e_j\Big)\, \D \eta_k=&\\
\int_{\R}  \e^{\frac ih \eta_k (v_k-v_k')}  g\Big(x,\frac{v+v'}{2},\eta+&i\sum_{j\in K\cup \{k\}}[\psi(x,v,v')]_{j}e_j\Big)\,   \D \eta_k.  
\end{align}
Applying this successively for each component of $\eta$ on the integrals in $J_2^{\mathbf s}$ finally gives $J_2^{\mathbf s}=J_3^{\mathbf s}$ where
%is analytic and bounded together with its derivatives on $\Sigma_\tau$; %in $S^0_\tau(1)$ and tends to 0 when $|\mathrm{Re }\,\eta_k|\to\infty$ in $\Sigma_\tau$. 
%\color{blue}(IPP+holom sous inté+cv dom pour la limite)
%\color{black}
\begin{align}
J_3^{\mathbf s}(x,v)=\int_{\R^d} \int_{|v'|\leq r} \e^{\frac ih \eta \cdot (v-v')}g\Big(x,\frac{v+v'}{2},\eta+i\psi(x,v,v')\Big)\chi(x,v') \zeta\big(\ell(x,v')\big)\partial_v\ell(x,v')\;\D v' \D \eta \,\1_{B(\mathbf s,r)}(x).
\end{align}
Combined with %(\ref{qf}), 
\eqref{pasxi} and \eqref{expdevant}, this yields for $|v|<2r$
\begin{align}\label{opgtruc}
(2 \pi h)^{d}&\mathrm{Op}_h(g)\Big(\zeta(\ell)\chi \e^{-\widetilde W/h} \partial_v \ell \,\1_{B_0(\mathbf s,r)} \Big)(x,v)=\e^{-\widetilde W(x,v)/h}J_3^{\mathbf s}(x,v).
\end{align}
Therefore, setting on $\mathbf j^W(\mathbf m)+B_0(0,2r)$
\begin{align*}
\tilde \om=\sum_{\mathbf s\in \mathbf j(\mathbf m)}\Big(\chi \zeta(\ell^{\mathbf s})  H_W \cdot \nabla \ell^{\mathbf s} \,\1_{B_0(\mathbf s,r)}+(2\pi h)^{-d}J_3^{\mathbf s}(x,v)\Big),
\end{align*}
we have according to \eqref{x0f}, \eqref{qf}, \eqref{vgrand} and \eqref{opgtruc}
$$P_hf=\frac h2 A_h^{-1}\tilde \om\,\e^{-\widetilde W/h}\1_{\mathbf j^W(\mathbf m)+B_0(0,2r)}+O\Big(h^\infty\e^{-\frac{S(\mathbf m)}{h}}\Big).$$
Hence it is sufficient to check that on $\mathbf j^W(\mathbf m)+B_0(0,2r)$
\begin{align*}
(\tilde\om-\om)\e^{-\widetilde W/h}=O\Big(h^\infty \e^{-\frac{S(\mathbf m)}{h}}\Big).
\end{align*}
This can be done easily using again the non stationary phase on an $h$-independent neighborhood of $(\mathbf s,0)$ on which $\chi\zeta(\ell)-1$ vanishes since item \ref{hypoln-1} from \eqref{hypol} implies that $\e^{-\widetilde W/h}=O(\e^{-(S(\mathbf m)+\delta)/h})$ outside of this neighborhood for some $\delta>0$.
\end{proof}

\begin{rema}\label{Ph*f}
Since $P_h^*=-X_0^h+Q_h$, it is clear from  the previous proof that 
$$P_h^*f_{\mathbf m,h}=\frac h2 A_h^{-1}\overset{*}{\om}{}^{\mathbf m,h}\,\e^{\frac{-\widetilde W_{\mathbf m,h}}{h}}\1_{\mathbf j^W(\mathbf m)+B_0(0,2r)}+O_{L^2}\Big(h^\infty\e^{-\frac{S(\mathbf m)}{h}}\Big)$$
with 
$$\overset{*}{\om}{}^{\mathbf m,h}=\sum_{ \mathbf s \in \mathbf j(\mathbf m)}\Big(- H_W \cdot \nabla \ell^{\mathbf s,h}+I^{\mathbf s,h}\Big).$$
\end{rema}

\section{Equations on $\ell^{\mathbf s,h}$}\label{sectionequations}
\noindent
From now on, we also fix $\mathbf s\in \mathbf j(\mathbf m)$.

\begin{lem}\label{expom}
%We have $\om(x,v)=O(h^{\infty})$ for $(x,v)\notin B_0(\mathbf s,r)$. 
The function $\om^{\mathbf m,h}$ admits the classical expansion $\om^{\mathbf m,h} \sim \sum_{j\geq 0}h^j\om^{\mathbf m}_j$ on $B_0(\mathbf s,2r)$
where 
\begin{align}\label{om0}
\om^{\mathbf m}_0=H_W \cdot \nabla \ell^{\mathbf s}_0+M_0\Big(x,v,i\big(\frac v2 +\ell^{\mathbf s}_0 \, \partial_v \ell^{\mathbf s}_0\big)\Big)\big(v+\ell^{\mathbf s}_0 \partial_v \ell^{\mathbf s}_0\big)\cdot \partial_v\ell^{\mathbf s}_0
\end{align}
and for $j\geq 1$, 
\begin{align}\label{omj}
\qquad \om^{\mathbf m}_{j}=H_W & \cdot \nabla \ell^{\mathbf s}_j+M_0\Big(x,v,i\big(\frac v2 +\ell^{\mathbf s}_0 \, \partial_v \ell^{\mathbf s}_0\big)\Big)( v+2\ell^{\mathbf s}_0 \partial_v \ell^{\mathbf s}_0)\cdot \partial_v \ell^{\mathbf s}_j \\
			&+i\, \ell^{\mathbf s}_0 \big({}^tv+\ell^{\mathbf s}_0 \, {}^t( \partial_v \ell^{\mathbf s}_0)\big) D_\eta M_0\big(x,v,i(v/2+\ell^{\mathbf s}_0 \partial_v \ell^{\mathbf s}_0)\big)\big( \partial_v \ell^{\mathbf s}_j\big) \, \partial_v \ell^{\mathbf s}_0 \nonumber\\
			&+M_0\Big(x,v,i\big(\frac v2 +\ell^{\mathbf s}_0 \, \partial_v \ell^{\mathbf s}_0\big)\Big) \partial_v \ell^{\mathbf s}_0 \cdot \partial_v \ell^{\mathbf s}_0 \, \ell^{\mathbf s}_j\nonumber\\
			&+i  \big({}^tv+\ell^{\mathbf s}_0 \, {}^t( \partial_v \ell^{\mathbf s}_0)\big) D_\eta M_0\big(x,v,i(v/2+\ell^{\mathbf s}_0 \partial_v \ell^{\mathbf s}_0)\big)\big( \partial_v \ell^{\mathbf s}_0\big) \, \partial_v \ell^{\mathbf s}_0 \, \ell^{\mathbf s}_j \nonumber\\
			&+R_j(\ell^{\mathbf s}_0, \dots, \ell^{\mathbf s}_{j-1})\nonumber
\end{align}
where $R_j : \big( \mathcal C^\infty(B_0(\mathbf s,2r) )\big)^j \to \mathcal C^\infty(B_0(\mathbf s,2r) )$ and $D_\eta$ denotes the partial differential with respect to the variable $\eta$.
\end{lem}

\begin{proof} %
%The first assertion of the lemma comes from the fact that $\mathrm{supp}$ $\ell \subseteq B_0(\mathbf s,r)$. Now
Once again, we drop some of the exponents and indexes $\mathbf m$, $\mathbf s$ and $h$ in the proof.
\sloppy Denote $B_\infty(0,2r)=\{v',\eta \in \R^{2d}\,;\,\max (|v'|,|\eta|)< 2r\}$.
The first terms of $\om_0$ and $\om_j$ are both easily obtained thanks to the expansion of $\ell$ on $B_0(\mathbf s,2r)$.
Hence it remains to get an expansion of $g(x,v/2+v'/2,\eta+i\psi(x,v,v'))$ that we will then be able to combine with the stationnary phase to get an expansion of the whole term $I^{\mathbf s,h}$ of $\om$.
Let us start with an expansion of $\psi$ :
the expansion of $\ell$ yields
$$\partial_v \widetilde W-v/2\sim \sum_{j\geq 0}h^j\sum_{k=0}^j\ell_k\partial_v \ell_{j-k}\quad \text{on } B_0(\mathbf s,2r)$$
so using \eqref{psi}, we get 
\begin{align}\label{exppsi}
\psi \sim \sum_{j\geq 0}h^j\psi_j\quad \text{on } B_0(\mathbf s,2r)\times \{|v'|\leq 2r\}
\end{align}
where 
\begin{align}\label{psi0}
\psi_0(x,v,v')=\frac{v+v'}{4}+\int_0^1\big(\ell_0\partial_v \ell_0\big)(x,v'+t(v-v'))\D t
\end{align}
and for $j \geq 1$, 
\begin{align}\label{psij}
\psi_j(x,v,v')=\int_0^1\sum_{k=0}^j\big(\ell_k\partial_v \ell_{j-k}\big)(x,v'+t(v-v'))\D t.
\end{align}
%\begin{align*}
%&\int_{\R^d} \int_{|v'|\leq r} \e^{\frac ih \eta \cdot (v-v')}g\Big(x,\frac{v+v'}{2},\eta+i\psi(x,v,v')\Big) \partial_v\ell(x,v')\;\D v' \D \eta \sim\\
%&\sum_{j\geq 0} h^j\int_{\R^d} \int_{|v'|\leq r} \e^{\frac ih \eta \cdot (v-v')}g\Big(x,\frac{v+v'}{2},\eta+i\psi(x,v,v')\Big) \partial_v\ell_j(x,v')\;\D v' \D \eta
%\end{align*}
Besides, since $M^h \sim \sum_{n \geq 0}h^nM_n$ in $\mathcal M_{d}\big(S^0_\tau(\langle(v,\eta) \rangle^{-2})\big)$, we deduce thanks to Proposition \ref{expsharp} and Remark \ref{defg} that $g$ also has a classical expansion $g \sim \sum_{n \geq 0}h^ng_n$ in $\mathcal M_{1,d}\big(S^0_\tau(\langle(v,\eta) \rangle^{-1})\big)$, where the $(g_n)$ are given by 
\begin{align}\label{g0}
g_0(x,v,\eta)=\Big(-i\, {}^t\eta+ \, \frac {{}^tv}{2}\Big)M_0(x,v, \eta)
\end{align}
 and 
\begin{align}\label{gn}
g_n(x,v,\eta)=\Big(-i\, {}^t\eta+ \, \frac {{}^tv}{2}\Big)M_n(x,v,\eta)-\frac 12( {}^t\nabla_v-\frac i2  {}^t\nabla_\eta)M_{n-1}(x,v,\eta)
\end{align}
for $n\geq 1.$
According to Corollary \ref{bousin}, we have
$$
g_n\Big(x,\frac{v+v'}{2},\eta +i\psi(x,v,v')\Big)\sim \sum_{j\geq 0} h^j g_{n,j}(x,v,v',\eta) \quad \text{on } B_0(\mathbf s,2r)\times B_\infty (0,2r)
$$
with 
\begin{align}\label{gn0}
g_{n,0}(x,v,v',\eta)=g_n\Big(x,\frac{v+v'}{2},\eta +i\psi_0(x,v,v')\Big)
\end{align}
and for $j \geq 1$
\begin{align}\label{gnj}
g_{n,j}(x,v,v',\eta)=iD_\eta g_n\Big(x,\frac{v+v'}{2},\eta+i\psi_0(x,v,v')\Big) \big(\psi_j(x,v,v')\big)+R^1_j(\ell_0, \dots, \ell_{j-1})
\end{align}
where $R^1_j : \big( \mathcal C^\infty(B_0(\mathbf s,2r) )\big)^j \to \mathcal C^\infty(B_0(\mathbf s,2r) )$.
Using the expansion of $g$ itself and Proposition \ref{expcompo}, we get
$$g\Big(x,\frac{v+v'}{2},\eta +i\psi(x,v,v')\Big)\sim_h\sum_{n\geq 0} h^n g_n\Big(x,\frac{v+v'}{2},\eta +i\psi(x,v,v')\Big)$$
on $B_0(\mathbf s,2r)\times B_\infty (0,2r)$ so we can use Proposition \ref{expinexp} which yields
\begin{align}\label{expg}
g\Big(x,\frac{v+v'}{2},\eta +i\psi(x,v,v')\Big)\sim \sum_{j\geq 0} h^j \sum_{n=0}^j g_{n,j-n}(x,v,v',\eta)
\end{align}
on $B_0(\mathbf s,2r)\times B_\infty (0,2r)$.
Thus, using the expansion \eqref{expg} that we just got, the one of $\partial_v \ell$, and the one for an oscillatory integral given by the stationnary phase (see for instance \cite{Zworski}, Theorem 3.17) as well Proposition \ref{expinexp}, we finally get
\begin{align}\label{expinteosci}
I^{\mathbf s,h}\sim \sum_{j\geq 0} h^j I_j \quad \text{on } B_0(\mathbf s,2r),
\end{align}
where
$$I_j(x,v)=\sum_{n_1+n_2+n_3+n_4=j}\frac{1}{i^{n_1}n_1!}\big(\partial_{v'} \cdot \partial_\eta \big)^{n_1} \Big( g_{n_2,n_3}(x,v,v',\eta) \partial_v \ell_{n_4}(x,v') \Big) \Bigg| \mathop{}_{\substack{v'=v \\  \eta =0}}.$$
%\color{blue} Pour les dérivées, il faudra surement prendre le $\#$ avec $g$ car on ne peut pas appliquer Prop $\ref{linfcinf}$: par exemple pour $\partial_v$ on obtiendrait une "expansion" qui commence à $h^{-1}$. Ca se passe bien pour $h\partial_v$ par exemple. 
%\color{black}
We can already use \eqref{gn0} to deduce the expression of $\omega_0$ by noticing that according to \eqref{psi0}, $\psi_0(x,v,v)=v/2+\ell_0 \partial_v \ell_0$.
For $j \geq 1$, the terms of $I_j$ in which the function $\ell_j$ appears are obviously the one given by $n_4=j$, but also the one given by $n_3=j$ according to \eqref{gnj}.
Indeed, in that case, we have using \eqref{psij} that
\begin{align*}
g_{0,j}(x,v,v,0)=i\ell_0  D_\eta g_0\big(x,v,i(v/2+&\ell_0 \partial_v \ell_0)\big)\big(\partial_v \ell_j\big)\\
				&+iD_\eta g_0\big(x,v,i(v/2+\ell_0 \partial_v \ell_0)\big)\big(\partial_v \ell_0\big) \, \ell_j+R^2_j(\ell_0, \dots, \ell_{j-1})
\end{align*}
where $R^2_j : \big( \mathcal C^\infty(B_0(\mathbf s,2r) )\big)^j \to \mathcal C^\infty(B_0(\mathbf s,2r) )$.
We can now conclude as for any $X \in \R^d$,
\begin{align*}D_\eta g_0\big(x,v,i(v/2+\ell_0 \partial_v \ell_0)\big)(X)=-i\,{}^tXM_0&\big(x,v,i(v/2+\ell_0 \partial_v \ell_0)\big) \\
					&+\big({}^tv+\ell_0 \, {}^t( \partial_v \ell_0)\big) D_\eta M_0\big(x,v,i(v/2+\ell_0 \partial_v \ell_0)\big)(X) 
\end{align*}
according to \eqref{g0}.
\end{proof}
\hip
Denote $(m_{p,q}^n)_{p,q}$ the entries of the matrix $M_n$ from Hypothesis \ref{hypom}. 
Since we have for $X\in \R^d$ 
$$ D_\eta M_0\big(x,v,i(v/2+\ell_0 \partial_v \ell_0)\big)\big( X\big)=\Big( \partial_{\eta}m^0_{p,q}\big(x,v,i(v/2+\ell_0 \partial_v \ell_0)\big)\cdot X\Big)_{1\leq p,q\leq d}$$
we get by putting
\begin{align}\label{U}
\quad U(x,v)=M_0\Big(x,v,&i\big(\frac v2 +\ell_0 \, \partial_v \ell_0\big)\Big) \partial_v\ell_0 \\
&+\sum_{1\leq p,q \leq d}\big( v_p +\ell_0\partial_{v_p} \ell_0 \big) i \partial_{\eta} m^0_{p,q}\Big(x,v,i\big(\frac v2 +\ell_0 \, \partial_v \ell_0\big)\Big)\partial_{v_q}\ell_0 \nonumber
\end{align}
that equation \eqref{omj} reads 
\begin{align*}
\omega_j= \Bigg[H_W+\begin{pmatrix}
				0 \\
				M_0\Big(x,v,i\big(\frac v2 +\ell_0 \, \partial_v \ell_0\big)\Big)( v+\ell_0 \partial_v \ell_0)+\ell_0\, U
				\end{pmatrix}\Bigg] \cdot \nabla \ell_j+ U \cdot \partial_v \ell_0  \, \ell_j+R_j(\ell_0, \dots, \ell_{j-1}).
\end{align*}

\begin{lem}\label{ureelle}
Let $(x,v) \in B_0(\mathbf s,2r)$ and $|v'|<2r$. 
For any $n \in \N$, $\beta \in \N^d$ and $1 \leq p,q \leq d$, we have
%\begin{align*}
$$\partial^\beta_\eta m_{p,q}^n\Big(x,\frac{v+v'}{2},i\psi_0^{\mathbf m}(x,v,v')\Big)  \in i^{|\beta|}\R$$
and
$$ \partial^\beta_\eta g_n\Big(x,\frac{v+v'}{2},i\psi_0^{\mathbf m}(x,v,v')\Big)  \in i^{|\beta|}\R^d.$$
%\end{align*}
In particular, $U$ defined in \eqref{U} sends $B_0(\mathbf s,2r)$ in $\R^d$.
\end{lem}

\begin{proof} %
Since $\ell_0$ vanishes at $(\mathbf s,0)$, we can suppose that $r$ is such that $i\psi_0(x,v,v')$ is in 
\begin{align}\label{d0tau}
D(0, \tau)^d=\{z\in \C \, ; \, |z|<\tau\}^d
\end{align}
so by analyticity and using the parity of $m_{p,q}^n$%\color{red} (ok grace a parité $m^h$ si les $m_n$ dépendent pas de $h$)\color{black}
, we have
\begin{align*}
\partial^\beta_\eta m_{p,q}^n\Big(x,\frac{v+v'}{2},i\psi_0(x,v,v')\Big)=\mathop{\sum_{\gamma \in \N^d;}}_{\substack{|\gamma|+|\beta|\in 2\N}} i^{|\gamma|}\frac{\partial_\eta^{\gamma+\beta} m_{p,q}^n\big(x,\frac{v+v'}{2},0\big)}{\gamma !}\,\psi_0(x,v,v')^\gamma \in i^{|\beta|}\R.
\end{align*}
The result for $g_n$ follows easily using \eqref{g0} and \eqref{gn}.
\end{proof}
\hip
We also have the following result whose proof is postponed to Appendix \ref{rreelleapp} as it involves tedious calculations.

\begin{lem}\label{rreelle}
The term $R_j(\ell_0^{\mathbf s}, \dots, \ell_{j-1}^{\mathbf s})$ from Lemma \ref{expom} is real valued.
Moreover, it satisfies $R_j(\ell_0^{\mathbf s}, \dots, \ell_{j-1}^{\mathbf s})=-R_j(-\ell_0^{\mathbf s}, \dots, -\ell_{j-1}^{\mathbf s})$.
\end{lem}

\hip
In view of the results from Proposition $\ref{phf}$ and Lemma \ref{expom}, we want to find $\ell$ such that on $B_0(\mathbf s, 2r)$,
\begin{align}\label{eikon}
H_W \cdot \nabla \ell_0+M_0\Big(x,v,i\big(\frac v2 +\ell_0 \, \partial_v \ell_0\big)\Big)\big(v+\ell_0 \partial_v \ell_0\big)\cdot \partial_v\ell_0=0
\end{align}
and for $j \geq 1$
\begin{align}\label{transport}
 \Bigg[H_W+\begin{pmatrix}
				0 \\
				M_0\Big(x,v,i\big(\frac v2 +\ell_0 \, \partial_v \ell_0\big)\Big)( v+\ell_0 \partial_v \ell_0)+\ell_0\, U
				\end{pmatrix}\Bigg] \cdot \nabla \ell_j& \\
+ \partial_v \ell_0 \cdot U \, \ell_j+R_j(\ell_0&, \dots, \ell_{j-1})=0\nonumber
\end{align}
where $U$ was introduced in \eqref{U}.
Note that Lemmas \ref{ureelle} and \ref{rreelle} ensure that the fact that the $(\ell_j)_{j\geq 0}$ are real valued is compatible with equations \eqref{transport}.

\subsection{Solving for $\ell_0^{\mathbf s}$}
Denote $$p(x,v,\xi,\eta)= i\xi \cdot v - i\eta \cdot \partial_x V  +(-i \, {}^t\eta+\, {}^tv/2)M_0(x,v,\eta) (i\eta+v/2)$$ the principal symbol of the whole operator $P_h$ and $\tilde p(x,v,\xi,\eta)=-p(x,v,i\xi,i\eta)$ its complexification.
%We have 
%$$\tilde p$$
After computing the Hamiltonian of $\tilde p$
%The Hamiltonian of $\tilde p$ is 
%\begin{align*}
%H_{\tilde p}&=\begin{pmatrix} 
%		\partial_{(\xi, \eta)}\tilde p \\
%		-\partial_{(x,v)}\tilde p
%		\end{pmatrix}
%		=\begin{pmatrix}
%		v\\
%		-\partial_x V+\displaystyle\sum_{1\leq p,q \leq d}m^0_{p,q}(x,v,i\eta)\,(\eta_p e_q+\eta_q e_p) +i\sum_{1\leq p,q \leq d}(\eta_p \eta_q-v_pv_q/4)\partial_\eta m^0_{p,q}(x,v,i\eta)  \\
%		\mathrm{Hess}_xV \eta- \displaystyle\sum_{1\leq p,q \leq d}(\eta_p \eta_q -v_pv_q/4)\,\partial_x m^0_{p,q}(x,v,i\eta)\\
%		-\xi+\frac14 \displaystyle\sum_{1\leq p,q \leq d}m^0_{p,q}(x,v,i\eta)\,(v_p e_q+v_q e_p) -\sum_{1\leq p,q \leq d}(\eta_p \eta_q-v_pv_q/4)\partial_v m^0_{p,q}(x,v,i\eta)  
%		\end{pmatrix}
%\end{align*}
which vanishes at $(\mathbf s,0,0,0)$, we find that its linearization at this point is the matrix
\begin{align*}
F&=\begin{pmatrix} 
		0 & \mathrm{Id} & 0&0 \\
		-\mathrm{Hess}_{\mathbf s}V & 0&0&2M_0(\mathbf s,0,0)\\
		0&0&0&\mathrm{Hess}_{\mathbf s}V\\
		0&\frac12 M_0(\mathbf s,0,0)&-\mathrm{Id}&0
		\end{pmatrix}.
\end{align*}
%$$\color{red}------------------------- \color{black}$$
One can easily check that for any eigenvector $(x,v,\xi,\eta)$ of $F$ associated to an eigenvalue $\lambda$, the vector $(-x,v,\xi,-\eta)$ is an eigenvector associated to $-\lambda$ so the spectrum of $F$ is centrally symmetric with respect to the origin.
Moreover, writing 
$$F=\begin{pmatrix}
0&0&\mathrm{Id}&0\\
0&0&0&\mathrm{Id}\\
\mathrm{Id}&0&0&0\\
0&\mathrm{Id}&0&0
\end{pmatrix}
\begin{pmatrix} 
		0 &  0&0& \mathrm{Hess}_{\mathbf s}V\\
		0 &\frac12 M_0(\mathbf s,0,0)&-\mathrm{Id}&0\\
		0&\mathrm{Id}&0&0\\
		-\mathrm{Hess}_{\mathbf s}V&0&0&2 M_0(\mathbf s,0,0)
		\end{pmatrix}$$
and noticing that 
$$F\Big(\{v=\eta=0\}\Big) \cap \{v=\eta=0\}=\mathrm{Ker}\,F\cap \{v=\eta=0\}=\{0\},$$
we see that $F$ satisfies the assumptions of Lemma \ref{pasir}.
Therefore, $F$ has no eigenvalues in $i\R$ so it has $2d$ eigenvalues (counted with algebraic multiplicity) in $\{\mathrm{Re}\,z >0\}$ while the $2d$ others are in $\{\mathrm{Re}\,z <0\}$.
Therefore we can apply the stable manifold theorem to get %the existence of  
that the stable manifolds associated to $H_{\tilde p}$ given in a neighborhood of $(\mathbf s,0,0,0)$ by 
$$\Lambda_\pm=\Big\{(x,v,\xi,\eta)\, ; \lim_{t\to \mp \infty}\e^{tH_{\tilde p}}(x,v,\xi,\eta)=(\mathbf s,0,0,0)\Big\}$$
%satisfy 
%
%\color{red} (car on connait le tangeant et ce nest pas $(x,v)=0$). \color{black}
are both of dimension $2d$ and for all $\rho_\pm \in \Lambda_\pm$, we have 
\begin{align}\label{hptangent}
H_{\tilde p}(\rho_\pm)\in T_{\rho_\pm}\Lambda_\pm
\end{align}
and for $t>0$,
$$\big\|\e^{\mp tH_{\tilde p}}\rho_\pm-(\mathbf s,0,0,0)\big\|\leq C\e^{-t/C}\|\rho_\pm-(\mathbf s,0,0,0)\|.$$
%Moreover, the complexification of the tangent space of $\Lambda_\pm$ at $(\mathbf s,0,0,0)$ denoted $\C \otimes T_{\mathbf s} \Lambda_\pm$ is %\color{red} ( \color{black} the sum of the generalized eigenspaces of $F$ corresponding to eigenvalues with real part in $\R^*_\pm$: \color{red} ) \color{black}
%\begin{align}\label{tangent}
%\C \otimes T_{\mathbf s} \Lambda_\pm=\bigoplus_{\pm \mathrm{Re}\, \lambda >0}E_\lambda %\quad \color{red} \text{compléter quand j'aurai les bonnnes notations des sous-ev propres/ carac} \color{black}
%\end{align}
%where the notation $E_\lambda$ comes from Appendix \ref{F} and stands for the generalized eigenspace associated to the eigenvalue $\lambda$, a descrition of this space being given in Lemma \ref{specf}. 
Moreover, we have (see for instance \cite{DimassiSjostrand} Lemmas 3.2 and 3.3) that
\begin{align}\label{plam0}
\tilde p (\Lambda_\pm)=\{0\}
\end{align}
and $\Lambda_\pm$ are Lagrangian manifolds.
In order to get some parametrization for those manifolds, we follow the steps of \cite{HHS}, Lemma 8.1. 

\begin{lem}\label{transverse}
The tangent spaces $T_{(\mathbf s,0,0,0)}\Lambda_\pm$ that we denote for shortness $T_{\mathbf s} \Lambda_\pm$ are transverse to both $\{(\mathbf s,0)\}\times \R^{2d}$ and $\R^{2d}\times\{(0,0)\}$.
\end{lem}

\begin{proof}
We provide an adaptation of the proof from \cite{HHS} as some simplifications appear in our case.
Since we are working in the linearized case, we can assume that $\tilde p$ coincides with its quadratic approximation at $(\mathbf s,0,0,0)$ and for commodity we will work with the variable $x_{\mathbf s}=x-\mathbf s$ instead of $x$.
Note that if $a$ is a quadratic form, its Hamiltonian $H_a$ is then linear and we denote $F_a$ the associated matrix.
We then decompose $\tilde p=p_2+p_1-p_0$ where
$$p_2=M_0(\mathbf s,0,0)\eta \cdot \eta, \quad p_1=v\cdot \xi-\mathrm{Hess}_sVx_{\mathbf s}\cdot \eta \quad \text{and} \quad p_0=\frac14 M_0(\mathbf s,0,0)v \cdot v.$$
It is clear that $p_2+p_0$ is positive semi-definite, moreover, the subspace $\{v=\eta=0\}$ on which $p_2+p_0$ vanishes satisfies $\{v=\eta=0\}\cap F_{p_1}^{-1}\big(\{v=\eta=0\}\big)=\{0\}.$
Thus the quadratic form
$$\tilde q=(p_2+p_0)+(p_2+p_0)\circ F_{p_1}$$
is positive definite.
Let us denote $L_\pm=\Lambda_\pm\cap\{x_{\mathbf s}=v=0\}$.
To prove that $L_\pm=\{0\}$, it is sufficient to establish that $\tilde q=0$ on $L_\pm$.
In order to do so, we will show that $L_\pm$ is an $F_{p_1}$-invariant subspace on which $p_2+p_0=0$.
Indeed, it is clear that $p_0=p_1=0$ on $L_\pm$ and thanks to \eqref{plam0} we deduce that $p_2$ also vanishes on $L_\pm$ so in particular $p_2+p_0=0$ on $L_\pm$.
It also implies that $L_\pm$ is included in $\{\eta=0\}$ so $F_{p_2}|_{L_\pm}=0$.
Besides, we clearly have $F_{p_0}|_{L_\pm}=0$ so $F_{p_1}$ coincides on $L_\pm$ with $F_{\tilde p}$ which leaves $\Lambda_\pm$ invariant according to \eqref{hptangent}.
Since it is easy to see that $\{x_{\mathbf s}=v=0\}$ is also invariant under $F_{p_1}$, we can conclude as announced that $L_\pm=\{0\}$.
The proof that $\Lambda_\pm\cap\{\xi=\eta=0\}=\{0\}$ is similar.
\end{proof}
\hip
Since $\Lambda_\pm$ are Lagrangian manifolds such that $T_{\mathbf s} \Lambda_\pm$ are transverse to $\{(\mathbf s,0)\}\times \R^{2d}$, there exist $\phi_\pm \in \mathcal C^\infty(B_0(\mathbf s,2r),\R)$ vanishing together with their gradients at $(\mathbf s,0)$ and such that
\begin{align}\label{lamphi}
\Lambda_\pm=\Big\{\Big((x,v, \nabla\phi_\pm (x,v) \Big) \, ; (x,v)\in B_0(\mathbf s,2r)\Big\}.
\end{align}
Therefore, $T_{\mathbf s} \Lambda_\pm$ coincide with the graphs of the matrices Hess$_{(\mathbf s,0)}\phi_\pm$ which are then invertible according to Lemma \ref{transverse}. %also implies that det Hess$_{(\mathbf s,0)}\phi_\pm\neq 0$.
Now we need a result similar to the one of Proposition 8.2 in \cite{HHS}.

\begin{lem}\label{defpos}
The Hessian matrix of $\pm \phi_\pm$ at $(\mathbf s,0)$ is definite positive.
\end{lem}

\begin{proof}
The proof is simply an adaptation of the one found in \cite{HHS}.
Here again we will assume that $\tilde p$ coincides with its quadratic approximation at $(\mathbf s,0,0,0)$ and work with the variable $x_{\mathbf s}=x-\mathbf s$ instead of $x$.
For $\delta \in [0,1]$, let us denote 
\begin{align*}
\tilde p^\delta&=(1-\delta)\tilde p+\delta\big(\xi^2+\eta^2-(x_{\mathbf s}^2+v^2)\big)\\
	&=p_2^\delta+(1-\delta)p_1-p_0^\delta
\end{align*}
where 
$$p_2^\delta=(1-\delta) p_2+\delta(\xi^2+\eta^2) \qquad \text{and} \qquad p_0^\delta=(1-\delta) p_0+\delta(x_{\mathbf s}^2+v^2).$$
Note in particular that $\tilde p^0=\tilde p$ and that $\tilde p^1=\big(\xi^2+\eta^2-(x_{\mathbf s}^2+v^2)\big)$ corresponds to the well know Schr\"odinger case (see for instance \cite{DimassiSjostrand}, chapter 3).
Besides, we have that
\begin{align*}
F_{\tilde p^\delta}=\begin{pmatrix}
0&0&\mathrm{Id}&0\\
0&0&0&\mathrm{Id}\\
\mathrm{Id}&0&0&0\\
0&\mathrm{Id}&0&0
\end{pmatrix}\left[(1-\delta)\begin{pmatrix} 
		0 &  0&0& \mathrm{Hess}_{\mathbf s}V\\
		0 &\frac12 M_0(\mathbf s,0,0)&-\mathrm{Id}&0\\
		0&\mathrm{Id}&0&0\\
		-\mathrm{Hess}_{\mathbf s}V&0&0&2 M_0(\mathbf s,0,0)
		\end{pmatrix}
+2\delta \, \mathrm{Id} \right]
\end{align*}
so Lemma \ref{pasir} easily yields that the eigenvalues of $F_{\tilde p^\delta}$ cannot cross $i\R$ for some $\delta \in (0,1]$.
Moreover, it is clear that for $\delta \in (0,1]$, the quadratic form $p_2^\delta+p_0^\delta$ is positive definite, so the results of Lemma \ref{transverse} are true for the $2d$-dimensional Lagrangian planes
$$\Lambda_\pm^\delta=\Big\{(x_{\mathbf s},v,\xi,\eta)\, ; \lim_{t\to \mp \infty}\e^{tF_{\tilde p^\delta}}(x,v,\xi,\eta)=0\Big\}$$
for all $\delta\in [0,1]$.
In particular, there exist $\phi_\pm^\delta \in \mathcal C^\infty(B_0(\mathbf s,2r),\R)$ %vanishing togethe\delta with thei\delta g\deltaadients at $(\mathbf s,0)$ and 
such that
\begin{align}\label{lamalpha}
T_{\mathbf s}\Lambda_\pm^\delta=\Lambda_\pm^\delta=\Big\{\Big(x_{\mathbf s},v, \mathrm{Hess}_{(\mathbf s,0)}\phi_\pm^\delta\begin{pmatrix}x_{\mathbf s}\\v \end{pmatrix} \Big) \, ; (x_{\mathbf s},v)\in \R^{2d}\Big\}.
\end{align}
Hence the graph of $\mathrm{Hess}_{(\mathbf s,0)}\phi_\pm^\delta$ is given by $T_{\mathbf s}\Lambda_\pm^\delta$ which also corresponds to the sum of the generalized eigenspaces of $F_{\tilde p^\delta}$ associated to eigenvalues in $\{\pm \mathrm{Re}\, z<0\}$ and therefore depends continuously on $\delta$.
Besides, by Lemma \ref{transverse}, $\mathrm{Hess}_{(\mathbf s,0)}\phi_\pm^\delta$ is invertible for all $\delta\in [0,1]$ and we know from the Schr\"odinger case that $\pm\mathrm{Hess}_{(\mathbf s,0)}\phi_\pm^1>0$ so necessarily $\pm\mathrm{Hess}_{(\mathbf s,0)}\phi_\pm>0$.
\end{proof}

%Here we use some of the notations and results from Appendix \ref{F}.
%First notice that combinig \eqref{tangent} and \eqref{lamphi}, we can deduce 
%$$\big\{\big(z,\mathrm{Hess}_{(\mathbf s,0)}\phi_- z\big)\, ; z \in \C^{2d}\big\}=\bigoplus_{\mathrm{Re}\, \lambda >0}E_\lambda.$$
%Hence %for $(z,\tilde z) \in E_\lambda$ where Re $\lambda >0$, 
%for $z \in \C^{2d}$, there exists $(\alpha_\lambda)_{\mathrm{Re}\, \lambda >0} \in \C^{2d}$ such that 
%$$z=\sum_{\mathrm{Re}\, \lambda >0} \alpha_\lambda e_{\lambda,1}\qquad \text{ and } \qquad \mathrm{Hess}_{(\mathbf s,0)}\phi_- z=\sum_{\mathrm{Re}\, \lambda >0} \alpha_\lambda e_{\lambda,2}$$
%where $(e_{\lambda,1}, e_{\lambda,2})\in E_\lambda$.
%Thus, using \eqref{propre}, \eqref{carac} and the orthogonality of the eigenspaces of $\mathrm{Hess}_{\mathbf s}V$, we see that
%$$\langle \mathrm{Hess}_{(\mathbf s,0)}\phi_- z, z \rangle=\sum_{\lambda, \lambda'} \alpha_\lambda \alpha_{\lambda'}\langle e_{\lambda,2}, e_{\lambda',1} \rangle=\sum_{\mathrm{Re}\,\lambda >0} |\alpha_\lambda|^2\langle e_{\lambda,2}, e_{\lambda,1} \rangle$$
%and using once again \eqref{propre} and \eqref{carac} as well as the fact that if $\lambda \in S_\mu$ with $\mu>m^2/4$, then $|\lambda|^2=\mu$, one can easily check that $\langle e_{\lambda,2}, e_{\lambda,1} \rangle<0$ for all $(e_{\lambda,1}, e_{\lambda,2})\in E_\lambda$.\\
%The same proof holds for $\phi_-$ by changing some of the signs.
%\end{proof}
%$$\color{green}------------------------- \color{black}$$
\hip
At this point, one can proceed as in \cite{BonyLPMichel}, Lemma 3.2 to establish the following Lemma.

\begin{lem}\label{l0exist}
There exists $\ell_0^{\mathbf s} \in \mathcal C^\infty(B_0(\mathbf s,2r),\R)$ such that for $(x,v) \in B_0(\mathbf s,2r)$, 
$$\phi_+(x,v)=W(x,v)-W(\mathbf s,0)+\frac{\ell_0^{\mathbf s}(x,v)^2}{2}.$$
In particular, $\ell_0^{\mathbf s}$ vanishes at $(\mathbf s,0)$.
Moreover, $\{\ell_0^{\mathbf s}\neq 0\}$ is dense in $B_0(\mathbf s,2r)$.
\end{lem}

\hip
This function also appears to solve \eqref{eikon} as we see in the next Proposition.

\begin{prop}\label{phinu}
The function $\ell_0^{\mathbf s}$ from Lemma $\ref{l0exist}$ is a solution of \eqref{eikon} in $B_0(\mathbf s,2r)$.
Moreover, the vector $\nabla \ell_0^{\mathbf s} (\mathbf s,0)$ that we denote $\nu^{\mathbf s}= \begin{pmatrix}\nu_1^{\mathbf s}\\ \nu_2^{\mathbf s}\end{pmatrix}$ is not 0 and satisfies $\Phi^{\mathbf s} \nu^{\mathbf s}=\big(-M_0(\mathbf s,0,0) \nu_2^{\mathbf s} \cdot \nu_2^{\mathbf s}\big) \,\nu^{\mathbf s}$, where 
\begin{align}\label{matphi}
\Phi^{\mathbf s}=\begin{pmatrix}
0&-\mathrm{Hess}_{\mathbf s} V\\
\mathrm{Id}&M_0(\mathbf s,0,0)
\end{pmatrix}.
\end{align}
In particular, since $\Phi^{\mathbf s}$ is invertible, $\nu_2^{\mathbf s}\neq 0$. Finally,
$$\det \bigg(\mathrm{Hess}_{(\mathbf s,0)}\bigg(W+\frac{(\ell_0^{\mathbf s})^2}{2}\bigg)\bigg)=2^{-2d}\big|\det (\mathrm{Hess}_{\mathbf s}V)\big|.$$
\end{prop}

\begin{proof} The proof is the same as in \cite{BonyLPMichel}, Lemma 3.3 after matching the notations by setting $\Lambda(\mathbf s)=\Phi^{\mathbf s}$, $b^0=H_W$,
$$A^0(\mathbf s)=\begin{pmatrix}0&0\\
0&M_0(\mathbf s,0,0)
\end{pmatrix} \quad \text{and }\quad
B(\mathbf s)=\begin{pmatrix}0&\mathrm{Id}\\
-\mathrm{Hess}_{\mathbf s}V&0
\end{pmatrix}. 
$$
\end{proof}

\subsection{Solving for $\big(\ell_j^{\mathbf s}\big)_{j\geq 1}$}

Once again we drop some exponents $\mathbf s$ for shortness.
Now that $\ell_0$ is given by Lemma $\ref{l0exist}$ and Proposition \ref{phinu}, we can solve the transport equations \eqref{transport} by induction, so we suppose that $\ell_0,\dots,\ell_{j-1}$ are given and we want to find a solution $\ell_j$ to \eqref{transport}.
Denote 
$$\widetilde U=H_W+\begin{pmatrix}
				0 \\
				M_0\Big(x,v,i\big(\frac v2 +\ell_0 \, \partial_v \ell_0\big)\Big)( v+\ell_0 \partial_v \ell_0)+\ell_0\, U
				\end{pmatrix}\in \mathcal C^\infty(B_0(\mathbf s,2r))$$
and 
$$\alpha=\partial_v \ell_0 \cdot U \in \mathcal C^\infty(B_0(\mathbf s,2r))$$
where $U$ was introduced in \eqref{U}.
The function $\ell_j$ must satisfy $(\widetilde U \cdot \nabla +\alpha) \ell_j=-R_j(\ell_0,\dots,\ell_{j-1})$ so we are intersted in the operaor $\mathcal L=\widetilde U \cdot \nabla +\alpha$ that we decompose as $\mathcal L=\mathcal L_0^{\mathbf s}+\mathcal L_>$
with 
\begin{align*}
\mathcal L_0^{\mathbf s}
%=\begin{pmatrix}
%				0&\mathrm{Id} \\
%				-\mathrm{Hess}_{\mathbf s}V +2M_0(\mathbf s,0,0)\nu_2\, {}^t\nu_1 & M_0(\mathbf s,0,0)(\mathrm{Id}+ 2 \nu_2\, {}^t\nu_2)
%				\end{pmatrix}
%		\begin{pmatrix}x-\mathbf s\\
%				v
%		\end{pmatrix}\cdot \nabla+M_0(\mathbf s,0,0) \nu_2 \cdot \nu_2\\
		=\widetilde U_0^{\mathbf s}\begin{pmatrix}x-\mathbf s\\
				v
		\end{pmatrix}\cdot \nabla+\alpha_0^{\mathbf s}
\end{align*}
where $\widetilde U_0^{\mathbf s}$ is the differential of $\widetilde U$ at $(\mathbf s,0)$ and $\alpha_0^{\mathbf s}=\alpha(\mathbf s,0)$, that is
$$%\begin{align}\label{alpha0}
\widetilde U_0^{\mathbf s}=\begin{pmatrix}
				0&\mathrm{Id} \\
				-\mathrm{Hess}_{\mathbf s}V +2M_0(\mathbf s,0,0)\nu_2^{\mathbf s}\, {}^t\nu_1^{\mathbf s} & M_0(\mathbf s,0,0)(\mathrm{Id}+ 2 \nu_2^{\mathbf s}\, {}^t\nu_2^{\mathbf s})
				\end{pmatrix} 
$$%\end{align}
and
\begin{align}\label{alpha0} 
\alpha_0^{\mathbf s}=M_0(\mathbf s,0,0) \nu_2^{\mathbf s} \cdot \nu_2^{\mathbf s}.
\end{align}
As usual, we will often omitt the exponents $\mathbf s$ in the notations.
Notice that if we denote $\mathcal P^n_{hom}$ the space of homogeneous polynomials of degree $n$ in the variables $(x-\mathbf s,v)$, we have $\mathcal L_0 \in \mathscr L(\mathcal P^n_{hom})$
and for $P \in \mathcal P^n_{hom}$, $\mathcal L_>P(x,v) =O\big((x-\mathbf s,v)^{n+1}\big)$ near $(\mathbf s,0)$.

\begin{lem}\label{l0inv}
The negative eigenvalue $-\alpha_0^{\mathbf s}$ of the matrix $\Phi^{\mathbf s}$ from Proposition \ref{phinu} is its only one (counting multiplicity) in $\{\mathrm{Re}\,z\leq 0\}$.
Moreover, all the eigenvalues of $\widetilde U_0^{\mathbf s}$ have positive real part and the operator $\mathcal L_0^{\mathbf s}$ is invertible on $\mathcal P^n_{hom}$.
\end{lem}

\begin{proof}
It is sufficient to prove the first statement. Indeed, if $-\alpha_0$ is the only eigenvalue of $\Phi$ in $\{\mathrm{Re\,z\leq 0}\}$, we can then remark that 
$${}^t\widetilde U_0=\Phi+2\begin{pmatrix}
0&\nu_1\,{}^t\nu_2 M_0(\mathbf s,0,0)\\
0&\nu_2\,{}^t\nu_2 M_0(\mathbf s,0,0)
\end{pmatrix}$$
and since the last term has its range included in $\C \nu$ and sends $\nu$ on $2\alpha_0 \nu$, the matrix of ${}^t\widetilde U_0$ in a basis $(\nu,b_2,\dots,b_{2d})$ in which $\Phi$ becomes triangular is also triangular and has on its diagonal the eigenvalues of $\Phi$ except for $-\alpha_0$ which is replaced by $+\alpha_0$.
Hence $\mathrm{Spec} (\widetilde U_0)=\mathrm{Spec} ({}^t\widetilde U_0)\subset \{\mathrm{Re} \,z>0\}$ and we can conclude thanks to Lemma A.1 from \cite{BonyLPMichel}.
%Since $\alpha_0=M_0(\mathbf s,0,0) \nu_2 \cdot \nu_2>0$, it is sufficient according to Lemma 7.1  in \cite{BonyLPMichel} to prove that $\mathrm{Spec} (\widetilde U_0)\subset \{\mathrm{Re} \,z>0\}$.
%We saw in Lemma $\ref{phinu}$ that $-\alpha_0$ is an eigenvalue of $\Phi$, 
Let us then prove that $-\alpha_0$ is the only eigenvalue (counting multiplicity) of $\Phi$ in $\{\mathrm{Re\,z\leq 0}\}$.
We proceed as in \cite{BonyLPMichel}, Lemma 2.6.
For $t \in [0,1]$, consider the matrix 
\begin{align*}
\Phi_t&=2\,\mathrm{Hess}_{\mathbf s}W\begin{pmatrix}
(1-t)\mathrm{Id}&-t\mathrm{Id}\\
t\mathrm{Id}&tM_0(\mathbf s,0,0)+(1-t)\mathrm{Id}
\end{pmatrix}
%\\&=\begin{pmatrix}
%\mathrm{Hess}_{\mathbf s}V&0\\
%0& \mathrm{Id}
%\end{pmatrix}\begin{pmatrix}
%(1-t)\mathrm{Id}&-t\mathrm{Id}\\
%t\mathrm{Id}&tm_0(\mathbf s,0,0)+(1-t)\mathrm{Id}
%\end{pmatrix}
\end{align*}
which trivially satisfies the assumptions of Lemma \ref{pasir} for $t\in [0,1)$.
It is also the case of $\Phi_1$ as $\Phi_1(x,0)=(0,x)$.
Hence for every $t\in [0,1]$, $\Phi_t$ has no eigenvalues in $i\R$ and since these eigenvalues depend continuously on $t$, we get that 
$$\#\big(\mathrm{Spec} \, \Phi_1 \cap \{\mathrm{Re} \,z<0\}\big)=\#\big(\mathrm{Spec} \, \Phi_0 \cap \{\mathrm{Re} \,z<0\}\big).$$
But $\Phi_0=2\,\mathrm{Hess}_{\mathbf s}W$ has exactly one negative eigenvalue (with multiplicity) while all the others are positive since $\mathbf s \in \mathcal U^{(1)}$, so we have indeed showed that $-\alpha_0$ is the only eigenvalue of $\Phi=\Phi_1$ (counting multiplicity) in $\{\mathrm{Re} \,z \leq 0\}$.
\end{proof} 
\hip
One can then proceed as in \cite{BonyLPMichel}, section 3.3 (see also \cite{DimassiSjostrand}, chapter 3), i.e use Lemma \ref{l0inv} to find an approximate solution of \eqref{transport} using formal power series and then refine it into an actual solution using again Lemma \ref{l0inv} as well as the characteristic method.
We then get the following result.
\begin{prop}
For all $j\geq 1$, there exists $\ell^{\mathbf s}_j \in \mathcal C^\infty(B_0(\mathbf s,2r))$ solving \eqref{transport}.
Moreover, $\ell_j^{\mathbf s}$ is real valued in view of Lemmas \ref{ureelle} and \ref{rreelle}.
\end{prop}

\section{Computation of the small eigenvalues}\label{sectionvp}

Now that we have found $(\ell_j)_{j\geq 0} \subset \mathcal C^\infty(B_0(\mathbf s,2r),\R)$ solving \eqref{eikon} and \eqref{transport} with $\ell_0$ vanishing at $(\mathbf s,0)$, we can use a Borel procedure to construct $\ell \in \mathcal C^\infty(\R^{2d},\R)$ supported in $B_0(\mathbf s,3r)$ and satisfying $\ell \sim \sum_{j \geq 0}\ell_j$ on $B_0(\mathbf s,2r)$. 
%as well as the last statement from Lemma \ref{phinu} to obtain the following:
%\begin{prop}\label{lexist}
%There exists a function $\ell$ satisfying all the properties listed in \eqref{hypol} and such that the coefficients from its classical expansion solve \eqref{eikon} and \eqref{transport}.
%\end{prop}
%Moreover, using Lemmas \ref{defpos} and \ref{l0exist}, we see that $\ell$ satisfies the assumptions of Lemma \ref{llisse}.
\begin{rema}\label{signl}
The properties \ref{hypol1}-\ref{hypoln-2} from \eqref{hypol} are satisfied by both the functions $\ell^{\mathbf s,h}$ and $-\ell^{\mathbf s,h}$.
Moreover, by Lemma \ref{rreelle}, $(-\ell_j^{\mathbf s})_{j\geq 0}$ also solve \eqref{eikon} and \eqref{transport}. 
\end{rema}
\hip
We are now in position to prove that all the properties from \eqref{hypol} are satisfied.

\begin{prop}\label{lexist}
We can choose the signs of the functions $(\ell^{\mathbf s,h})_{ \mathbf j(\mathbf m)}$ such that \eqref{hypol} holds true and the coefficients from the classical expansion of $\ell^{\mathbf s,h}$ solve \eqref{eikon} and \eqref{transport}.
%Suppose that the $(\ell^{\mathbf s,h})_{(\mathbf s,0)\in \mathbf j(m)}$ have been chosen such that there exists a smooth function $\phi$ admitting a non degenerate minimum at $(\mathbf s,0)$ and satisfying 
%$$\phi=W-W(\mathbf s,0)+\frac{(\ell^{\mathbf s,h})^2}{2}+O(h)\qquad \text{on }B_0(\mathbf s,r).$$ 
\end{prop}

\begin{proof} 
Recall that by item \ref{jvide2} from Hypothesis \ref{jvide}, each function $\ell^{\mathbf s,h}$ corresponds to a unique $\mathbf m \in \mathcal U^{(0)}\backslash\{\underline{\mathbf m}\}$.
%Since $\big(B_0(\mathbf s,r)\cap \{|\ell^{\mathbf s,h}|\leq \gamma\}\big)$ are open sets, we have that 
%$$\partial \Big[\Big(E(\mathbf m)+ 2\varepsilon \Big)\backslash \Big(\bigsqcup_{\mathbf j(\mathbf m)} \big(B_0(\mathbf s,r)\cap \{|\ell^{\mathbf s,h}|\leq \gamma\}\big) \Big)\Big]\subseteq \partial\Big(E(\mathbf m)+ 2\varepsilon \Big)\backslash \Big(\bigsqcup_{\mathbf j(\mathbf m)} \big(B_0(\mathbf s,r)\cap \{|\ell^{\mathbf s,h}|\leq \gamma\}\big) \Big) \bigcup \partial\Big(\bigsqcup_{\mathbf j(\mathbf m)} \big(B_0(\mathbf s,r)\cap \{|\ell^{\mathbf s,h}|\leq \gamma\}\big) \Big)\cap \Big(E(\mathbf m)+ 2\varepsilon \Big). $$
Thanks to Lemmas \ref{defpos} and \ref{l0exist}, it is clear that item \ref{hypoln-1} from \eqref{hypol} is satisfied by both $\ell^{\mathbf s,h}$ and $-\ell^{\mathbf s,h}$.
Hence according to Remark \ref{signl}, it is sufficient to prove that the signs of $(\ell^{\mathbf s,h})_{\mathbf j(\mathbf m)}$ can be chosen so that $\theta_{\mathbf m,h}$ is smooth on a neighborhood of supp $\chi_{\mathbf m}$.
From \eqref{thetainte}, \eqref{theta1} and \eqref{theta0} we see that the only parts on which it is not clear that $\theta_{\mathbf m,h}$ is smooth are 
$$F_{1}=\bigsqcup_{\mathbf s \in \mathbf j(\mathbf m)}\Big(\{|\ell_0^{\mathbf s}|\leq 2\gamma\}\cap \partial B_0(\mathbf s,r)\Big), \quad  F_{2}=\bigsqcup_{\mathbf s \in\mathbf j(\mathbf m)}\Big(B_0(\mathbf s,r)\cap \{|\ell_0^{\mathbf s}|=2\gamma\}\Big)$$
$$\qquad \text{ and } \quad F_3=\partial \Big(E(\mathbf m)+B(0,\varepsilon)\Big)\backslash \Big(\bigsqcup_{\mathbf s \in\mathbf j(\mathbf m)} \big(B_0(\mathbf s,r)\cap \{|\ell_0^{\mathbf s}|\leq 2\gamma\}\big) \Big).$$
Let $\mathbf s\in \mathbf j(\mathbf m)$ and $(x,v) \in \overline{B_0(\mathbf s,r)}\backslash \{(\mathbf s,0)\}$ such that $\ell_0^{\mathbf s}(x,v)=0$.
Using Lemma \ref{l0exist}, we see that if $r>0$ is small enough,
\begin{align}\label{zerodel}
W(x,v)-W(\mathbf s,0)=\phi_+(x,v)>0
\end{align}
because $(\mathbf s,0)$ is a non degenerate local minimum of $\phi_+$.
Hence, $\{\ell_0^{\mathbf s}=0\}\cap B_0(\mathbf s,r)\subset \{W\geq \boldsymbol \sigma(\mathbf m)\}$.
Now assume by contradiction that for any $r>0$, the function $\ell_0^{\mathbf s}$ takes both positive and negative values on $E(\mathbf m)\cap B_0(\mathbf s,r)$.
%Since $(\mathbf s,0)\in \mathcal V^{(1)}$, $\{W<\boldsymbol \sigma (\mathbf m)\}\cap$ would 
Then according to Lemma \ref{1.4}, the two CCs of $U_r\cap \{W<\boldsymbol \sigma (\mathbf m)\}$ are both included in $E(\mathbf m)$ (the one on which $\ell_0^{\mathbf s}>0$ and the one where $\ell_0^{\mathbf s}<0$).
This is a contradiction with the fact that $\mathbf s\in \mathcal V^{(1)}$.
Therefore $\ell_0^{\mathbf s}$ has a sign on $E(\mathbf m)\cap B_0(\mathbf s,r)$ and %in view of Remark \ref{signl}, 
we can choose it so that $\ell_0^{\mathbf s}$ is a positive function on $E(\mathbf m)\cap B_0(\mathbf s,r)$. %after changing this sign if needed (which is legit in view of the hypothesis we made on $\ell^{\mathbf s,h}$), 
By uniform continuity, we can then choose $\varepsilon(\gamma)>0$ small enough so that 
\begin{align}\label{1.10}
\Big( \big(E(\mathbf m)+ B(0,\varepsilon) \big) \cap B_0(\mathbf s,r)\Big) \subseteq \big\{\ell_0^{\mathbf s} \geq -\gamma\big\}.
\end{align}
Similarly, if we denote $\Omega_{\mathbf s}$ the other CC of $\{W<\boldsymbol \sigma (\mathbf m)\}$ which contains $(\mathbf s,0)$ on its boundary, we have since $(\mathbf s,0)$ is not a critical point of $\ell_0^{\mathbf s}$ that this function is negative %$\ell^{\mathbf s,h}$ is a negative function 
on $\Omega_{\mathbf s}\cap B_0(\mathbf s,r)$ and 
\begin{align}\label{1.10bis}
\Big( \big(\Omega_{\mathbf s}+B(0, \varepsilon) \big) \cap B_0(\mathbf s,r)\Big) \subseteq \big\{\ell_0^{\mathbf s} \leq \gamma\big\}.
\end{align}
Choosing once again $\varepsilon(r)$ small enough, we can even assume that 
\begin{align}\label{1.11}
\Big( \overline{E(\mathbf m)+B(0,\varepsilon)}\,  \cap \,   \overline{\Omega_{\mathbf s}+B(0,\varepsilon)} \Big) \subseteq B_0(\mathbf s,r).
\end{align}
We first prove that $F_{1}$ does not meet the support of $\chi_{\mathbf m}$.
Recall that $\Omega$ denotes the CC of $\{W\leq \boldsymbol \sigma(\mathbf m)\}$ containing $\mathbf m$.
For $\mathbf s\in \mathbf j(\mathbf m)$, we can deduce from \eqref{zerodel} that if $(x,v)\in \partial B_0(\mathbf s,r)$ such that $\ell_0^{\mathbf s}(x,v)=0$, then $(x,v) \notin  \Omega$.
Hence $|\ell_0^{\mathbf s}|$ must attain a positive minimum on $\partial B_0(\mathbf s,r)\cap  \Omega$, so we can choose $\gamma(r)>0$ such that $\partial B_0(\mathbf s,r)\cap \{|\ell_0^{\mathbf s}|\leq 2 \gamma\}$ does not intersect $ \Omega$.
It follows that we can choose $\varepsilon (\gamma)>0$ such that 
\begin{align*}%\label{F1}
F_{1} \subseteq \big( \R^{2d} \backslash \overline{\Omega+B(0,\varepsilon)} \big) \subseteq \big( \R^{2d} \backslash  \mathrm{supp }\,\chi_{\mathbf m} \big).
\end{align*}
Now we show that $\theta_{\mathbf m,h}$ is smooth on $F_{2}\cap (\Omega+B(0,\varepsilon))$: let $\mathbf s\in \mathbf j(\mathbf m)$ and $(x,v) \in B_0(\mathbf s,r)\cap \{\ell_0^{\mathbf s}=2\gamma\}\cap (\Omega+B(0,\varepsilon))$.
According to \eqref{1.10bis} and the fact that $\ell^{\mathbf s,h}=\ell_0^{\mathbf s}+O(h)$, there exists a small ball $B$ centered in $(x,v)$ such that 
$$B\subset \Big(B_0(\mathbf s,r)\cap \{\ell^{\mathbf s,h}> \gamma\}\cap \big(E(\mathbf m)+B(0,\varepsilon)\big)\Big).$$
Thus $\theta_{\mathbf m,h}=1$ on $B$ and $\theta_{\mathbf m,h}$ is smooth at $(x,v)$. 
Similarly, for $(x,v) \in B_0(\mathbf s,r)\cap \{\ell^{\mathbf s,h}=-2\gamma\}\cap (\Omega+B(0,\varepsilon))$, we can show that $\theta_{\mathbf m,h}=0$ in a neighborhood of $(x,v)$.\\
It only remains to prove that, as for $F_{1}$, the set $F_3$ does not meet the support of $\chi_{\mathbf m}$.
First we remark that thanks to \eqref{1.10}, we can forget the absolute value in the definition of $F_3$:
$$F_3=\partial \Big(E(\mathbf m)+B(0,\varepsilon)\Big)\backslash \Big(\bigsqcup_{\mathbf j(\mathbf m)} \big(B_0(\mathbf s,r)\cap \{\ell_0^{\mathbf s}\leq 2\gamma\}\big) \Big).$$
If $(x,v)\in F_3 \cap B_0(\mathbf s,r)$, we have that $\ell_0^{\mathbf s}(x,v)>2\gamma$ so using \eqref{1.10bis}, we see that $(x,v)$ is outside $\Omega_{\mathbf s}+B(0,\varepsilon)$.
Since it is not in $(E(\mathbf m)+B(0,\varepsilon))$ either, it is outside $\Omega+B(0,\varepsilon)$ which contains the support of $\chi_\mathbf m$.
Now if $(x,v)\in F_3 \backslash \big( \mathbf j^W(\mathbf m)+B_0(0,r)\big)$, \eqref{1.11} implies that $(x,v)$ is outside $\cup_{\mathbf j(\mathbf m)}(\Omega_{\mathbf s}+B(0,\varepsilon))$ so it is also outside $\Omega+B(0,\varepsilon)$ for $\varepsilon$ small enough and the proof is complete.
\end{proof}

\begin{lem}\label{Pff}
Let $\mathbf m \in \mathcal U^{(0)}\backslash \{\underline{\mathbf m}\}$ and denote $\tilde f_{\mathbf m,h}=f_{\mathbf m,h}/\|f_{\mathbf m,h}\|$ where $f_{\mathbf m,h}$ was defined in \eqref{quasim}.
With the notation \eqref{alpha0}, we have that
$$\langle P_h \tilde f_{\mathbf m,h}, \tilde f_{\mathbf m,h} \rangle=h\e^{-2\frac{S(\mathbf m)}{h}}\frac{\det (\mathrm{Hess}_{\mathbf m}V)^{1/2}}{2\pi} \tilde B_h(\mathbf m)\in \R$$
with $\tilde B_h(\mathbf m)$ admitting a classical expansion whose first term equals 
$$\sum_{\mathbf s \in \mathbf j(\mathbf m)} |\det (\mathrm{Hess}_{\mathbf s}V)|^{-1/2} \,\alpha_0^{\mathbf s}.$$
\end{lem}

\begin{proof} 
Since $X_0^h$ is a skew-adjoint differential operator and $f_{\mathbf m,h}$ is real valued, we have 
\begin{align}\label{x0ff}
\langle X_0^h f_{\mathbf m,h}, f_{\mathbf m,h} \rangle=0.
\end{align}
Besides, we know from \eqref{qf} that 
\begin{align}\label{bhf}
b_h f_{\mathbf m,h}=h(\partial_v \theta) \chi \e^{-W_{\mathbf m}/h}+O_{L^2}(h^\infty \e^{-S(\mathbf m)/h})
\end{align}
so we easily deduce from the fact that $(\partial_v \theta) \chi \e^{-W_{\mathbf m}/h}=O_{L^2}( \e^{-S(\mathbf m)/h})$ and the boundedness of Op$_h(M^h)$ that 
$$\langle Q_h f_{\mathbf m,h}, f_{\mathbf m,h} \rangle=h^2\Big\langle\mathrm{Op}_h(M^h)\big((\partial_v \theta) \chi \e^{-W_{\mathbf m}/h} \big)\, ,\, (\partial_v \theta) \chi \e^{-W_{\mathbf m}/h}\Big\rangle+O\Big(h^\infty \e^{-\frac{2S(\mathbf m)}{h}}\Big).$$
Since we have with the notation \eqref{ytilde}
$$(\partial_v \theta) \chi \e^{-W_{\mathbf m}/h}=\frac{A_h^{-1}}{2}\e^{-\widetilde W_{\mathbf m}/h}\chi\sum_{\mathbf s\in \mathbf j(\mathbf m)}\zeta(\ell^{\mathbf s})  \partial_v \ell^{\mathbf s} \,\1_{B_0(\mathbf s,r)}$$
and using \eqref{opgtruc} with $M$ instead of $g$, we get that
\begin{align}\label{pff}
\qquad\langle P_h f_{\mathbf m,h}, f_{\mathbf m,h} \rangle=\frac {h^2}{4} A_h^{-2}\sum_{\mathbf s\in \mathbf j(\mathbf m)}\int_{B_0(\mathbf s,r)}\e^{-2\widetilde W_{\mathbf m}(x,v)/h}\chi \zeta(\ell^{\mathbf s})\tilde I^{\mathbf s}(x,v)\cdot\, \partial_v &\ell^{\mathbf s} \,\D (x,v)\\
		&+O\Big(h^\infty \e^{-2\frac{S(\mathbf m)}{h}}\Big).
\end{align}
where 
\begin{align*}
\tilde I^{\mathbf s}(x,v)=(2\pi h)^{-d}\int_{\R^d}\int_{|v'|\leq r} \e^{\frac ih \eta \cdot (v-v')}\chi(x,v') \zeta\big(\ell^{\mathbf s}(x,v')\big) M\Big(x,\frac{v+v'}{2},\eta+i\psi(x,v,v')\Big) \partial_v\ell^{\mathbf s}(x,v')\;\D v' \D \eta.
\end{align*}
Mimicking the proof of Proposition \ref{expdl}, one can show that $\zeta(\ell)$ admits a classical expansion whose first term is $\zeta(\ell_0)$.
Besides, since $M$ and $\psi$ also have a classical expansion, we could use the stationnary phase (see for instance \cite{Zworski}, Theorem 3.17) as well Proposition \ref{expinexp} to get an expansion of $\tilde I$ similar to the one obtained in \eqref{expinteosci}.
%%Here we can once again use the stationnary phase (cf for instance \cite{Zworski}, Theorem 3.17) to see 
%In particular, it is clear that $\tilde I=O_{L^\infty}(1)$.
%We can then conclude with \eqref{approxa} and the fact that $\widetilde W \geq W_{\mathbf m}\geq S(\mathbf m)+\tilde \varepsilon$ on the support of $\partial_v \chi$.\\
%Putting together \eqref{x0ff}, \eqref{qff1}, \eqref{qff2} and \eqref{opm} we now know that
%\begin{align}\label{pff}
%\qquad\langle P_h f_{\mathbf m,h}, f_{\mathbf m,h} \rangle&=\frac {h^2}{4} A_h^{-2}\sum_{\mathbf s\in \mathbf j(\mathbf m)}\int_{B_0(\mathbf s,r)}\e^{-2\widetilde W(x,v)/h}\chi \zeta(\ell)\tilde I(x,v)\cdot\, \partial_v \ell \,\D (x,v)\\
%		&\quad+O\Big(h^\infty \e^{-2\frac{S(\mathbf m)}{h}}\Big).
%\end{align}
%As we have just mentionned, we can find a classical expansion of $\tilde I$ as we did in \eqref{expinteosci} and use 
Thus, we get that $\tilde I\cdot\, \partial_v \ell\sim \sum_{k\geq 0}h^k a_k$ where 
$$a_0(x,v)=\chi(x,v) \zeta\big(\ell_0(x,v)\big)M_0\Big(x,v,i\big(\frac v2 +\ell_0 \, \partial_v \ell_0\big)\Big) \partial_v \ell_0(x,v)\cdot \partial_v \ell_0(x,v).$$
%and for $k\geq 1$,
%\begin{align*}
%a_k(x,v)=\chi(x,v) &\zeta\big(\ell(x,v)\big)\mathop{\sum}_{\substack{n_1+n_2+n_3\\+n_4+n_5=k}}\frac{1}{i^{n_1}n_1!}\\
%		&\times\big(\partial_{v'} \cdot \partial_\eta \big)^{n_1} \Big( M_{n_2,n_3}(x,v,v',\eta) \partial_v \ell_{n_4}(x,v') \Big) \Bigg| \mathop{}_{\substack{v'=v \\  \eta =0}}\cdot \, \partial_v \ell_{n_5}(x,v)
%\end{align*}
%with $M_{n_2,n_3}$ defined in a similar way as in \eqref{gnj}.
Hence, using the fact that on $B_0(\mathbf s,r)$,
$$\widetilde W-S(\mathbf m)=W_{\mathbf m}+\frac{\ell^2_0}{2}-S(\mathbf m)+\Big(\frac{\ell^2}{2} -\frac{\ell^2_0}{2}\Big),$$
it is clear that 
\begin{align}\label{expqff}
\qquad e^{2S(\mathbf m)/h}\int_{B_0(\mathbf s,r)}\e^{-2\widetilde W(x,v)/h}&\chi \zeta(\ell)\tilde I(x,v)\cdot\, \partial_v \ell \,\D (x,v)\sim_h\\
&\sum_{k\geq 0}h^k\int_{B_0(\mathbf s,r)}\e^{-2\frac{W_{\mathbf m}(x,v)+\ell^2_0(x,v)/2-S(\mathbf m)}{h}}\e^{-\frac{(\ell^2-\ell_0^2)(x,v)}{h}}\chi \zeta(\ell)a_k \,\D (x,v).
\end{align}
%Now let us apply the Morse lemma to the $h$-independent function $W_{\mathbf m}+\ell^2_0/2$: since $W_{\mathbf m}(\mathbf s,0)+\ell_0^2(\mathbf s,0)/2=S(\mathbf m)$, there exists $\phi:\phi^{-1}(B_0(\mathbf s,r)) \to B_0(\mathbf s,r)$ a smooth diffeomorphism sending $0\in \phi^{-1}(B_0(\mathbf s,r))$ on $(\mathbf s,0)$, whose differential at $0$ is unitary and such that 
%\begin{align}\label{intemoche}
%\qquad \quad \int_{B_0(\mathbf s,r)} \e^{-2\frac{W_{\mathbf m}(x,v)+\ell^2_0(x,v)/2-S(\mathbf m)}{h}}&\e^{-\frac{(\ell^2-\ell_0^2)(x,v)}{h}}\chi \zeta(\ell)a_k \,\D (x,v)=\\
%\int_{\phi^{-1}(B_0(\mathbf s,r))}\e^{-\frac{\mathrm{Hess_{(\mathbf s,0)}}(W+\ell^2_0/2) (y,w) \cdot (y,w)}{h}} &\e^{-\frac{(\ell^2-\ell^2_0)(\phi(y,w))}{h}}\chi(\phi(y,w)) \zeta(\ell(\phi(y,w))\nonumber\\
%&\qquad \qquad \times a_k(\phi(y,w)) \big|\det D_{(y,w)}\phi \big|\D (y,w).\nonumber
%\end{align}
We would like to apply Proposition \ref{expscal} so we need to check that the assumptions are satisfied.
First, $\mathrm{Hess_{(\mathbf s,0)}}(W_{\mathbf m}+\ell^2_0/2)$ is definite positive by Lemma \ref{defpos}.
Besides, $h^{-1}(\ell^2-\ell_0^2)$ admits a classical expansion whose first term is $2(\ell_1 \ell_0)$.
Therefore, using the expansion of $\zeta(\ell)$ as well as Proposition \ref{expdl}, one easily gets that the function
$$\e^{-\frac{(\ell^2-\ell^2_0)}{h}} \big(\zeta\circ\ell\big)$$
admits a classical expansion whose first term is $\e^{-2 (\ell_1 \ell_0)}\big(\zeta\circ\ell_0\big)$.
Thus, according to Propositions \ref{expscal} and \ref{phinu}, there exists $(b_{k,j})$ such that
\begin{align}
\frac{|\det (\mathrm{Hess}_{\mathbf s}V)|^{1/2}}{(2\pi h)^{d}}\int_{B_0(\mathbf s,r)}\e^{-2\frac{W_{\mathbf m}(x,v)+\ell^2_0(x,v)/2-S(\mathbf m)}{h}}\e^{-\frac{(\ell^2-\ell_0^2)(x,v)}{h}}\chi \zeta(\ell)a_k \,\D (x,v)
\sim \sum_{j\geq 0}h^j b_{k,j}
\end{align}
where $b_{k,0}=a_k(\mathbf s,0)$.
Hence, using \eqref{pff}, \eqref{expqff} and Proposition \ref{expinexp}, we deduce that 
\begin{align}\label{exppff}
4A_h^2(2\pi)^{-d} h^{-d-2}\e^{2S(\mathbf m)/h}\langle P_h f_{\mathbf m,h}, f_{\mathbf m,h} \rangle\sim \sum_{k\geq 0}h^k c_k
\end{align}
with
%$$c_0=\sum_{\mathbf j(\mathbf m)} \det (\mathrm{Hess}_{(\mathbf s,0)}\widetilde W)^{-1/2}a_0(\mathbf s,0)=\sum_{\mathbf j(\mathbf m)} \det (\mathrm{Hess}_{(\mathbf s,0)}\widetilde W)^{-1/2}M_0(\mathbf s,0,0)\nu_2 \cdot \nu_2$$
%and for $k\geq 1$,
%$$c_k=\sum_{(\mathbf s,0)\in \mathbf j(\mathbf m)} |\det (\mathrm{Hess}_{\mathbf s}V)|^{-1/2} \sum_{j=0}^k \frac{\big(\frac12 \mathrm{Hess_{(\mathbf s,0)}}(W+\ell^2_0)\nabla \cdot \nabla\big)^j b_{k-j}(0)}{j!}.$$
%In particular,
$$c_0=\sum_{\mathbf s \in \mathbf j(\mathbf m)} |\det (\mathrm{Hess}_{\mathbf s}V)|^{-1/2}M_0(\mathbf s,0,0)\nu_2^{\mathbf s} \cdot \nu_2^{\mathbf s}=\sum_{\mathbf s \in\mathbf j(\mathbf m)} |\det (\mathrm{Hess}_{\mathbf s}V)|^{-1/2} \,\alpha_0^{\mathbf s}.$$
Similarly, thanks to item \ref{mseul} from Hypothesis \ref{jvide}, one can use Proposition \ref{expscal} as we already did to see that there exists $(\tilde c_k)_{k\geq 0}$ such that
\begin{align}\label{expff}
\frac{\det(\mathrm{Hess_{\mathbf m}V})^{1/2}}{(2\pi h)^d}\|f_{\mathbf m,h}\|^2\sim \sum_{k\geq 0}h^k\tilde c_k
\end{align}
with $\tilde c_0=1$.
The conclusion follows from \eqref{exppff}, \eqref{approxa} and \eqref{expff}.
\end{proof}

\begin{lem}\label{Pf^2}
Let $\mathbf m \in \mathcal U^{(0)}\backslash \{\underline{\mathbf m}\}$.
Using the notations from Lemma \ref{Pff}, we have 
\begin{enumerate}[label=\roman*)]
\item $\|P_h\tilde f_{\mathbf m,h}\|^2=O(h^\infty \langle P_h\tilde f_{\mathbf m,h}, \tilde f_{\mathbf m,h} \rangle)$
\item $\|P_h^*\tilde f_{\mathbf m,h}\|^2=O(h\langle P_h\tilde f_{\mathbf m,h}, \tilde f_{\mathbf m,h} \rangle)$.
\end{enumerate}
\end{lem}

\begin{proof} 
To prove $i)$, first remark that thanks to \eqref{x0f}-\eqref{vgrand} we have
\begin{align}\label{xloin}
\int_{\R^{2d}\backslash (\mathbf j^W(\mathbf m)+B_0(0,2r))}|P_h f_{\mathbf m,h}(x,v)|^2\D (x,v)  =O\Big(h^\infty \e^{-2\frac{S(\mathbf m)}{h}}\Big).
\end{align}
%=O(h^\infty \langle P_h\tilde f_{\mathbf m,h}, \tilde f_{\mathbf m,h} \rangle)$$
%according to Lemma \ref{Pff}.
%Besides, the oscillatory integral appearing in \eqref{opgtruc} is $O_{L^2(B(\mathbf s,r)\times \R^d_v)}(1)$; this can be seen by doing the usual integrations by parts when $|v|>>r$.
%Moreover, for $|v|\geq r$ and $x\in B(\mathbf s,r)$, we have 
%\begin{align*}
%\widetilde W(x,v)\geq W_{\mathbf m}(x,v)&\geq \frac{V(x)-V(\mathbf m)}{2}+\frac{r^2}{4}\\
%			&\geq W_{\mathbf m}(\mathbf s,0) -\frac{\sup_{x\in K}|\partial_xV(x)|}{2}|x-\mathbf s|+\frac{r^2}{4}\\
%			&\geq S(\mathbf m)+\frac{r^2}{8} 
%\end{align*}
%by the definition of $r$.
%Thus, we can also deduce using \eqref{r2/8} and \eqref{approxa} that 
%\begin{align}\label{vloin}
%\int_{B(\mathbf j(\mathbf m),r)}\int_{|v|\geq r}|P_h f_{\mathbf m,h}(x,v)|^2\D x \D v =O\Big(\e^{-2\frac{S(\mathbf m)+ r^2/8}{h}}\Big).
%\end{align}
Besides, we saw that thanks to Proposition \ref{expscal} and Lemma \ref{defpos}, we have for $\mathbf s \in \mathbf j(\mathbf m)$,
\begin{align}\label{e-wtilde}
\int_{B_0(\mathbf s,2r)}\e^{-2\frac{\widetilde W(x,v)}{h}}\D (x,v)=O\Big(h^d \e^{-2\frac{S(\mathbf m)}{h}}\Big).
\end{align}
Moreover, the function $\omega$ from Proposition \ref{phf} is $O_{L^\infty(B_0(\mathbf s,2r))}(h^\infty)$ by Lemma \ref{expom} and the construction of the $(\ell^{\mathbf s,h})_{\mathbf s \in \mathbf j(\mathbf m)}$.
Hence, by Proposition \ref{phf},
\begin{align}\label{pres}
\int_{B_0(\mathbf s,2r)}|P_h f_{\mathbf m,h}(x,v)|^2\D (x,v) =O\Big(h^\infty \e^{-2\frac{S(\mathbf m)}{h}}\Big).
\end{align}
The conclusion follows from \eqref{xloin}, \eqref{pres} as well as \eqref{expff} and Lemma \ref{Pff}.
The proof of $ii)$ can be obtained similarly with the use of Proposition \ref{expscal} and Remark \ref{Ph*f} after noticing that $\overset{*}{\om}$ also admits a classical expansion whose first term vanishes on $\mathbf j^W(\mathbf m)$.
\end{proof}
\hip
From now on, we denote 
\begin{align}\label{lamtilde}
\tilde \lambda_{\mathbf m,h}=\langle P_h \tilde f_{\mathbf m,h}, \tilde f_{\mathbf m,h} \rangle=\langle Q_h \tilde f_{\mathbf m,h}, \tilde f_{\mathbf m,h} \rangle
\end{align}
for which we computed a classical expansion in Lemma \ref{Pff}.

\begin{lem}\label{f1f2}
For $\mathbf m$ and $\mathbf m'$ two distinct elements of $\mathcal U^{(0)}$, we have
\begin{enumerate}[label=\roman*)]
\item $\langle P_h\tilde f_{\mathbf m,h}, \tilde f_{\mathbf m',h} \rangle=O\Big(h^\infty\sqrt{\tilde \lambda_{\mathbf m,h}\tilde \lambda_{\mathbf m',h}}\Big)$
\item There exists $c>0$ such that $\langle \tilde f_{\mathbf m,h}, \tilde f_{\mathbf m',h} \rangle=O(\e^{-c/h})$ \label{f1f22}
\end{enumerate}
\end{lem}

\begin{proof} 
i): The result is obvious when one of the two minima is $\underline{\mathbf m}$.
Recall the labeling of the minima that we introduced rigth before Hypothesis \ref{jvide} as well as the map $\pi_x$ from Lemma \ref{corres}.
Let us first suppose that $\mathbf m=\mathbf m_{k,j}$ and $\mathbf m'=\mathbf m_{k,j'}$ with $j\neq j'$ and $k\neq 1$ and denote $E=E(\mathbf m)$ and $E'=E(\mathbf m')$.
In particular $\boldsymbol \sigma(\mathbf m)=\boldsymbol \sigma(\mathbf m')$.
Thanks to \eqref{suppf} and the fact that $P_h$ is local in $x$, we have 
$$\mathrm{supp}\, P_h\tilde f_{\mathbf m,h}\subseteq  \big(\pi_x(E)\times \R^d_v\big)+B(0,\varepsilon') \qquad \text{and} \qquad \mathrm{supp}\,\tilde f_{\mathbf m',h}\subseteq \big(E'+B(0,\varepsilon')\big)$$
so up to taking $\varepsilon'$ small enough, it is sufficient to show that $\overline{\pi_x(E)\times \R^d_v}$ and $\overline{E'}$ do not intersect.
Since our labeling is adapted, $E$ and $E'$ are two distinct CCs of $\{W<\boldsymbol \sigma(\mathbf m)\}$ so by Lemma \ref{corres}, $\pi_x(E)\times \R^d_v$ and $E'$ are two disjoint open sets.
Thus, using successively Remark \ref{obs} and \eqref{projadh}, we get 
\begin{align}
\overline{\pi_x(E)\times \R^d_v}\cap\overline{E'}&=\Big(\partial\big(\pi_x(E)\big)\times \R^d_v\Big) \cap \partial E'\\
		&\subseteq \Big(\partial\big(\pi_x(E)\big)\times\{0\}\Big) \cap \partial E'\\
		&\subseteq \Big(\partial\big(\pi_x(E)\big) \cap \partial\big(\pi_x(E')\big)\Big)\times\{0\}.
\end{align}
which is empty thanks to Lemma \ref{1.4bis} and item \ref{jvide2} from Hypothesis \ref{jvide}.\\
Let us now treat the case $\mathbf m=\mathbf m_{k,j}$ and $\mathbf m'=\mathbf m_{k',j'}$ with $k,k'\geq 2$ and $k\neq k'$.
We can suppose that $k<k'$ (i.e $\boldsymbol \sigma(\mathbf m)>\boldsymbol \sigma(\mathbf m')$) because we can work with $P_h^*$ instead of $P_h$ if needed.
We decompose $P_h\tilde f_{\mathbf m,h}$ as in \eqref{x0f} and \eqref{qf} and once again we use \eqref{suppf} to get 
$$\mathrm{supp}\,\tilde f_{\mathbf m',h}\subseteq \big(E'+B(0,\varepsilon')\big)\subseteq \Big\{W<\frac{\boldsymbol \sigma(\mathbf m)+\boldsymbol \sigma(\mathbf m')}{2}\Big\}$$
as well as the fact that $P_h$ is local in $x$ to get a localization of the support of the first term from \eqref{qf}:
\begin{align}
\mathrm{supp}\,\Big(\mathrm{Op}_h(g)\big((\partial_v \theta_{\mathbf m}) \chi_{\mathbf m} \e^{-W_{\mathbf m}/h}\big)\Big)\subseteq \Big( \big(\mathbf j(\mathbf m)+B(0,r)\big)\times \R^d_v\Big)\subseteq \Big\{W>\frac{\boldsymbol \sigma(\mathbf m)+\boldsymbol \sigma(\mathbf m')}{2}\Big\}
\end{align}
as $W$ increases with the norm of $v$.
Hence, the support of the first term from \eqref{qf} does not meet the one of $\tilde f_{\mathbf m',h}$.
The same goes easily for the first term of \eqref{x0f}.
For the second term of \eqref{x0f}, its support is contained in the support of $\nabla \chi_{\mathbf m}$ which is itself contained in $\{W\geq \boldsymbol \sigma(\mathbf m)+\tilde \varepsilon\}$ so it clearly does not meet the support of $\tilde f_{\mathbf m',h}$.
It only remains to treat the second term from \eqref{qf}, i.e $\mathrm{Op}_h(g)\big(\theta_{\mathbf m} (\partial_v \chi_{\mathbf m}) \e^{-W_{\mathbf m}/h}\big)$.
To this aim, notice that \eqref{bhf} yields $b_hf_{\mathbf m',h}=O_{L^2}(\e^{-S(\mathbf m')/h})$ and since by the support properties of $\nabla \chi_{\mathbf m}$ we also have $\theta_{\mathbf m} (\partial_v \chi_{\mathbf m}) \e^{-W_{\mathbf m}/h}=O_{L^2}(h^\infty \e^{-S(\mathbf m)/h})$, we get using the Cauchy-Schwarz inequality and the boundedness of $\mathrm{Op}_h(M)$
\begin{align}
\Big\langle \mathrm{Op}_h(g)\big(\theta_{\mathbf m} (\partial_v \chi_{\mathbf m}) \e^{-W_{\mathbf m}/h}\big)\, , \, f_{\mathbf m',h} \Big \rangle &=\Big\langle \mathrm{Op}_h(M)\big(\theta_{\mathbf m} (\partial_v \chi_{\mathbf m}) \e^{-W_{\mathbf m}/h}\big)\, , \, b_hf_{\mathbf m',h} \Big \rangle\\
			&=O\Big(h^\infty \e^{-\frac{S(\mathbf m)+S(\mathbf m')}{h}}\Big)
\end{align}
which proves the first item.
%Using again the support properties of $\nabla \chi_{\mathbf m}$ as well as the fact that
%$$\overline{\pi_x\big(E'+B(0,\varepsilon')\big)\times \{0\}}\subseteq \{W<\boldsymbol \sigma(\mathbf m)\}$$
%we get that $|v|$ is bounded from below by some positive constant on 
%$$\mathrm{supp}\, \nabla \chi_{\mathbf m}\cap \Big(\overline{\pi_x\big(E'+B(0,\varepsilon')\big)}\times \R^d_v\Big).$$
%Hence, choosing $r$ small enough, the support of \eqref{2.6-2} is contained in $\{x \notin \pi_x\big(E'+B(0,\varepsilon')\big)\}\times \R^d_v$
%so in particular it does not intersect with the support of $\tilde f_{\mathbf m',h}$.
\hip
ii): Here we can suppose that $V(\mathbf m)\geq V(\mathbf m')$.
Let us first treat the case where $V(\mathbf m)= V(\mathbf m')$.
Then according to item \ref{mseul} from Hypothesis \ref{jvide}, $E$ and $E'$ are two disjoint open sets.
Hence, as we saw earlier, Lemma \ref{1.4bis} and item \ref{jvide2} from Hypothesis \ref{jvide} imply that $\overline{E}\cap \overline{E'}=\emptyset$.
The conclusion then follows from \eqref{suppf}.\\
If $V(\mathbf m)> V(\mathbf m')$, then item \ref{mseul} from Hypothesis \ref{jvide} implies that $(\mathbf m,0)$ is the only global minimum of $W|_{E+B(0,\varepsilon')}$.
Therefore using \eqref{suppf}, we can easily compute %for instance the Morse Lemma and Proposition \ref{expscal} as well as (\ref{suppf}), we can compute 
\begin{align*}
\langle  f_{\mathbf m,h},  f_{\mathbf m',h} \rangle=\int_{E+B(0,\varepsilon')} \theta_{\mathbf m}\theta_{\mathbf m'}\chi_{\mathbf m}\chi_{\mathbf m'} \e^{-\frac{2V-V(\mathbf m)-V(\mathbf m')+v^2}{2h}}\D(x,v)=O\Big(\e^{-\frac{V(\mathbf m)-V(\mathbf m')}{2h}}\Big).
\end{align*}
The conclusion immediately follows from \eqref{expff}.
\end{proof}

%\begin{rema}\label{0termaterm}
%For the first item, we have actually proven that 
%$$\langle X_0^h\tilde f_{\mathbf m,h}, \tilde f_{\mathbf m',h} \rangle=\langle Q_h\tilde f_{\mathbf m,h}, \tilde f_{\mathbf m',h} \rangle=0.$$
%\end{rema}
\hip
Let us consider once again the spectral projection introduced in \eqref{Pi0}.
We saw in particular that $\Pi_0=O(1)$.

\begin{lem}\label{1-pi0}
For any $\mathbf m\in \mathcal U^{(0)}$, we have 
$$\|(1-\Pi_0)\tilde f_{\mathbf m,h}\|=O\Big(h^\infty \sqrt{\tilde \lambda_{\mathbf m,h}}\Big)\qquad \text{and} \qquad \|(1-\Pi_0^*)\tilde f_{\mathbf m,h}\|=O\Big(h^{-3/2}\sqrt{\tilde \lambda_{\mathbf m,h}}\Big).$$
\end{lem}

\begin{proof} 
We simply recall the proof from \cite{LPMichel}:
we write 
\begin{align*}
(1-\Pi_0)\tilde f_{\mathbf m,h}&=\frac{1}{2i\pi}\int_{|z|= c h^2}\big(z^{-1}-(z-P_h)^{-1}\big)\tilde f_{\mathbf m,h}\D z\\
		&=\frac{-1}{2i\pi}\int_{|z|= c h^2}z^{-1}(z-P_h)^{-1}P_h\tilde f_{\mathbf m,h}\D z.
\end{align*}
We can then conclude using Lemma \ref{Pf^2} and the resolvent estimate from Theorem \ref{thmRobbe}.
The proof for the adjoint is almost identical.
\end{proof}

\begin{lem}\label{Pi0fon}
The family $(\Pi_0\tilde f_{\mathbf m,h})_{\mathbf m\in \mathcal U^{(0)}}$ is almost orthonormal:
there exists $c>0$ such that 
$$\langle \Pi_0\tilde f_{\mathbf m,h}, \Pi_0\tilde f_{\mathbf m',h}  \rangle=\delta_{\mathbf m, \mathbf m'}+O(\e^{-c/h}).$$
In particular, it is a basis of the space $H=\mathrm{Ran}\, \Pi_0$ introduced in \eqref{Pi0}.\\
Moreover, we have
$$\langle P_h\Pi_0\tilde f_{\mathbf m,h}, \Pi_0\tilde f_{\mathbf m',h}  \rangle=\delta_{\mathbf m, \mathbf m'}\tilde \lambda_{\mathbf m,h}+O\Big(h^\infty\sqrt{\tilde \lambda_{\mathbf m,h}\tilde \lambda_{\mathbf m',h}}\Big).$$
\end{lem}

\begin{proof}
The proof is the same as the one of Proposition 4.10 in \cite{LPMichel}.
\end{proof}
\hip
Let us re-label the local minima $\mathbf m_1,\dots, \mathbf m_{n_0}$ so that $(S(\mathbf m_j))_{j=1,\dots,n_0}$ is non increasing in $j$.
For shortness, we will now denote 
$$\tilde f_j=\tilde f_{\mathbf m_j,h} \qquad \text{and} \qquad \tilde \lambda_j=\tilde \lambda_{\mathbf m_j,h}$$
which still depend on $h$.
Note in particular that according to Lemma \ref{Pff}, $\tilde\lambda_j =O(\tilde\lambda_k)$ whenever $1\leq j\leq k \leq n_0$.
We also denote $(\tilde u_j)_{j=1,\dots,n_0}$ the orthogonalization by the Gram-Schmidt procedure of the family $(\Pi_0\tilde f_j)_{j=1,\dots,n_0}$ and 
$$u_j=\frac{\tilde u_j}{\|\tilde u_j\|}.$$
%In particular, thanks to Lemma \ref{Pi0fon} \color{red}+ mémoire \color{black}, it holds 
%$$u_{\mathbf m}=\Pi_0\tilde f_{\mathbf m,h}+O(\e^{-c/h}).$$
In this setting and with our previous results, we get the following (see \cite{LPMichel}, Proposition 4.12 for a proof).

\begin{lem}\label{Puu'}
For all $1\leq j,k \leq n_0$, it holds 
$$\langle P_h u_j, u_k \rangle=\delta_{j,k}\tilde \lambda_j +O\Big(h^\infty \sqrt{\tilde \lambda_j \tilde \lambda_k} \Big).$$
\end{lem}
\hip
In order to compute the small eigenvalues of $P_h$, let us now consider the restriction $P_h|_H:H\to H$.
We denote $\hat u_j=u_{n_0-j+1}$, $\hat \lambda_j=\tilde\lambda_{n_0-j+1}$ and $\mathcal M$ the matrix of $P_h|_H$ in the orthonormal basis $(\hat u_1,\dots,\hat u_{n_0})$.
Since $\hat u_{n_0}=u_1=\tilde f_1$, we have 
$$\mathcal M=\begin{pmatrix}
		\mathcal M'&0\\
		0&0
\end{pmatrix} \qquad \text{where }\qquad \mathcal M'=\Big(\langle P_h\hat u_{j}, \hat u_k \rangle\Big)_{1\leq j,k\leq n_0-1}$$
and it is sufficient to study the spectrum of $\mathcal M'$.
We will also denote $\{\hat S_1< \dots < \hat S_p\}$ the set $\{S(\mathbf m_j)\, ; \, 2\leq j \leq n_0\}$ and for $1\leq k \leq p$, $E_k$ the subspace of $L^2(\R^{2d})$ generated by $\{\hat u_r\, ; \,  S(\mathbf m_r)=\hat S_k\}$.
Finally, we set $\varpi_k=\e^{-(\hat S_k-\hat S_{k-1})/h}$ for $2\leq k \leq p$ and $\varepsilon_j(\varpi)=\prod_{k=2}^j \varpi_k=\e^{-(\hat S_j-\hat S_1)/h}$ for $2\leq j \leq p$ (with the convention $\varepsilon_1(\varpi)=1$).

\begin{prop}\label{propgas}
There exists a diagonal matrix $M^\#_h$ admitting a classical expansion whose first term is 
$$M_0^\#=\mathrm{diag}\bigg(\sum_{\mathbf s \in \mathbf j(\mathbf m_{n_0-j+1})} \frac{\det (\mathrm{Hess}_{\mathbf m_{n_0-j+1}}V)^{1/2}}{2\pi  |\det (\mathrm{Hess}_{\mathbf s}V)|^{1/2}} \alpha_0^{\mathbf s} \, ; \, 1\leq j \leq n_0-1\bigg)$$
such that
$$h^{-1}\e^{2\hat S_1/h}\mathcal M'=\Omega(\varpi) \big(M^\#_h+O(h^\infty)\big) \Omega(\varpi)$$
where $\Omega(\varpi)=\mathrm{diag}\big(\varepsilon_1(\varpi)\mathrm{Id}_{E_1},\dots,\varepsilon_p(\varpi)\mathrm{Id}_{E_p}\big)$.
\end{prop}

\begin{rema}\label{remgas}
In the words of Definition 6.7 from \cite{BonyLPMichel}, the last Proposition implies that $h^{-1}\e^{2\hat S_1/h}\mathcal M'$ is a classical graded symmetric matrix.
\end{rema}

\begin{proof}
According to Lemma \ref{Puu'}, we can decompose $\mathcal M'=\mathcal M'_1+\mathcal M'_2$ with 
$$\mathcal M'_1=\mathrm{diag}(\hat\lambda_j \, ; \, 1\leq j \leq n_0-1)\qquad \text{and} \qquad \mathcal M'_2=\Big(O\Big(h^\infty \sqrt{\hat\lambda_j \hat \lambda_k}\Big)\Big)_{1\leq j,k\leq n_0-1}.$$
\sloppy We will take $M^\#_h=h^{-1}\e^{2\hat S_1/h}\Omega(\varpi)^{-1}\mathcal M'_1 \Omega(\varpi)^{-1}$ which is clearly diagonal, so we just need to check that it has the proper classical expansion and that $h^{-1}\e^{2\hat S_1/h}\Omega(\varpi)^{-1}\mathcal M'_2 \Omega(\varpi)^{-1}=O(h^\infty)$.
It is easy to compute 
$$h^{-1}\e^{2\hat S_1/h}\Omega(\varpi)^{-1}\mathcal M'_1 \Omega(\varpi)^{-1}=h^{-1}\mathrm{diag}\Big(\e^{2\hat S_{j'}/h}\hat \lambda_j \, ; \, 1\leq j \leq n_0-1 \Big)$$
where $1\leq j' \leq p$ is such that %there eixsts $(\mathbf m,0) \in \mathcal U^{(0)}$ satisfying 
$\hat S_{j'}=S(\mathbf m_{n_0-j+1})$.
%and $\hat \lambda_j=\tilde \lambda_{\mathbf m,h}$.
Hence Lemma \ref{Pff} yields 
\begin{align*}
h^{-1}\e^{2\hat S_1/h}\Omega(\varpi)^{-1}\mathcal M'_1 \Omega(\varpi)^{-1}=\mathrm{diag}\bigg( \frac{\det (\mathrm{Hess}_{\mathbf m_{n_0-j+1}}V)^{1/2}}{2\pi}\tilde B_h(\mathbf m_{n_0-j+1})  \, ; \, 1\leq j \leq n_0-1\bigg)
\end{align*}
where $\tilde B_h(\mathbf m_{n_0-j+1})$ was introduced in Lemma \ref{Pff} and admits a classical expansion whose first term is 
$$\sum_{\mathbf s \in \mathbf j(\mathbf m_{n_0-j+1})} |\det (\mathrm{Hess}_{\mathbf s}V)|^{-1/2} \,\alpha_0^{\mathbf s}$$
so $M^\#_h$ has the desired expansion.
Similarly, still using Lemma \ref{Pff}, one easily gets 
$$\Omega(\varpi)^{-1}\mathcal M'_2 \Omega(\varpi)^{-1}=\Big(O\Big(h^\infty \sqrt{\hat \lambda_j \hat \lambda_k}\,\varepsilon_{j'}(\varpi)^{-1} \varepsilon_{k'}(\varpi)^{-1}\Big)\Big)_{1\leq j,k \leq n_0-1}$$
where $1\leq j'\leq p$ and $1\leq k'\leq p$ are such that $\sqrt{\hat \lambda_j}\,\varepsilon_{j'}(\varpi)^{-1}$ and $\sqrt{\hat \lambda_k}\,\varepsilon_{k'}(\varpi)^{-1}$ are both $O(\sqrt{h}\,\e^{-\hat S_1/h})$ so the proof is complete.
\end{proof}
\hip
$\textit{Proof of Theorem \ref{thmToto}}$.
According to Remark \ref{remgas}, it now suffices to combine the result of Proposition \ref{propgas} with Theorem 4 from \cite{BonyLPMichel} which gives a description of the spectrum of classical graded almost symmetric matrices.
Indeed, using the notations from this reference, we have for $1\leq j \leq p$ that 
$$\mathcal J\circ\mathcal R_j\Big(M_h^\# +O(h^\infty)\Big)=\mathcal J\circ\mathcal R_j\big(M_h^\# \big)+O(h^\infty)$$
and the result comes easily since $M_h^\#$ is diagonal.
Therefore, we have actually proved that $B_h(\mathbf m)$ from Theorem \ref{thmToto} and $\tilde B_h(\mathbf m)$ from Lemma \ref{Pff} have the same classical expansion.
\hspace*{\fill} $\Box$

%\begin{align*}
%Q_h(w^h\e^{-W/h})&=\mathrm{Op}_h(g^h)\big(h\zeta(\ell)\e^{-\widetilde W/h}\partial_v\ell\big)\\
%			&=\frac{h}{(2 \pi h)^{2d}}\int \int \e^{\frac ih(\xi (x-x')+\eta (v-v'))}g^h\Big(\frac{x+x'}{2},\frac{v+v'}{2},\eta\Big)\qquad \text{\color{red} car gh dépend pas de } \xi \\
%			&\qquad \qquad \qquad \qquad\qquad \qquad\times \zeta(\ell(x',v')) \e^{-\widetilde W(x',v')/h}\partial_v\ell(x',v')\;\D(x',v') \D(\xi, \eta)\\
%			&=\frac{h}{(2 \pi h)^{2d}}\e^{-\widetilde W(x,v)/h}\int \int \e^{\frac ih\big((\xi, \eta)-i\psi(x,v,v')\big)\cdot \big(x-x',v-v'\big)}\\
%			&\qquad \qquad \qquad \qquad\qquad \qquad\times  g^h\Big(\frac{x+x'}{2},\frac{v+v'}{2},\eta\Big)\zeta(\ell(x',v')) \partial_v\ell(x',v')\;\D(x',v') \D(\xi, \eta)\\
%\end{align*}

%\newpage

\section{Return to equilibrium and metastability} \label{sectionral}

%\color{green}\begin{rema}
%Si $M^h$ paire aussi en $v$, alors on a la PT-sym car $b_hU_\kappa=-U_\kappa b_h$ avec $U_\kappa$ le chgmt de Robbe pour PT-sym.
%\end{rema}\color{black}
The goal of this section is to prove Corollaries \ref{ral} and \ref{meta}.
We assume that the hypotheses of Theorem \ref{thmToto} are satisfied and we choose $\mathbf m^*$ among the elements of $\mathcal U^{(0)}\backslash\{\underline{\mathbf m}\}$ for which $S$ is maximal such that the expansion of $\det (\mathrm{Hess}_{\mathbf m^*}V)^{1/2}B_h(\mathbf m^*)$ is minimal.
%$$\e^{-2\frac{S(\mathbf m^*)}{h}}\det (\mathrm{Hess}_{\mathbf m^*}V)^{1/2}\sum_{\mathbf s \in \mathbf j(\mathbf m^*)} |\det (\mathrm{Hess}_{\mathbf s}V)|^{-1/2} \,\alpha_0^{\mathbf s}$$
%is minimal.
According to Lemma \ref{Pff} and Theorem \ref{thmToto}, one can think of $\lambda_{\mathbf m^*,h}$ as the non zero eigenvalue of $P_h$ with the smallest real part modulo $O(h^\infty\e^{-2S(\mathbf m^*)/h})$.
We will denote $\Pro_1$ the orthogonal projection on $\mathrm{Ker}$ $P_h$ and for shortness $ \lambda^*$ instead of $ \lambda_{\mathbf m^*,h}$.
\hop
$\textit{Proof of Corollary \ref{ral}}$.
We follow the proof of Theorem 1.11 in \cite{LPMichel}. 
We have that 
$$\|\e^{-tP_h/h}-\Pro_1\|\leq \|\e^{-tP_h/h}\Pi_0-\Pro_1\|+\|\e^{-tP_h/h}(1-\Pi_0)\|.$$
and thanks to Proposition \ref{1-pi0l2} and Proposition 2.1 from \cite{HS-restosg}, we easily get
\begin{align}\label{gp}
\e^{-tP_h/h}(1-\Pi_0)=O(\e^{-cht}).
\end{align}
Thus it suffices for the first statement to prove that 
$$ \|\e^{-tP_h/h}\Pi_0-\Pro_1\|\leq C_N \e^{-\mathrm{Re}\, \lambda^*(1-C_N h^N)t/h}.$$
We recall that thanks to the resolvent estimates from Theorem \ref{thmRobbe}, $\Pi_0=O(1)$ and since $\Pro_1$ is an orthogonal projection on $\mathrm{Ker}$ $P_h$, we have that
$$\e^{-tP_h/h}\Pi_0-\Pro_1=\e^{-tP_h/h}(\Pi_0-\Pro_1)$$
and $(\Pi_0-\Pro_1)=O(1)$.
Therefore, it is sufficient to prove that 
\begin{align}\label{pi0-p1} \|\e^{-tP_h/h}|_{\mathrm{Ran}(\Pi_0-\Pro_1)}\|\leq C_N \e^{-\mathrm{Re}\, \lambda^*(1-C_N h^N)t/h}.
\end{align}
Besides, we saw in Section \ref{sectionrough} that $\mathrm{Ker}$ $P_h=\C \mathcal M_h$ where $\mathcal M_h$ was defined in \eqref{muh} and since the operator $\Pi_0$ from \eqref{Pi0} satisfies $\Pi_0^*\mathcal M_h=\mathcal M_h$, we get that $\mathcal M_h^\perp$ is invariant under $\Pi_0$ so $\mathrm{Ran}(\Pi_0-\Pro_1)=H\cap \mathcal M_h^\perp$.
Thus, with the notations from Proposition \ref{propgas} and according to \eqref{pi0-p1}, it only remains to show that
$$\|\e^{-t\mathcal M'/h}\|\leq C_N \e^{-\mathrm{Re}\, \lambda^*(1-C_N h^N)t/h}.$$
This can be done following the steps of \cite{LPMichel}, proof of Theorem 1.11 as with the notation \eqref{lamtilde} we have $\mathrm{Re}\,\lambda^*\leq \tilde\lambda_{\mathbf m^*,h}(1+C_Nh^{N})$.
The only difference is that here we have to apply the resolvent estimates given by Theorem 4 from \cite{BonyLPMichel} instead of the ones given by Theorem A.4 from \cite{LPMichel}.
For the last statement, we now asume that for $\mathbf m\in \mathcal U^{(0)}\backslash \{\mathbf m^*\}$, the expansion of $\lambda(\mathbf m,h)$ given by Theorem \ref{thmToto} differs from the one of $\lambda^*=\lambda(\mathbf m^*,h)$.
In that case, it is clear that $\lambda^*$ is a simple eigenvalue but it also happens to be a real one.
Indeed, using the fact that $X_0^h$ and $b_h$ are differential operators with real coefficients and that $M^h$ is real valued and even in the variable $\eta$, we get that $\lambda$ is an eigenvalue of $P_h$ if and only if $\overline{\lambda}$ is an eigenvalue of $P_h$.
The rest of the proof is then also similar to the end of the proof of Theorem 1.11 from \cite{LPMichel}. 
\hspace*{\fill} $\Box$
\hop
Finally, the proof of Corollary \ref{meta} is a straightforward adaptation of the one of Corollary 1.6 from \cite{BonyLPMichel}.
(Note that our notations $t_k^-$ and $t_k^+$ differ from that in \cite{BonyLPMichel}).

\addtocontents{toc}{\SkipTocEntry}
\subsection*{Acknowledgements}
The author is grateful to Laurent Michel for his advice through this work and especially for his suggestions in the proof of Lemma \ref{1impl2}, as well as Jean-François Bony for helpful discussions. % as well as Laurent Michel, especially for the proof of Lemma \ref{1impl2}.
This work is supported by the ANR project QuAMProcs 19-CE40-0010-01.

\appendix

\section{Proof of Lemma \ref{1impl2}}\label{rho}

Let us begin by showing that there exists a self-adjoint operator $A$ sucht that 
\begin{align}\label{factoa}
\varrho(H_0)=b_h^*\circ A \circ b_h.
\end{align}
Since $\varrho(0)=0$, there exists an analytic function $\tilde \varrho$ such that $\varrho(z)=z\tilde \varrho(z)$ and $|\tilde \varrho(z)|\leq C \langle z \rangle^{-1}$.
Using Cauchy's formula, one easily gets that for all $z_0\in \{\mathrm{Re}\, z > -\frac{1}{2C}\}$ and $f$ an analytic function on $\{\mathrm{Re}\, z > -\frac1C\}$ satisfying $f(z)=O(\langle z \rangle^{-\beta})$ for some $\beta>0$, we have that
\begin{align}\label{fz0}
f(z_0)=\frac{-1}{2i\pi}\int_{\{\mathrm{Re}\, z = -\frac{1}{2C}\}} f(z) (z_0-z)^{-1} \D z.
\end{align}
Working with a Hilbert basis of eigenfunctions of $H_0$, this identity yields 
\begin{align}\label{fh0}
f(H_0)=\frac{-1}{2i\pi}\int_{\{\mathrm{Re}\, z = -\frac{1}{2C}\}} f(z) (H_0-z)^{-1} \D z.
\end{align}
Besides, denoting 
$$b_h=\begin{pmatrix}
b_h^1\\
\vdots \\
b_h^d
\end{pmatrix},$$
we have $b_h H_0=(b_h^j H_0)_{1 \leq j\leq d}$ and using the identity $b_h^j H_0=b_h^* b_hb_h^j+hb _h^j $, we get
$b_h H_0=H_1 b_h$ where 
\begin{align}\label{h1}
H_1=\begin{pmatrix}
H_0+h & & \\
 & \ddots & \\
  & & H_0+h
\end{pmatrix}.
\end{align}
In particular, if $u$ is an eigenfunction of $H_0$ associated to a positive eigenvalue, the function $b_h u$ is an eigenfunction of $H_1$ associated to the same eigenvalue and therefore
\begin{align}\label{bhh1}
H_0 (H_0-z)^{-1}=b_h^* (H_1-z)^{-1} b_h.
\end{align}
It follows using \eqref{fh0} with $f=\tilde \varrho$ that \eqref{factoa} holds with $A=\tilde \varrho(H_0+h)\otimes\mathrm{Id}$:
$$\varrho(H_0)=H_0 \tilde \varrho(H_0)=b_h^* \circ \tilde \varrho(H_0+h)\otimes\mathrm{Id} \circ b_h.$$
We can improve the integrability in the integral representation of $\tilde\varrho(H_0+h)$ by writing
$$\tilde \varrho(z)=\frac{\tilde \varrho(z)}{1+z}+\frac{\varrho(z)-\varrho_\infty}{1+z}+\frac{\varrho_\infty}{1+z}$$
which yields always thanks to \eqref{fh0}
\begin{align}\label{rhotildeopti}
\qquad\tilde\varrho(H_0+h)\otimes\mathrm{Id}=\frac{-1}{2i\pi}&\int_{\{\mathrm{Re}\, z = -\frac{1}{2C}\}} \frac{\tilde \varrho(z)}{1+z} (H_1-z)^{-1} \D z\\ 
														&\qquad\qquad+\frac{-1}{2i\pi}\int_{\{\mathrm{Re}\, z = -\frac{1}{2C}\}} \frac{\varrho(z)-\varrho_\infty}{1+z} (H_1-z)^{-1} \D z+\varrho_\infty (H_1+1)^{-1}.
\end{align}
Besides, it is well known (see for instance \cite{DimassiSjostrand}) that the resolvent $(H_1-z)^{-1}$ is a pseudo-differential operator and we denote its symbol $R_z(v,\eta)$.
Thanks to \cite{DeKa}, we even have the explicit expression $R_z(v,\eta)=G_z(v^2/2+2\eta^2)\,\mathrm{Id}$ where $G_z$ is an entire function defined by
\begin{align}\label{gz}
G_z(\mu)=2h^{-1}\int_0^1(1-s)^{-\frac zh}(1+s)^{\frac zh+d-2}\e^{-\frac sh\mu} \D s=2\int_0^{h^{-1}}(1-h\sigma)^{-\frac zh}(1+h\sigma)^{\frac zh+d-2}\e^{-\sigma \mu} \D \sigma.
\end{align}
Let us then set in view of \eqref{rhotildeopti}
\begin{align}\label{Mh}
\qquad M^h(v,\eta)=\frac{-1}{2i\pi}\int_{\{\mathrm{Re}\, z = -\frac{1}{2C}\}} \frac{\tilde \varrho(z)}{1+z} R_z(v,\eta) \D z +\frac{-1}{2i\pi}\int_{\{\mathrm{Re}\, z = -\frac{1}{2C}\}} \frac{\varrho(z)-\varrho_\infty}{1+z} R_z&(v,\eta) \D z \\
														&+\varrho_\infty R_{-1}(v,\eta)
\end{align}
and we now want to show that $M^h$ is a matrix of symbols matching the properties listed in Hypothesis \ref{hypom}.
To this purpose, we need to study more carefully the function $R_z$ for $z$ fixed such that $\mathrm{Re}\,z\leq -1/2C$.
We already saw that it is analytic in both variables $v$ and $\eta$.
Now if we take $(v,\eta)\in \R^d\times \Sigma_\tau$ and put $\mu=v^2/2+2\eta^2$, we get that $\mu$ belongs to the sector
$$D_\tau=\{\mu\in\C\,;\, |\mathrm{Im}\,\mu|\leq \mathrm{Re}\,\mu+4d\tau^2\}.$$
One can then easily adapt Theorem 10 from \cite{DeKa} to show that for $n\in \N$ and $\mu\in D_\tau$, we have
\begin{align}\label{mupetit}
|\partial_\mu^n G_z(\mu)|&\leq C \int_0^{h^{-1}} \sigma^n (1-h\sigma)^{-\mathrm{Re}\,z/h}(1+h\sigma)^{\mathrm{Re}\,z/h}\e^{-\mathrm{Re}\,\mu \sigma} \D \sigma \\
		&\leq C \int_0^{+\infty} \sigma^n \e^{-(\mathrm{Re}\,\mu-2\mathrm{Re}\,z) \sigma} \D \sigma \leq C_n\langle \mu\rangle^{-(n+1)}
\end{align}
since $\mathrm{Re}\,\mu-2\mathrm{Re}\,z>0$ for $\tau$ small enough.
From \eqref{mupetit} we can already conclude that $M^h\in \mathcal M_d\big(S^0_\tau(\langle (v,\eta) \rangle^{-2})\big)$. % (note that since here $R_z$ is also analytic in $v$, using the Cauchy formula as mentionned following Definition \ref{symbolo}, one only needs to check the estimate on $\R^d\times \Sigma_\tau$ for $M^h$ and not for any of its derivatives).
%Indeed, it is clear that $M^h$ is a symbol in $\mathcal M_d\big(S^0(\langle (v,\eta) \rangle^{-2})\big)$ and analytic in the variable $\eta\in \Sigma_\tau$.
%For item \ref{majobande} from Definition \ref{symbolo}, since $M^h$ is also analytic in the variable $v$, one can use the Cauchy formula as mentionned following Definition \ref{symbolo} to get that it is sufficient to check the estimate on $\R^d\times \Sigma_\tau$ only for $M^h$ (and not for any of its derivatives).
Thus $\tilde \varrho(H_0+h)\otimes\mathrm{Id}=\mathrm{Op}_h(M^h)$ with $M^h$ sending $\R^{2d}$ in $\mathcal M_d(\R)$ as $H_0$ is self-adjoint.
Moreover, since $R_z$ is diagonal and even in the variable $\eta$, it is also the case of $M^h$. 
It only remains to prove that $M^h$ satisfies items \ref{expm} and \ref{minom} from Hypothesis \ref{hypom}.
In order to avoid some tedious computations, instead of proving the whole expansion from item \ref{expm}, we only show that $M^h$ admits a principal term $M_0$ in $\mathcal M_d\big(S^0_\tau(\langle (v,\eta) \rangle^{-2})\big)$ from which we will deduce that item \ref{minom} is satisfied.
One easily gets for $\mathrm{Re} \, z\leq -1/2C$ and $\mu\in D_\tau$ fixed by dominated convergence that 
\begin{align}\label{G0}
\lim_{h\to 0}G_z(\mu)=2\int_0^\infty \e^{\sigma (2z-\mu)}\D \sigma=\frac{1}{\mu/2-z}=:G_z^0(\mu).
\end{align}
We would like to get some estimates of the derivatives $\partial_\mu^n (G_z-G_z^0)$ in $O(h\langle \mu \rangle^{-n-1})$ on $D_\tau$ uniformly in $z\in \{\mathrm{Re}\, z\leq -1/2C\}$ in order to apply the formula \eqref{Mh} to those. 
We have
%\color{red} manque la partie de l'inté jusqua l'infini\color{black}
\begin{align}\label{g-g0}
\partial_\mu^n (G_z-G_z^0)(\mu)&=2\int_0^{h^{-1}} \bigg[ \exp\bigg(z\,\Big[\frac1h \ln\Big(\frac{1+h\sigma}{1-h\sigma}\Big)-2\sigma\Big]+(d-2)\ln(1+h\sigma)\bigg)-1\bigg](-\sigma)^n\e^{\sigma(2z-\mu)}\D\sigma\nonumber\\
			&\qquad \qquad- 2\int_{h^{-1}}^{\infty} (-\sigma)^n e^{\sigma(2z-\mu)}\D\sigma\nonumber\\
			&=2\int_0^{h^{-1}/2} \bigg[ \exp\bigg(z\,\Big[\frac1h \ln\Big(\frac{1+h\sigma}{1-h\sigma}\Big)-2\sigma\Big]+(d-2)\ln(1+h\sigma)\bigg)-1\bigg](-\sigma)^n\e^{\sigma(2z-\mu)}\D\sigma\\
			&\qquad \qquad +O\Big(\e^{\frac{\mathrm{Re}\,(2z-\mu)}{Ch}}\Big)\nonumber.
\end{align}
Let us denote 
$$g_{z,h}(\sigma)=\bigg[ \exp\bigg(z\,\Big[\frac1h \ln\Big(\frac{1+h\sigma}{1-h\sigma}\Big)-2\sigma\Big]+(d-2)\ln(1+h\sigma)\bigg)-1\bigg](-\sigma)^n$$
and observe that for all $0\leq k \leq n$, one has 
\begin{align}\label{gbord}
\partial_\sigma^k g_{z,h}(0)=0 \qquad \text{and }\qquad \partial_\sigma^k g_{z,h}(h^{-1}/2)=O(h^{-n}\langle z \rangle^k).
\end{align}
Besides, on $\sigma\in [0,h^{-1}/2]$, it holds 
\begin{align}\label{gn+1}
\partial_\sigma^{n+1} g_{z,h}(\sigma)=\sum_{j=1}^{n+1}O\big(h\langle z \rangle^{j} \langle \sigma \rangle^{j}\sigma^{j-1}\big).
\end{align}
Now, let us do $n+1$ integrations by parts in the first term from \eqref{g-g0}.
By \eqref{gbord}, each boundary term is $O(h^{-n}\langle z \rangle^{k} \langle 2z-\mu \rangle^{-(k+1)}\e^{\mathrm{Re}\,(2z-\mu)/Ch})$ while the remaining integral term satisfies
\begin{align}
\bigg|\frac{2}{(\mu-2z)^{n+1}}\int_0^{h^{-1}/2}\partial_\sigma^{n+1} g_{z,h}(\sigma) \e^{\sigma(2z-\mu)}\D\sigma\bigg|&\leq C_nh\sum_{j=1}^{n+1}\frac{\langle z\rangle^j}{|2z-\mu|^{n+1}}\int_0^\infty \sigma^{j-1}\langle \sigma \rangle^{j} \e^{\sigma\, \mathrm{Re}(2z-\mu)}\D\sigma\\
			&\leq C_n h \langle \mu \rangle^{-(n+1)}
\end{align}
thanks to \eqref{gn+1}.
Thus, we have shown that for $n\in \N$, $\mu\in D_\tau$ and $\mathrm{Re}\,z\leq -1/2C$,
$$|\partial_\mu^n (G_z-G_z^0)(\mu)|\leq C_n h \langle \mu \rangle^{-(n+1)}.$$
Putting $R_z^0(v,\eta)=G_z^0(v^2/2+2\eta^2)\,\mathrm{Id}$ and defining $M_0(v,\eta)$ as in \eqref{Mh} with $R_z$ replaced by $R_z^0$, we deduce that
$$|\partial^\alpha (M^h-M_0)(v,\eta)|\leq C_\alpha h\langle (v,\eta) \rangle^{-2}\qquad \text{on }\R^d\times \Sigma_\tau$$
so item \ref{expm} from Hypothesis \ref{hypom} holds true.
Finally, by definition of $M_0$ and thanks to \eqref{G0} and \eqref{fz0}, we have
\begin{align}\label{m0}
M_0(v,\eta)=\tilde \varrho\big(v^2/4+\eta^2\big)\,\mathrm{Id} \geq \frac1C \langle (v,\eta) \rangle^{-2}\,\mathrm{Id}
\end{align}
by assumption on $\varrho$.
Therefore item \ref{minom} from Hypothesis \ref{hypom} holds true and the proof is complete.

\section{Linear algebra Lemma}

We use the following lemma which is inspired by \cite{BonyLPMichel}, Lemma 2.6.
\begin{lem}\label{pasir}
Let $M\in \mathcal M_{d'}(\C)$ such that $M=S(A+T)$ with $S$ hermitian and invertible, $A$ skew-hermitian and $T$ hermitian positive semidefinite.
Suppose moreover that
$$M(\mathrm{Ker}\,T) \cap \mathrm{Ker}\,T=\mathrm{Ker}\,M\cap \mathrm{Ker}\,T=\{0\}.$$
Then $M$ has no spectrum in $i\R$.
\end{lem}

\begin{proof}
Let $\lambda\in \R$ and $X\in \mathrm{Ker}\, [M-i\lambda]$, we first show that $X \in$ Ker $T$.
Since $T$ is hermitian positive semidefinite, it is sufficient to show that $\langle TX,X\rangle=0$.
Using the properties of $S$, $A$ and $T$ we have 
\begin{align}
\langle TX,X\rangle&=\mathrm{Re}\,\big\langle (A+T)X,X\big\rangle\\
	&=\mathrm{Re}\,\big\langle S^{-1}S(A+T)X,X\big\rangle \\
	&=\mathrm{Re}\,\big(i\lambda \big\langle S^{-1}X,X\big\rangle\big)\\
	&=0
\end{align}
so $X \in$ Ker $T$.
Thanks to the assumption, it only remains to prove that $X \in $ Ker $M$.
This can be done easily by noticing that
$$MX=i\lambda X \in M(\mathrm{Ker}\,T) \cap \mathrm{Ker}\,T$$
so $MX=0$ by assumption.
\end{proof}

\section{Asymptotic expansions}

Let $d'\in \N^*$.
Here we use the convention $\sum_{j=0}^{-1}a_j=0$ for any sequence $(a_j)_{j\geq 0}$ in  a vector space.
For $K\subseteq \R^{d'}$, the notation $a=O_{\mathcal C^{\infty}(K)}(h^N)$ (respectively $a=O_{L^{\infty}(K)}(h^N)$) means that for all $\beta \in \N^{d'}$, there exists $C_{\beta, N}$ such that $\|\partial^\beta a\|_{\infty,K}\leq C_{\beta, N} h^N$ (resp. there exists $C_{ N}$ such that $\| a\|_{\infty,K}\leq C_{ N} h^N$).
We will also use the notations from Definition $\ref{symbolo}$ and \eqref{exph}.

\begin{prop}\label{expcompo}
Let $m\in \N^*$ ; $d_1,\dots, d_m \in \N^*$ and for $1 \leq j \leq m$,  $K_j\subset \R^{d_j}$ some compact sets.
Let a smooth function
$$\phi_h:\prod_{j=1}^m K_j\to K\subset \Sigma_\tau$$
such that $\phi_h=O_{\mathcal C^{\infty}(\prod_{j=1}^m K_j)}(1)$.
Consider $g^h \sim_h \sum_{n \geq 0}h^ng_n$ in $S^0_\tau(1)$ or in $\mathcal C^{\infty}(K)$ if $\phi_h$ actually takes values in $\R^d$.
Then 
$$g^h\circ \phi_h \sim_h \sum_{n \geq 0}h^n(g_n\circ \phi_h)$$
in $\mathcal C^{\infty}(\prod_{j=1}^m K_j)$.
\end{prop}

\begin{proof} 
Let $N \in \N$ and denote $r_N=g^h-\sum_{n=0}^{N-1}h^n g_n=O_{S^0_\tau(1)}(h^N)$.
\begin{align*}
g^h\circ \phi_h &=\Big(\sum_{n=0}^{N-1}h^n g_n+r_N\Big)\circ \phi_h\\
		&=\sum_{n=0}^{N-1}h^n (g_n\circ \phi_h) +r_N\circ \phi_h.
\end{align*}
But since all the derivatives of $\phi_h$ are bounded uniformly in $h$, and the ones of $r_N$ are $O_{L^{\infty}(\Sigma_\tau)}(h^N)$, we see that $r_N\circ \phi_h$ is $O_{\mathcal C^{\infty}(\prod_{j=1}^m K_j)}(h^N)$ so we have the announced result.
\end{proof}

%\begin{prop}\label{expserie}
%Let $K\subset \R^d$ a compact set.
%For $n\in \N$, let $a_n\sim_h \sum_{j\geq 0}h^ja_{n,j}$ in $\mathcal C^{\infty}(K)$ such that the serie $\sum_n a_n$ %and $\sum_n a_{n,j}$ 
%converges in $\mathcal C^{\infty}(K)$ uniformly in $h$.% for any $j\geq 0$.
%Then for all $j\geq 0$, $\sum_n a_{n,j}$ converges in $\mathcal C^{\infty}(K)$ and 
%$$\sum_n a_n\sim_h \sum_{j\geq 0}h^j \sum_n a_{n,j} \quad \text{in } \mathcal C^{\infty}(K).$$
%\end{prop}
%
%\begin{proof}
%We start by noticing that 
%$$a_{n,j}=\lim_{h\to 0}h^{-j}\Big(a_n-\sum_{k=0}^{j-1}h^ka_{n,k}\Big)$$
%and since the convergence of $\sum_n a_n$ is uniform in $h$, we easily get that $\sum_n a_{n,j}$ is a Cauchy sequence so it converges in $\mathcal C^{\infty}(K)$ and 
%$$\sum_n a_{n,j}=\lim_{h\to 0}\sum_n h^{-j}\Big(a_n-\sum_{k=0}^{j-1}h^ka_{n,k}\Big).$$
%Therefore, for $N \in \N$, denoting $b_{N,n}=a_n-\sum_{j= 0}^{N-1}h^j  a_{n,j}=O_{\mathcal C^{\infty}(K)}(h^N)$, we get that the serie $\sum_n b_{N,n}$ converges uniformly in $h$, so its sum is $O_{\mathcal C^{\infty}(K)}(h^N)$.
%Hence, we can now write 
%\begin{align*}
%\sum_n a_n&=\sum_n \Big( \sum_{j=0}^{N-1}h^j a_{n,j}+b_{N,n}\Big)\\
%		&=\sum_{j= 0}^{N-1}h^j \sum_n a_{n,j}+O_{\mathcal C^{\infty}(K)}(h^N)
%\end{align*}
%\end{proof}

\begin{prop}\label{expsharp}
Since the matrix $M^h$ from Hypothesis \ref{hypom} satisfies $M^h \sim \sum_{n \geq 0}h^nM_n$ in $\mathcal M_{d}\big(S^0_\tau(\langle(v,\eta) \rangle^{-2})\big)$, 
the vector of symbols $g^h$ defined in Remark \ref{defg} also admits a classical expansion $g^h \sim \sum_{n \geq 0}h^ng_n$ in $\mathcal M_{1,d}\big(S^0_\tau(\langle(v,\eta) \rangle^{-1})\big)$, where the $(g_n)$ are given by 
\begin{align*}%\label{g0app}
g_0(x,v,\eta)=\Big(-i\, {}^t\eta+ \, \frac {{}^tv}{2}\Big)M_0(x,v, \eta)
\end{align*}
 and 
$$g_n(x,v,\eta)=\Big(-i\, {}^t\eta+ \, \frac {{}^tv}{2}\Big)M_n(x,v,\eta)-\frac 12( {}^t\nabla_v-\frac i2  {}^t\nabla_\eta)M_{n-1}(x,v,\eta)$$
for $n\geq 1.$
\end{prop}

\begin{proof} 
We have $$g^h=(-i\, {}^t\eta+\, {}^tv/2)M^h-\frac{h}{2}( {}^t\nabla_v-\frac i2  {}^t\nabla_\eta)M^h$$ and the last term clearly admits the expansion $$-\sum_{n\geq 1}h^n\frac 12( {}^t\nabla_v-\frac i2  {}^t\nabla_\eta)M_{n-1}$$ in $S^0_\tau(\langle(v,\eta) \rangle^{-2})$.
For the first term of $g^h$, it suffices to notice that for any $N\in \N$, $$\Big(-i\, {}^t\eta+ \, \frac {{}^tv}{2}\Big)\;O_{\mathcal M_{d}\big(S^0_\tau(\langle(v,\eta) \rangle^{-2})\big)}(h^N)=O_{\mathcal M_{1,d}\big(S^0_\tau(\langle(v,\eta) \rangle^{-1})\big)}(h^N).$$
\end{proof}

\begin{prop}\label{expinexp}
Let $K$ a compact set in $\R^{d'}$ and $a\sim_h \sum_{n\geq 0}h^n a_n$ in $\mathcal C^{\infty}(K)$ such that for all $n \geq 0$, $a_n\sim_h \sum_{j\geq 0}h^j a_{n,j}$ in $\mathcal C^{\infty}(K)$.
Then  
$$a \sim_h \sum_{n \geq 0}h^n \sum_{j=0}^n a_{j,n-j} \quad \text{in } \mathcal C^{\infty}(K).$$
\end{prop}

\begin{proof} 
It suffices to write for $N \in \N$
\begin{align*}
a&=\sum_{n=0}^{N-1} h^n \Big( \sum_{j=0}^{N-1-n} h^j a_{n,j}+O_{\mathcal C^{\infty}(K)}(h^{N-n}) \Big) +O_{\mathcal C^{\infty}(K)}(h^N)\\
		&=\sum_{n = 0}^{N-1}h^n \sum_{j=0}^n a_{j,n-j} +O_{\mathcal C^{\infty}(K)}(h^N).
\end{align*}
\end{proof}

\begin{prop}\label{linfcinf}
Let $K$ a compact set in $\R^{d'}$ and $a \in \mathcal C^\infty(K)$ such that for all $\beta \in \N^{d'}$, there exists $a_{\beta,j}\in \mathcal C^\infty(K)$ such that $\partial^\beta a \sim \sum_{j\geq 0}h^ja_{\beta,j}$ in $L^\infty(K)$.
Then $a_{\beta,j}=\partial^\beta a_{0,j}$, i.e
$$a \sim \sum_{j\geq 0}h^ja_{0,j} \quad \text{in } \mathcal C^\infty(K).$$
\end{prop}

\begin{proof}
For simplicity, we take $d'=1$.
Let us denote $a_j=a_{0,j}$.
By induction, it is sufficient to prove the result for $\beta=1$, i.e prove that $a_{1,j}=a_j'$.
Here again, it suffices to prove the case $j=0$ which we can then apply to the function $h^{-1}(a-a_0)$ and so on.
Let $x$ in the interior of $K$ and $t\in \R^*$ in a neighborhood of 0.
We look at the differential fraction
\begin{align*}
\frac{a_0(x+t)-a_0(x)}{t}&=\frac{a(x+t)-a(x)}{t}+\frac{O(h)}{t}\\
				&=a'(x)+t\int_0^1 (1-s)a''(x+st) \D s+\frac{O(h)}{t}\\
				&=a_{1,0}(x)+O(h)+t\int_0^1 (1-s)a''(x+st) \D s+\frac{O(h)}{t}\\
				&\xrightarrow[h \to 0]{} a_{1,0}(x)+t\int_0^1 (1-s)a_{2,0}(x+st) \D s.
\end{align*}
Taking now the limit $t \to 0$, we get $a_0'(x)=a_{1,0}(x)$ which was the desired result.
\end{proof}

\begin{prop}\label{expdl}
Recall the notation \eqref{d0tau} and let $K\subset \R^{d'}$ a compact set, $\Psi:K\to D(0, \tau)^d$ a smooth function such that $\Psi \sim \sum_{j\geq 0}h^j \Psi_j$ in $\mathcal C^\infty(K)$ and $b$ an analytic function on $\Sigma_\tau$.
Then 
\begin{align}\label{expaopsi}
b \circ \Psi \sim \sum_{j\geq 0}h^j b_j 
\end{align}
in $\mathcal C^\infty(K)$, with
$$b_0=b \circ \Psi_0 \quad\; \text{and for }j \geq 1, \;\quad b_j=\sum_{|\beta|=1}^j \frac{\partial^\beta b \circ \Psi_0}{\beta !}\sum_{s\in S_{\beta,j}} \prod_{k \in K_\beta} \bigg( \sum_{a \in A_{\beta, s, k}} \prod_{l=1}^{\beta_k} \big(\Psi_{a_l}\big)_k \bigg),$$
where $K_\beta=$ supp $\beta=\{k \in \llbracket 1,d \rrbracket \, ;\,\beta_k\neq 0\}$, 
$S_{\beta,j}=\{s \in \N^d \, ;\, \mathrm{supp}\, s=K_\beta, \, |s|=j \text{ and }s\geq \beta\}$ and 
$A_{\beta, s, k}=\{a \in (\N^*)^{\beta_k} \, ;\, |a|=s_k\}$.
\end{prop}

\begin{proof} 
We first prove that \eqref{expaopsi} holds in $L^\infty(K)$.
Doing a Taylor expansion of $b$, we have for $N \in \N^*$ that 
\begin{align}\label{taylor}
b\circ \Psi &= b \circ \Psi_0 + \sum_{|\beta|=1}^{N-1} \frac{\partial^\beta b \circ \Psi_0  }{\beta !}(\Psi- \Psi_0)^\beta+ O\Big((\Psi- \Psi_0)^N\Big) \nonumber\\
		&= b \circ \Psi_0 + \sum_{|\beta|=1}^{N-1} \frac{\partial^\beta b \circ \Psi_0  }{\beta !}(\Psi- \Psi_0)^\beta+O_{L^\infty(K)}(h^N)
\end{align}
since $\Psi-\Psi_0=O_{\mathcal C^\infty(K)}(h)$. %\color{red} peut être directement $O_{\mathcal C^\infty(K)}(1)$ ce qui éviterait la fin, mais pas clair\color{black}.
Now one can see that 
$$(\Psi-\Psi_0)^\beta \sim \sum_{j\geq |\beta|}h^j\sum_{s\in S_{\beta,j}} \prod_{k \in K_\beta} \bigg( \sum_{a \in A_{\beta, s, k}} \prod_{l=1}^{\beta_k} \big(\Psi_{a_l}\big)_k \bigg)$$
so \eqref{taylor} gives
\begin{align*}
b\circ \Psi &= b \circ \Psi_0+\sum_{|\beta|=1}^{N-1} \frac{\partial^\beta b \circ \Psi_0  }{\beta !} \Bigg[\sum_{j= |\beta|}^{N-1} h^j\sum_{s\in S_{\beta,j}} \prod_{k \in K_\beta} \Big( \sum_{a \in A_{\beta, s, k}} \prod_{l=1}^{\beta_k} \big(\Psi_{a_l}\big)_k \Big)+O_{\mathcal C^\infty(K)}(h^N) \Bigg]+O_{L^\infty(K)}(h^N) \\
		&=b \circ \Psi_0+\sum_{j=1}^{N-1} h^j \sum_{|\beta|=1}^{j} \frac{\partial^\beta b \circ \Psi_0  }{\beta !} \sum_{s\in S_{\beta,j}} \prod_{k \in K_\beta} \Big( \sum_{a \in A_{\beta, s, k}} \prod_{l=1}^{\beta_k} \big(\Psi_{a_l}\big)_k \Big) +O_{L^\infty(K)}(h^N)
\end{align*}
which proves that \eqref{expaopsi} holds in $L^\infty(K)$.\\
Besides, the derivatives of $b \circ \Psi$ are linear combinations of products of some derivatives of $\Psi$ with some $\partial^\gamma b \,\circ \Psi$ where $\gamma$ is a integer multi-index.
Hence the expansion of $\Psi$ in $\mathcal C^\infty(K)$ and the result that we just proved applied to $\partial^\gamma b \circ \Psi$ instead of $b\circ \Psi$ yield that for all $\beta \in \N^{d'}$, $\partial^\beta (b\circ \Psi)$ admits a classical expansion in $L^\infty(K)$ whose coefficients are smooth.
Therefore, Proposition $\ref{linfcinf}$ enables us to conclude that \eqref{expaopsi} holds in $\mathcal C^\infty(K)$.
\end{proof}

\begin{cor}\label{bousin}
Using the notations from the proof of Lemma \ref{expom}, we have
$$
g_n\Big(x,\frac{v+v'}{2},\eta +i\psi(x,v,v')\Big)\sim \sum_{j\geq 0} h^j g_{n,j}(x,v,v',\eta) \quad \text{on } B_0(\mathbf s,2r)\times B_\infty (0,2r)
$$
with 
$$
g_{n,0}(x,v,v',\eta)=g_n\Big(x,\frac{v+v'}{2},\eta +i\psi_0(x,v,v')\Big)
$$
and for $j \geq 1$
\begin{align*}
g_{n,j}(x,v,v',\eta)=iD_\eta g_n\Big(x,\frac{v+v'}{2},\eta+i\psi_0(x,v,v')\Big) \big(\psi_j(x,v,v')\big)+R^1_j(\ell_0, \dots, \ell_{j-1})
\end{align*}
where $R^1_j : \big( \mathcal C^\infty(B_0(\mathbf s,2r) )\big)^j \to \mathcal C^\infty(B_0(\mathbf s,2r) ).$
\end{cor}

\begin{proof}
Since $\psi(\mathbf s,0, 0)=O(h)$, we can suppose that $r$ was chosen small enough so that $(x,v,v',\eta)\mapsto\eta +i\psi(x,v,v')$ sends $B_0(\mathbf s,2r) \times B_\infty(0,2r)$ in $D(0, \tau)^d$
Hence we can use Proposition \ref{expdl} %as well as Proposition $\ref{expcompo}$
to get that 
$$
g_n\Big(x,\frac{v+v'}{2},\eta +i\psi(x,v,v')\Big)\sim \sum_{j\geq 0} h^j g_{n,j}(x,v,v',\eta) \quad \text{on } B_0(\mathbf s,2r)\times B_\infty (0,2r)
$$
with 
$$
g_{n,0}(x,v,v',\eta)=g_n\Big(x,\frac{v+v'}{2},\eta +i\psi_0(x,v,v')\Big)
$$
and for $j \geq 1$
\begin{align}\label{gnjapp}
\qquad g_{n,j}(x,v,v',\eta)=\sum_{|\beta|=1}^j \frac{i^{|\beta|}}{\beta !}\partial^\beta_\eta g_n \Big(&x,\frac{v+v'}{2},\eta+i\psi_0(x,v,v')\Big)\sum_{s\in S_{\beta,j}} \prod_{k \in K_\beta} \bigg( \sum_{a \in A_{\beta, s, k}} \prod_{l=1}^{\beta_k} \big(\psi_{a_l}\big)_k \bigg) 
\end{align}
where $K_\beta=$ supp $\beta=\{k \in \llbracket 1,d \rrbracket \, ;\,\beta_k\neq 0\}$, 
$S_{\beta,j}=\{s \in \N^d \, ;\, \mathrm{supp}\, s=K_\beta, \, |s|=j \text{ and }s\geq \beta\}$ and 
$A_{\beta, s, k}=\{a \in (\N^*)^{\beta_k} \, ;\, |a|=s_k\}$.
Now, we see thanks to \eqref{gnjapp} that the terms of $g_{n,j}(x,v,v',\eta)$ for which $|\beta|=1$ yield 
\begin{align}\label{ljdsij}
iD_\eta g_n\Big(x,\frac{v+v'}{2},\eta+i\psi_0(x,v,v')\Big) \big(\psi_j(x,v,v')\big)
\end{align}
while the terms for which $|\beta|>1$ only feature the functions $\ell_0, \dots , \ell_{j-1}$.
\end{proof}
\hip  
Finally, we state the version of Laplace's method for integral approximation that we use in this paper.

\begin{prop}\label{expscal}
Let $x_0 \in  \R^{d'}$, $K$ a compact neighborhood of $x_0$ and $\varphi\in \mathcal C^\infty(K)$ such that $x_0$ is a non degenerate minimum of $\varphi$ and its only global minimum on $K$.
Let also $a_h \sim \sum_{j\geq 0}h^j a_j $ in $\mathcal C^\infty (K)$ and denote $H\in \mathcal M_{d'}(\R)$ the Hessian of $\varphi$ at $x_0$.
The integral
$$\frac{\det (H)^{1/2}}{(2\pi h)^{d'/2}} \int_{K}a_h(x)\e^{-\frac{\varphi(x)-\varphi(x_0)}{h}}\D x$$
admits a classical expansion whose first term is given by $a_0(x_0)$.
\end{prop}

\section{Proof of Lemma \ref{rreelle}}\label{rreelleapp}
According to the proof of Corollary \ref{bousin} and the end of the proof of Lemma \ref{expom} from which we keep the notations, we have the following expression for $R_j$:
\begin{align}\label{Rj}
\qquad \;\, R_j(\ell_0, \dots, \ell_{j-1})(x,v)=&\mathop{\sum_{n_1+n_2+n_3+n_4=j}}_{\substack{n_3, n_4  \neq j}}\frac{1}{i^{n_1}n_1!}\big(\partial_{v'} \cdot \partial_\eta \big)^{n_1} \Big( g_{n_2,n_3}(x,v,v',\eta) \partial_v \ell_{n_4}(x,v') \Big) \Bigg| \mathop{}_{\substack{v'=v \\  \eta =0}} \\
					&+\sum_{|\beta|=2}^j \frac{i^{|\beta|}}{\beta !}\partial^\beta_\eta  g_0 \Big(x,\frac{v+v'}{2},i\big(\frac v2 +\ell_0 (x,v)\, \partial_v \ell_0(x,v)\big)\Big)\nonumber \\
					&\qquad \qquad\qquad \qquad \times \sum_{s\in S_{\beta,j}} \prod_{k \in K_\beta} \bigg( \sum_{a \in A_{\beta, s, k}} \prod_{l=1}^{\beta_k} \big(\psi_{a_l}(x,v,v)\big)_k \bigg) \partial_v \ell_0(x,v) \nonumber\\
					&+iD_\eta g_0\big(x,v,i(v/2+\ell_0 (x,v)\, \partial_v \ell_0(x,v))\big)\sum_{k=1}^{j-1}\big(\ell_k\partial_v \ell_{j-k}\big)(x,v)\; \partial_v \ell_0(x,v).\nonumber
\end{align}
Using Lemma $\ref{ureelle}$ and \eqref{gnjapp}, it is clear that the last two terms of $R_j(\ell_0, \cdots, \ell_{j-1})$ given by \eqref{Rj} and the terms of the first sum for which $n_1=0$ are real valued.
For the rest of the first term, we start by noticing that one can establish by induction that for $n_1 \geq 1$,
\begin{align}\label{n1}
\big(\partial_{v'}\cdot \partial_\eta\big)^{n_1}=\sum_{p\in \llbracket 1,d\rrbracket^{n_1}}\partial_{v'}^{\gamma(p)}\partial_{\eta}^{\gamma(p)}
\end{align}
where using the notation \eqref{ej}, we define $\gamma(p)=\sum_{k=1}^{n_1}e_{p_k}$ (note that $|\gamma(p)|=n_1$).
Besides, we have for $0 \leq n_2 \leq j$ and $p\in \llbracket 1,d\rrbracket^{n_1}$
\begin{align}\label{n30}
\partial_{\eta}^{\gamma(p)}g_{n_2,0}(x,v,v',0)=\partial_{\eta}^{\gamma(p)}g_{n_2}\Big(x,\frac{v+v'}{2},i\psi_0(x,v,v')\Big)\in i^{n_1}\R^d
\end{align}
according to Lemma $\ref{ureelle}$ and in the case $j\geq 2$, for $1 \leq n_3 \leq j-1$
\begin{align}\label{n3}
\qquad \;\,\partial_{\eta}^{\gamma(p)}&g_{n_2,n_3}(x,v,v',0) \\
		&=\sum_{|\beta|=1}^{n_3} \frac{i^{|\beta|}}{\beta !}\partial^{\beta+\gamma(p)}_\eta g_{n_2} \Big(x,\frac{v+v'}{2},i\psi_0(x,v,v')\Big) \sum_{s\in S_{\beta,n_3}} \prod_{k \in K_\beta} \bigg( \sum_{a \in A_{\beta, s, k}} \prod_{l=1}^{\beta_k} \big(\psi_{a_l}\big)_k \bigg)\in i^{n_1}\R^d \nonumber
\end{align}
where we used \eqref{gnjapp} and Lemma $\ref{ureelle}$ once again.
The combination of \eqref{n1}, \eqref{n30} and \eqref{n3} enables us to conclude that the term
$$\mathop{\sum_{n_1+n_2+n_3+n_4=j}}_{\substack{n_1\neq 0;\, n_3, n_4  \neq j}}\frac{1}{i^{n_1}n_1!}\big(\partial_{v'} \cdot \partial_\eta \big)^{n_1} \Big( g_{n_2,n_3}(x,v,v',\eta) \partial_v \ell_{n_4}(x,v') \Big) \Bigg| \mathop{}_{\substack{v'=v \\  \eta =0}}$$
from \eqref{Rj} 
is also real so $R_j(\ell_0^{\mathbf s}, \dots, \ell_{j-1}^{\mathbf s})$ is real valued.
For the last statement, it suffices to use the formula \eqref{Rj} after noticing that $\psi$ (and hence the $(g_{n_2,n_3})$) remain unchanged when $\ell$ is replaced by $-\ell$.

\nocite{}
\bibliography{bibBoltz} 
%\addcontentsline{toc}{section}{Bibliography}
\bigskip
\scshape \small Thomas Normand, Institut de Math\'ematiques de Bordeaux, Universit\'e de Bordeaux
   %\centerline{351, Cours de la Lib\'eration - F 33 405 TALENCE, France}
\normalfont

\end{document}